\documentclass[12pt,reqno,letterpaper]{amsart}
\usepackage{latexsym, amssymb, amscd, amsfonts, xypic,euler}
\usepackage{setspace}
\usepackage{color}
\usepackage{bbm}
\usepackage{enumerate}
\usepackage{appendix}

\usepackage[OT2,T1]{fontenc}
\DeclareSymbolFont{cyrletters}{OT2}{wncyr}{m}{n}
\DeclareMathSymbol{\Sha}{\mathalpha}{cyrletters}{"58}

\usepackage{hyperref}  		

\newlength{\baseunit}

\newcount{\numlines}                
\newcommand{\comments}[1]{}

\def \cI {\mathcal I}
\def \cM {\mathcal M}
\def \cA {\mathcal A}
\def \cP {\mathcal P}
\def \cS {\mathcal S}

\def \vP {\vec{P}}
\def \vR {\vec{R}}

\def \CZ {CZ^\circ}

\def \Wh {\mbox{Wh}}
\def \dist {\mbox{dist}}

\def \oA {\overline{\mathcal{A}}}

\def \oE {{\overline{E}}}

\def \oa {{\overline{\alpha}}}
\def \ob {{\overline{\beta}}}
\def \onu {{\overline{\nu}}}

\def \sp {\mbox{dp}}

\def \cl {\mbox{cl}}

\def \rst {{\mbox{\scriptsize rst}}}
\def \th {{\mbox{\scriptsize th}}}
\def \new {{\mbox{\scriptsize new}}}
\def \spec {{\mbox{\scriptsize special}}}
\def \ky {{\mbox{\scriptsize key}}}

\def \hx {{\widehat{x}}}
\def \hz {{\widehat{z}}}
\def \hQ {{\widehat{Q}}}
\def \hy {{\widehat{y}}}
\def \hE {{\widehat{E}}}

\def \hP {{\widehat{P}}}

\newcommand{\R}{\mathbb{R}}
\newcommand{\Z}{\mathbb{Z}}
\newcommand{\N}{\mathbb{N}}
\newcommand{\X}{{\mathbb{X}}}

\newcommand{\LR}{L^{m,p}(\R^n)}
\newcommand{\LE}{L^{m,p}(E)}
\newcommand{\WR}{W^{m,p}(\R^n)}
\newcommand{\WE}{W^{m,p}(E)}
\newcommand{\diam}{\mbox{diam}}

\newcommand{\supp}{\mbox{supp}}

\newcommand{\otheta}{\overline{\theta}}

\newcommand{\key}{\mathcal{K}}

\DeclareMathOperator*{\Blim}{Blim}

\newcommand{\oeq}[1]{\overset{ {\scriptscriptstyle #1}}{=}}
\newcommand{\oleq}[1]{\overset{ {\scriptscriptstyle #1}}{\leq}}
\newcommand{\ogeq}[1]{\overset{ {\scriptscriptstyle #1}}{\geq}}
\newcommand{\olesssim}[1]{\overset{ {\scriptscriptstyle #1}}{\lesssim}}

\newcommand{\osimeq}[1]{\overset{ {\scriptscriptstyle #1}}{\simeq}}

\newtheorem*{extthm}{Extension Theorem for $\mathbf{(E,z)}$}
\newtheorem*{ml}{Main Lemma for $\mathbf{\cA}$}
\newtheorem{rek}{Remark}
\newtheorem{defn}{Definition}
\newtheorem{thm}{Theorem}
\newtheorem{lem}{Lemma}
\newtheorem{cor}{Corollary}
\newtheorem{prop}{Proposition}

\numberwithin{lem}{section}

\title{Sobolev Extension by Linear Operators}
\author{Charles L. Fefferman}
\thanks{The first author is partially supported by NSF and ONR grants DMS 09-01040 and N00014-08-1-0678}
\author{Arie Israel}
\thanks{The second author is partially supported by an NSF postdoctoral fellowship, DMS-1103978}
\author{Garving K. Luli}
\thanks{The third author is partially supported by NSF and ONR grants DMS 09-01040 and N00014-08-1-0678}

\begin{document}

\begin{abstract}
Let $L^{m,p}(\R^n)$ be the Sobolev space of functions with $m^\th$ derivatives lying in $L^p(\R^n)$. Assume that $n< p < \infty$. For $E \subset \R^n$, let $L^{m,p}(E)$ denote the space of restrictions to $E$ of functions in $L^{m,p}(\R^n)$. We show that there exists a bounded linear map $T : \LE \rightarrow \LR$ such that, for any $f \in \LE$, we have $Tf = f$ on $E$. We also give a formula for the order of magnitude of  $\|f\|_{L^{m,p}(E)}$ for a given $f : E \rightarrow \R$ when $E$ is finite.
\end{abstract}
\maketitle

\tableofcontents

\section{Introduction} \label{sec_intro}
Fix positive integers $m,n$ and let $f$ be a real-valued function defined on an (arbitrary) given subset $E \subset \R^n$. How can we tell whether $f$ extends to a $C^m$ function $F$ on the whole $\R^n$? If such an $F$ exists, then how small can we take its $C^m$ norm? What can we say about the derivatives $\partial^\alpha F(x)$ at a given point $x$? Can we take $F$ to depend linearly on $f$?

Suppose $E$ is finite. Can we compute an extension $F$ whose $C^m$ norm has the least possible order of magnitude? How many computer operations does it take?

The above questions were answered in \cite{F1,F2,F3,F4,F5, FK1,FK2}, building on earlier work of H. Whitney, G. Glaeser, Y. Brudnyi, P. Shvartsman, E. Bierstone, P. Milman and W. Paw\l ucki \cite{BMP,BS1,BS2,G,W1,W2,W3}.

Now we want to address the analogous questions for Sobolev spaces in place of $C^m$. Important first steps were taken by P. Shvartsman \cite{S1,S2,S3} and A. Israel \cite{I}; we discuss these papers later in the introduction.

To state our results, we set up notation. We work in the Sobolev space $\X = \LR$ ($n<p<\infty$) with seminorm
$$\|F\|_{\X} = \max_{|\alpha| = m} \left( \int_{\R^n} |\partial^\alpha F(x)|^p dx \right)^{1/p},$$
or else in $\X = C^m(\R^n)$ with norm
$$\|F\|_{C^m(\R^n)} = \max_{|\alpha| \leq m} \sup_{x \in \R^n} |\partial^\alpha F(x)|.$$

Given $E \subset \R^n$, we write $\X(E)$ for the space of restrictions to $E$ of functions in $\X$. The space $\X(E)$ has a natural seminorm
$$\|f\|_{\X(E)} = \inf \{ \|F\|_{\X} : F \in \X, \; F=f \; \mbox{on}\; E\}.$$

When $\X = \LR$, our standing assumption $n < p < \infty$ guarantees by the Sobolev theorem that any $F \in \X$ has continuous derivatives up to order $(m-1)$. Consequently, $F$ may be restricted to $E$ and our definition of $\X(E)$ makes sense.

When $\X = \LR$, then $c,C,C'$ etc. denote constants depending only on $m,n,p$. Similarly, when $\X = C^m(\R^n)$, then $c,C,C'$ etc. denote constants depending only on $m,n$. These symbols may denote different constants in different occurrences.

The simplest form of our results is as follows. Let $\X = \LR$ until further notice.

\begin{thm}[Existence of Linear Extension Operators]\label{thm1} Given $E \subset \R^n$, there exists a linear map $T \colon \X(E) \rightarrow \X$ such that
$$Tf = f \; \mbox{on} \; E \;\;\; \mbox{and} \;\;\; \|Tf\|_{\X} \leq C \|f\|_{\X(E)} \;\;\; \mbox{for all} \;\;\; f \in \X(E).$$
\end{thm}

\begin{thm}[Computation of the norm]\label{thm2}
Suppose $E \subset \R^n$ is finite; let $N = \#(E)$ be the number of points in $E$. Then there exist linear functionals $\xi_1,\ldots,\xi_L : \X(E) \rightarrow \R$, such that $L \leq C N$ and
$$c \sum_{\ell=1}^L |\xi_\ell(f)|^p \leq \|f\|_{\X(E)}^p \leq C \sum_{\ell=1}^L |\xi_\ell(f)|^p \;\;\; \mbox{for all} \;\;\; f \in \X(E).$$
\end{thm}

For finite $E \subset \R^n$, we can say more about the linear map $T$ in Theorem \ref{thm1} and the linear functionals $\xi_1,\ldots,\xi_L$ in Theorem \ref{thm2} ; they have ``assisted bounded depth'', a notion that we explain below. We expect that this will be useful when we look for algorithms analogous to those in \cite{FK1,FK2}.

To motivate the idea of assisted bounded depth, let us compute the variance of $N$ given real numbers $x_1,\ldots,x_N$. Two standard formulas are
\begin{align*}
&\mbox{(Var 1)} \qquad \sigma^2 =\frac{1}{2N^2} \sum_{i,j=1}^N (x_i-x_j)^2 \;\;\; \mbox{and} \\
&\mbox{(Var 2)} \qquad \sigma^2 =\frac{1}{N} \sum_{i=1}^N (x_i-\overline{x})^2 \quad \mbox{where} \;\;\; \overline{x} = \frac{1}{N} \sum_{j=1}^N x_j.
\end{align*}
Note that it takes $\sim N^2$ computer operations to apply formula (Var 1), but only $\sim N$ operations to apply (Var 2). By precomputing $\overline{x}$, we save a lot of work. We will return later to the problem of computing a variance.

Now we explain the notion of assisted bounded depth. Let $E \subset \R^n$ be finite, and let $N=\#(E)$. Any linear functional $\omega$ on the space $\X(E)$ can be written in the form

\begin{equation} \label{intro1}
\omega(f) = \sum_{x \in E} \lambda(x) f(x), \;\;\; \mbox{for coefficients} \; \lambda(x) \; \mbox{independent of} \; f.
\end{equation}
We write $\sp(\omega)$ (the ``depth'' of $\omega$) to denote the number of nonzero coefficients $\lambda(x)$. A collection of functionals $\Omega = \{\omega_1,\ldots,\omega_S\}$ on $\X(E)$ will be said to be a set of ''assists'' provided we have 
$$\sum_{s=1}^S \sp(\omega_s) \leq C N.$$
The significance of this condition is that for any given $f \in \X(E)$, it takes at most $CN$ computer operations to calculate all the values $\omega_1(f),\ldots,\omega_S(f)$. The assists $\omega_s(f)$ ($s=1,\ldots,S$) will play the r\^{o}le of the mean $\overline{x}$ from our earlier remarks on computing a variance.

Let $\Omega = \{\omega_1,\ldots,\omega_S\}$ be a collection of assists for $\X(E)$. A linear functional $\xi : \X(E) \rightarrow \R$ is said to have ``$\Omega$-assisted bounded depth'' if it can be expressed in the form
\begin{equation} \label{intro2}
\xi(f) = \sum_{y \in E} \lambda(y) f(y) + \sum_{s=1}^S \mu_s \omega_s(f), \;\;\; \mbox{for all} \; f \in \X(E),
\end{equation}
where the $\lambda(y), \mu_s \in \R$ are independent of $f$, and at most $C$ of the coefficients $\lambda(y), \mu_s$ are nonzero.

Similarly, a linear map $T : \X(E) \rightarrow \X$ is said to have $\Omega$-assisted bounded depth if it can be expressed in the form
\begin{equation} \label{intro3}
Tf(x) = \sum_{y \in E} \lambda(x,y) f(y) + \sum_{s=1}^S \mu_s(x) \omega_s(f), \;\;\; \mbox{for all} \; f \in \X(E), \; x \in \R^n;
\end{equation}
where $\lambda(x,y), \mu_s(x) \in \R$ are independent of $f$, and for each $x \in \R^n$, at most $C$ of the $\lambda(x,y), \mu_s(x)$ are nonzero.

An operator or functional is said to have ``bounded depth'' if it has $\Omega$-assisted bounded depth with $\Omega$ equal to the empty set (no assists). See \cite{F4, L1}.

We will prove the following result.
\begin{thm}\label{thm3}
Suppose $E \subset \R^n$ is finite. Then there exist assists $\Omega =  \{\omega_1,\ldots,\omega_S\}$ such that the operator $T$ in Theorem \ref{thm1} , and the functionals $\xi_1,\ldots,\xi_L$ in Theorem \ref{thm2} can be taken to have $\Omega$-assisted bounded depth.
\end{thm}

If an oracle told us the coefficients in \eqref{intro1} for the assists $\omega_1,\ldots,\omega_S$, and in \eqref{intro3} for the operator $T$ in Theorems \ref{thm1}  and \ref{thm3} , then we could efficiently compute $Tf$ for any given $f \in \X(E)$. We could first precompute the assists $\omega_1(f),\ldots,\omega_S(f)$ using at most $CN$ computer operations by applying \eqref{intro1}. After that we could compute $Tf(x)$ at any given point $x \in \R^n$ using at most $C$ operations, by applying \eqref{intro3}.

Similarly, if an oracle told us the coefficients in \eqref{intro1} for the assists $\omega_1,\ldots,\omega_S$, and in \eqref{intro2} for the functionals $\xi_1,\ldots,\xi_L$ in Theorems \ref{thm2}  and \ref{thm3} , then we could compute the order of magnitude of $\|f\|_{\X(E)}$ for any given $f$ by first computing the assists $\omega_1(f),\ldots,\omega_S(f)$ and then computing $\xi_1(f),\ldots,\xi_L(f)$ in Theorems \ref{thm2}  and \ref{thm3} . This would require at most $CN$ computer operations. 

We hope that the coefficients arising in our formula for the assists $\omega_1,\ldots,\omega_S$ and the functionals $\xi_1,\ldots,\xi_L$ can be computed using at most $C N \log N$ operations; and that after one-time work $C N \log N$ we can compute the coefficients relevant to $Tf(x)$ at a given query point $x$ using at most $C \log N$ operations. This would provide an efficient algorithm for interpolation of data in the Sobolev space $\X  = \LR$ ($n < p < \infty$), analogous to the algorithms in Fefferman-Klartag \cite{FK1,FK2} for $C^m(\R^n)$.

Let us compare our present results to what we know about $C^m(\R^n)$. Switching over to $\X = C^m(\R^n)$, we recall the following results \cite{F3,F5,FK1}.

\begin{thm}\label{thm4}
For any $E \subset \R^n$, there exists a linear map $T : \X(E) \rightarrow \X$, such that 
$$Tf=f \; \mbox{on} \;E \;\;\; \mbox{and} \;\;\; \|Tf\|_\X \leq C \|f\|_{\X(E)} \;\;\; \mbox{for all} \; f \in \X(E).$$
Moreover, if $E$ is finite, then $T$ has bounded depth.
\end{thm}

\begin{thm}\label{thm5}
Let $E \subset \R^n$ be finite, and let $N=\#(E)$.

Then there exist subsets $S_1,\ldots,S_K \subset E$, with $K \leq C N$, and with $\#(S_k) \leq C$ for each $k$, such that
$$\|f\|_{\X(E)} \leq C \cdot \max_{1 \leq k \leq K} \|(f|_{S_k})\|_{\X(S_k)} \;\;\; \mbox{for all} \; f \in \X(E).$$
\end{thm}

\begin{cor}\label{cor1}
Let $E \subset \R^n$ be finite, and let $N = \#(E)$. Then there exist linear functionals $\xi_1,\ldots,\xi_L : \X(E) \rightarrow \R$, such that $L \leq CN$, each $\xi_\ell$ has bounded depth, and
$$c \cdot \max_{1 \leq \ell \leq L} |\xi_\ell(f)| \leq \|f\|_{\X(E)} \leq C \cdot \max_{1 \leq \ell \leq L} |\xi_\ell(f)| \;\;\; \mbox{for all} \; f \in \X(E).$$
\end{cor}

For the rest of the introduction, we again take $\X = \LR$ ($n < p < \infty$). Motivated by Theorems \ref{thm4}, \ref{thm5}  and Corollary 1, one might wonder whether we can dispense with the assists in Theorem \ref{thm3} , and take $T,\xi_1,\ldots,\xi_L$ to have bounded depth. An optimist might even conjecture that $\|f\|_{\X(E)}^p$ is comparable to $\sum_{\ell=1}^L \lambda_\ell \cdot \|(f|_{S_\ell})\|_{\X(S_\ell)}^p$ for subsets $S_1,\ldots,S_L \subset E$ and coefficients $\lambda_1,\ldots,\lambda_L$ independent of $f$, with $L \leq C N$ and $\#(S_\ell) \leq C$ for each $\ell$.

In fact, a counterexample \cite{FIL} shows that the extension operator $T$ in Theorem \ref{thm3}  cannot be taken to have bounded depth. On the other hand, the remarkable work of Batson-Spielman-Srivastava \cite{BSS} on ``sparsification'' gives hope that bounded-depth $\xi$'s may exist. We illustrate the result of \cite{BSS} by returning to the computation of the variance of real numbers $x_1,\ldots,x_N$.

Given $\epsilon>0$, there exist coefficients $\lambda_1,\ldots,\lambda_L >0$ and integers $i_1,\ldots,i_L$, $j_1,\ldots,j_L \in \{1,\ldots,N\}$ such that $L \leq C(\epsilon)N$, and such that the variance of $x_1,\ldots,x_N$ differs by at most a factor of $(1+\epsilon)$ from
$$\sum_{\ell=1}^L \lambda_\ell (x_{i_\ell} - x_{j_\ell})^2,$$
for any real numbers $x_1,\ldots,x_N$. Here, $C(\epsilon)$ depends only on $\epsilon$. Thus, one can compute a variance to within, say, a one percent error by using $O(N)$ functionals without assists. This is merely a special case of \cite{BSS}. See also \cite{SS,ST}.

Dealing with sums of $p^\th$ powers $(p>2)$ is more difficult than dealing with sums of squares. We don't know whether sparsification applies to our problems, or whether the functionals $\xi_1,\ldots,\xi_L$ in Theorems \ref{thm2}  and \ref{thm3}  can be taken to have bounded depth.

Theorems 1,2,3 deal with the homogeneous Sobolev space $\LR$. It is easy to pass from these results to analogous theorems for the inhomogeneous Sobolev space $\WR$; see section \ref{sec_pf}.

So far, we have looked for functions $F \in \LR$ that agree perfectly with a given $f : E \rightarrow \R$. More generally, we can look for functions $F$ that agree approximately with a given $f$. To do so (for $E$ finite), we specify a weight function $\mu : E \rightarrow (0,\infty)$ along with our usual $f : E \rightarrow \R$. We then look for $F \in \LR$ and $M \geq 0$ such that
$$\sum_{x \in E} \mu(x) \lvert F(x) - f(x)\rvert^p \leq M^p \;\; \mbox{and} \;\; \|F\|_{\LR} \leq M,$$
with $M$ having the smallest possible order of magnitude.

We believe that Theorems 1,2,3, and the algorithms to which (we hope) they give rise, can be extended to solve this more general problem. Analogous results for $C^m(\R^n)$ are given in \cite{F1,F5,FK1,FK2}.

Let us sketch the proof of Theorem \ref{thm1}  for the case of finite sets $E$. The general case (infinite E) follows by taking a Banach limit, and the proofs of Theorems \ref{thm2}  and \ref{thm3}  arise by examining carefully our proof of Theorem \ref{thm1} . We will oversimplify the discussion, so that the spirit of the ideas comes through.

We first set up a bit more notation. We will write $Q,Q',$ etc. to denote cubes in $\R^n$ (with sides parallel to the coordinate axes). We write $\delta_Q$ to denote the sidelength of the cube $Q$. For any real number $A>0$, we write $AQ$ to denote the cube having the same center as $Q$, and having sidelength $A \cdot \delta_Q$. If $F$ is a $C^{m-1}$ function on a neighborhood of a point $x \in \R^n$, then we write $J_x(F)$ (the ``jet'' of $F$ at $x$) to denote the $(m-1)^\rst$ degree Taylor polynomial of $F$ at $x$. Thus, $J_x(F)$ belongs to $\cP$, the vector space of real $(m-1)^\rst$ degree polynomials on $\R^n$ (``jets'').

Next, we introduce a local variant of the problem addressed by Theorem \ref{thm1} . Suppose we are given the following data:
\begin{itemize}
\item A cube $Q$.
\item A finite set $E \subset 3Q$.
\item A function $f : E \rightarrow \R$.
\item A point $x \in E$.
\item A jet $P \in \cP$.
\end{itemize}

The \underline{local interpolation problem} $LIP(Q,E,f,x,P)$ is the problem of finding a function $F \in \X(3Q)$, such that
\begin{equation} \label{intro4}
F = f \; \mbox{on} \; E, \; J_x (F) = P,
\end{equation}
and
\begin{equation} \label{intro5}
\|F\|_{\X(3Q)} \;\; \mbox{is as small as possible, up to a factor}  \; C.
\end{equation}
To prove Theorem \ref{thm1} , we will solve an arbitrary local interpolation problem $LIP(Q,E,f,x,P)$ in such a way that the interpolant $F$ depends linearly on the data $(f,P)$ for fixed $Q,E,x$.

Once we give such a linear solution, then Theorem \ref{thm1}  for finite $E$ follows easily by taking $Q$ to be a large enough cube containing $E$.

We will measure the difficulty of a local interpolation problem by assigning to it a ``label'' $\cA$. Here, $\cA$ is any subset of the set $\cM$ of all multi-indices $\alpha = (\alpha_1,\ldots,\alpha_n)$ of order $|\alpha| = \alpha_1 + \cdots + \alpha_n \leq m-1$. Labels $\cA$ come with a (total) order relation $<$. Roughly speaking, whenever $\cA < \cA'$, a local interpolation problem that carries the label $\cA$ is easier than one that carries the label $\cA'$. In our ordering, the empty set $\emptyset$ is maximal, and the set $\cM$ is minimal. Accordingly, $\emptyset$ labels the hardest interpolation problems, and $\cM$ labels the easiest problems.

To assign a label $\cA$ to a local interpolation problem, we introduce the convex set
\begin{equation} \label{intro6}
\sigma(x,E) = \{J_x(F) : F \in \X, \; \|F\|_\X \leq 1, \; F = 0 \; \mbox{on} \; E\} \subset \cP.
\end{equation}
This set is clearly relevant to the problem of finding $F \in \X$ such that $F=f$ on $E$; it measures our freedom of action in assigning $J_x(F)$ for such an interpolant.

Roughly speaking, a local interpolation problem $LIP(Q,E,f,x,P)$ carries a label $\cA \subset \cM$ if for each $\alpha \in \cA$, the monomial
\begin{equation} \label{intro7}
P_\alpha(y) = \delta_Q^{m - |\alpha| - n/p} (y-x)^\alpha \;\;\; \mbox{belongs to} \;\; \sigma(x,E).
\end{equation}
(Recall that $\delta_Q$ is the sidelength of $Q$, and that $\X = \LR$.)

Thus, we have defined the notion of a local interpolation problem that carries a given label $\cA$. We list a few relevant properties of labels.
\begin{itemize}
\item[\refstepcounter{equation} (\arabic{equation})\label{intro8}] Any problem that carries a label $\cA$ also carries the label $\cA'$ for any $\cA' \subset \cA$. If $\cA' \subset \cA$, then $\cA < \cA'$, where $<$ is the order relation used to rate the difficulty of a local interpolation problem.
\item[\refstepcounter{equation} (\arabic{equation})\label{intro9}] Every local interpolation problem carries the label $\emptyset$, since we then make no requirement that any monomial $P_\alpha$ belongs to $\sigma(x,E)$.
\item[\refstepcounter{equation} (\arabic{equation})\label{intro10}] On the other hand, a local interpolation problem carries the label $\cM$ only when $E$ is the empty set.
\item[\refstepcounter{equation} (\arabic{equation})\label{intro11}] The assignment of a label to $LIP(Q,E,f,x,P)$ depends only on $(Q,E,x)$; the function $f$ and the jet $P$ play no r\^{o}le.
\end{itemize}

For each label $\cA$, we will prove the following

\underline{MAIN LEMMA FOR $\cA$:} Any local interpolation problem $LIP(Q,E,f,x,P)$ that carries the label $\cA$ has a solution $F$ that depends linearly on the data $(f,P)$. 

In particular, the MAIN LEMMA FOR $\cA=\emptyset$ tells us that every local interpolation problem admits a solution $F$ depending linearly on the data $(f,P)$. (See \eqref{intro9}.) Consequently, the proof of Theorem \ref{thm1}  reduces to the task of proving the MAIN LEMMA for every label $\cA$. We proceed by induction on $\cA$, with respect to the order $<$. 

\underline{In the BASE CASE}, we take $\cA = \cM$, the smallest possible label under $<$. The MAIN LEMMA FOR $\cM$ holds trivially. For any local interpolation problem $LIP(Q,E,f,x,P)$ with label $\cM$, we simply take our interpolant to be $F=P$. We have $F = f$ on $E$ vacuously, since $E$ is empty; and $\|F\|_{\X} = 0$. (See \eqref{intro10}.)

\underline{For the INDUCTION STEP}, we fix a label $\cA \neq \cM$, and make the following inductive assumption:
\begin{equation} \label{intro12}
\mbox{The MAIN LEMMA FOR} \; \cA' \; \mbox{holds for all} \; \cA' < \cA.
\end{equation}
Under this assumption we will prove the MAIN LEMMA FOR $\cA$.

Thus, let
\begin{align}
LIP^\circ = LIP(Q^\circ,E^\circ,f^\circ,x^\circ,P^\circ) \; \mbox{be a local interpolation problem}&\notag{}\\
\mbox{that carries the label} \; \cA&. \label{intro13}
\end{align}
Our task is to find an interpolant $F \in \X(3Q^\circ)$ that solves the interpolation problem $LIP^\circ$ and depends linearly on the data $(f^\circ,P^\circ)$. To do so, we proceed in several steps, as follows.

\underline{STEP 1:} We partition $3Q^\circ$ into ``Calder\'on-Zygmund'' subcubes $Q_\nu$ ($\nu = 1,\ldots,\nu_{\max}$). Each $Q_\nu$ is such that $E_\nu : = E \cap 3Q_\nu$ is non-empty. We pick a point $x_\nu \in E_\nu$ for each $\nu$.

\underline{STEP 2:} For each $\nu$, we pick a jet $P_\nu$, depending linearly on the data $(f^\circ,P^\circ)$. We then set $f_\nu := f|_{E_\nu}$, and consider the local interpolation problem
\begin{equation} \label{intro14}
LIP_\nu = LIP(Q_\nu,E_\nu,f_\nu,x_\nu,P_\nu).
\end{equation}

\underline{STEP 3:} The partition of $3Q^\circ$ into Calder\'on-Zygmund cubes $Q_\nu$ has been defined to guarantee that each local interpolation problem $LIP_\nu$ carries a label $\cA' < \cA$ (except in the trivial case in which $E_\nu$ contains at most one point). Hence, by our inductive assumption \eqref{intro12}, we can solve each problem $LIP_\nu$, obtaining an interpolant $F_\nu \in \X(3Q_\nu)$ that depends linearly on $(f_\nu,P_\nu)$. Since also $f_\nu$ and $P_\nu$ depend linearly on $(f^\circ,P^\circ)$ (see STEP 2), it follows that each $F_\nu$ depends linearly on $(f^\circ,P^\circ)$.

\underline{STEP 4:} We introduce a partition of unity
$$1 = \sum_\nu \theta_\nu \;\;\; \mbox{on} \; 3Q^\circ, \; \mbox{with each} \; \theta_\nu \; \mbox{supported in} \; 3 Q_\nu.$$
Using that partition of unity, we patch together the local interpolants $F_\nu$ into a single interpolant $F$, by taking
$$F = \sum_\nu \theta_\nu F_\nu.$$
Thus, we obtain a function $F \in \X(3Q^\circ)$ that satisfies $F = f^\circ$ on $E^\circ$, and $J_{x^\circ} (F) = P^\circ$.
It remains to show that the seminorm of $F$ in $\X(3Q^\circ)$ is as small as possible up to a factor $C$.

This has a chance only if the jets $P_\nu$ in STEP 2 are carefully picked to be mutually consistent. Arranging this consistency is the hardest part of our proof, and the main difference between our arguments here and those in \cite{F1,F2,F3,F4} for $C^m(\R^n)$ interpolation.

To assign the jets $P_\nu$ in STEP 2 and achieve their needed mutual consistency, we first pick out from among the Calder\'on-Zygmund cubes $Q_\nu$ the subcollection of \underline{keystone cubes}. A given $Q_\nu$ is a keystone cube if every $Q_{\nu'}$ that meets $100Q_\nu$ is at least as big as $Q_\nu$.

We carry out STEP 2 by first assigning jets $P_\nu$ to the keystone cubes. Each keystone cube may be treated separately, without worrying about mutual consistency. For any Calder\'on-Zygmund cube $Q_\nu$ other than a keystone cube, we carefully select a nearby keystone cube $Q_{\kappa(\nu)}$, and then simply set 
$$P_\nu = P_{\kappa(\nu)}; \;\mbox{ the right-hand side has already been defined}.$$
The above procedure associates jets $P_\nu$ to all the Calder\'on Zygmund cubes $Q_\nu$, in such a way that the required mutual consistency is guaranteed.

This concludes our sketch of the proof of Theorem \ref{thm1}. We again warn the reader that it is oversimplified. Even the notation and definitions in the subsequent sections differ from those presented in the introduction.

Throughout this paper, we study $\LR$ only for $n < p  < \infty$. It would be natural to investigate the more general case $\frac{n}{m} < p < \infty$, since then any $F \in \LR$ is continuous, and may therefore be restricted to an arbitrary subset $E \subset \R^n$.

We briefly review the earlier work on Sobolev extensions. The first breakthrough was the discovery by P. Shvartsman \cite{S1} that Theorem \ref{thm1} for the Sobolev space $L^{1,p}(\R^n)$ (i.e., $m=1$) holds with $T$ given by the classical proof of the Whitney extension theorem \cite{W1}. Shvartsman gave a formula \cite{S1} for the order of magnitude of the seminorm $\|f\|_{\X(E)}$ when $\X = L^{1,p}(\R^n)$ and $E \subset \R^n$ is arbitrary (possibly infinite). See also \cite{S2}. The proof of Theorem \ref{thm1} for the Sobolev space $L^{m,p}(\R)$ (i.e., $n=1$) was given by G.K. Luli \cite{L2} when $E \subset \R$ is finite and by P. Shvartsman \cite{S2} in the general case.

The next significant progress was the proof by A. Israel of Theorems \ref{thm1} and \ref{thm2} for the case $\X = L^{2,p}(\R^2)$ (with the bound $L \leq C N^2$ in place of $L \leq C N$ in Theorem \ref{thm2}). See \cite{I}. The proof in \cite{I} makes explicit use of the keystone cubes. 

Our main technical achievement here is to combine the ideas used previously for interpolation problems with labels, and those used to exploit keystone cubes. In particular, we have to dispense with the convex sets called $\Gamma(x,M)$ in \cite{FK1,FK2}; these sets played a crucial r\^{o}le in our earlier analysis of $C^m(\R^n)$.

P. Shvartsman has lectured at several workshops on results and ideas that appear closely related to ours. He is now writing up his results, and we look forward to studying them; see also \cite{S3}. It would be interesting to study the relationship between our keystone cubes and the ``important'' cubes from \cite{S3}.

\textbf{Acknowledgments}

We are grateful to B. Klartag and A. Naor for introducing us to the idea of sparsification, and for several lively discussions. We are grateful also to the NSF, the ONR, and the American Institute of Math (AIM) for generous support. Key ideas arose during fruitful workshops held at AIM and at the College of William and Mary. We are grateful to P. Shvartsman and Nahum Zobin for lively discussions.

\section{Notation and Preliminaries}\label{sec_not}
\setcounter{equation}{0}

\subsection{Notation}
Fix integers $m, n \geq 1$ and a real number $p>n$. Unless we say otherwise, constants written $c,c',C,C',$ etc. depend only on $m,n,$ and $p$. They are called ``universal'' constants. The lower case letters denote small (universal) constants while the upper case letters denote large (universal) constants.

For non-negative quantities $A,B$, we write $A \simeq B$, $A \lesssim B$, or $A \gtrsim B$ to indicate (respectively) that $c B \leq A \leq CB$, $A \leq CB$, or $A \geq cB$.

This paper is divided into sections. The label (p.q) refers to formula (q) in section p. Within section p, we abbreviate (p.q) to (q).

A \emph{cube} $Q \subset \R^n$ is any subset of the form:
$$Q = \{a\} + (-\delta,\delta]^n \qquad (a \in \R^n, \; \delta>0).$$
The \textit{sidelength} of $Q$ is denoted $\delta_Q = 2\delta$, while the \textit{center} of $Q$ is the point $a$. For $A > 0$ let $A Q$ denote the cube having the same center as $Q$ but with sidelength $A \delta_Q$. A \textit{dyadic} cube $Q \subset \R^n$ is any cube of the form:
$$Q = (j_1 \cdot 2^k, (j_1+1) \cdot 2^k] \times (j_2 \cdot 2^k, (j_2+1) \cdot 2^k]  \times \cdots \times (j_n \cdot 2^k, (j_n+1) \cdot 2^k]\qquad (j_1,j_2,\ldots,j_n,k \in \Z).$$
To \textit{bisect} a cube $Q \subset \R^n$ is to partition it into $2^n$ disjoint subcubes of sidelength $\frac{1}{2} \delta_Q$. These subcubes are called the \textit{children} of $Q$. If $Q \subsetneq Q'$ are dyadic cubes we say that $Q'$ is an \textit{ancestor} of $Q$. Every dyadic cube $Q$ has a smallest ancestor called its \textit{parent}, which we denote by $Q^+$.

We use the following notation:
\begin{align*}
& |x| := |x|_\infty = \max\{|x_1|,\ldots,|x_n|\} && (x = (x_1,\ldots,x_n) \in \R^n);\\
& \dist(x,\Omega) := \inf \{|x-y|: y \in \Omega\} && (\Omega \subset \R^n, \; x \in \R^n); \\
& \dist(\Omega', \Omega) := \inf \{|x-y|:x \in \Omega', y \in \Omega\} && (\Omega, \Omega' \subset \R^n);\\
& B(\Omega,R) : = \{x \in \R^n : \dist(x,\Omega) \leq R\} &&(\Omega \subset \R^n, \; R>0); \; \mbox{and} \\
& \diam(S) := \max\{ |x-y| : x,y \in S\} &&(S \subset \R^n \; \mbox{finite}).
\end{align*}
The analogous objects defined with respect to the Euclidean norm $|x|_2 = (|x_1|^2 + \cdots + |x_n|^2)^{1/2}$ are denoted by $\dist_2(x,\Omega)$, $\dist_2(\Omega',\Omega)$, $B_2(\Omega, R)$ and $\diam_2(S)$.

We write $\cM$ for the collection of all multi-indices $\alpha$ of order $|\alpha| \leq m -1$. If $\alpha$ and $\beta$ are multi-indices, then $\delta_{\alpha \beta}$ denotes the Kronecker delta.

If $F$ is a $C^{m-1}$ function on a neighborhood of a point $y \in \R^n$, then we write $J_y(F)$ (the ``jet'' of $F$ at $y$) for the $(m-1)^\rst$ degree Taylor polynomial of $F$ at $y$.

Let $U \subset \R^n$ be a domain. The homogeneous Sobolev space $L^{m,p}(U)$ consists\footnote{Here, the derivatives $\partial^\alpha F$ are defined as distributions. Strictly speaking, $F \in L^{m,p}$ is defined only up to a set of measure zero. Under our assumption $n < p < \infty$, the Sobolev theorem allows us to regard any $F \in L^{m,p}$ as a function in $C_{loc}^{m-1,1- n/p}$. For a discussion of these technicalities, see \cite{Evans}.} of all functions $F: U \rightarrow \R$ with finite seminorm
$$\|F\|_{L^{m,p}(U)} := \max_{|\alpha| = m} \left( \int_{U} |\partial^\alpha F(x)|^p dx \right)^{1/p}.$$
Similarly, the inhomogenous Sobolev space $W^{m,p}(U)$ consists of all functions with finite norm
$$\|F\|_{W^{m,p}(U)} := \max_{|\alpha| \leq m} \left( \int_{U} |\partial^\alpha F(x)|^p dx \right)^{1/p}.$$
For $0<s<1$, the homogeneous H\"older space $\dot{C}^{m-1,s}(U)$ consists of all functions $F : U \rightarrow \R$ with finite seminorm
$$\|F\|_{\dot{C}^{m-1,s}(U)} := \;\;\;\; \max \left\{\frac{|\partial^\alpha F (x) - \partial^\alpha F(y)|}{|x-y|^s} : x,y \in U, \; |\alpha| = m-1 \right\}.$$
Likewise, the inhomogeneous H\"older space $C^{m-1,s}(U)$ consists of all functions with finite norm
\begin{align*}
\|F\|_{C^{m-1,s}(U)} :=& \;\;\;\; \max \left\{\frac{|\partial^\alpha F (x) - \partial^\alpha F(y)|}{|x-y|^s} : x,y \in U, \; |\alpha| = m-1 \right\} \\
&+\max \bigl\{\left|\partial^\alpha F(x)\right| : x \in U, \; |\alpha| \leq m-1 \bigr\} .\\
\end{align*}

Let $\cP$ denote the space of real-valued $(m-1)^{\rst}$ degree polynomials on $\R^n$. Then $\cP$ is a vector space of dimension $D := \mbox{dim}(\cP)$. For each $x \in \R^n$, the multiplication $\odot_x$ on $\cP$ is defined by
$$P \odot_x Q := J_x(P \cdot Q)\qquad (P,Q \in \cP).$$

Given a finite subset $E' \subset \R^n$, the space of Whitney fields on $E'$ is denoted
$$Wh(E') := \bigl\{(P_x)_{x \in E'} : P_x \in \cP \; \mbox{for all} \; x \in E' \bigr\}.$$
For every function $F$ that is $C^{m-1}$ on a neighborhood of $E'$, we define the Whitney field $J_{E'} (F) := (J_x(F))_{x \in E'} \in Wh(E')$ (the ``jet'' of $F$ on $E'$).

Given an arbitrary subset $E \subset \R^n$ and a finite subset $E' \subset \R^n$, we define the trace space
$$L^{m,p}(E;E') := \left\{ (f,\vP) : \vP \in Wh(E'), \; f : E \rightarrow \R, \; \exists \; F \in \LR \; \mbox{with} \; F|_E = f  \; \mbox{and} \; J_{E'}(F) = \vP\right\}.$$ 
This space comes equipped with the natural seminorm
$$\|(f,\vec{P})\|_{L^{m,p}(E;E')} := \inf \bigl\{\|F\|_{\LR} : F|_E = f \; \mbox{and} \; J_{E'} (F) = \vP \bigr\}.$$
If $E'$ is empty, then $L^{m,p}(E;E')$ is simply $L^{m,p}(E)$, and we take $\vP \equiv 0$. Similarly, if $E$ is empty, then $L^{m,p}(E;E') = Wh(E')$ (equipped with the obvious seminorm), and we take $f \equiv 0$. For $z \in \R^n$, we often write $L^{m,p}(E;z)$ in place of $L^{m,p}(E;\{z\})$. 

A function $F \in \LR$ is called an \textit{extension} of $(f,\vP) \in L^{m,p}(E;E')$ if 
\begin{equation*}
F= f \; \mbox{on} \; E \;\; \mbox{and} \;\; J_{E'} (F) = \vec{P}.
\end{equation*}
By a \textit{near optimal} extension of $(f,\vP)$, we mean an extension that satisfies 
$$\|F\|_{\LR} \leq C \|(f,\vP)\|_{L^{m,p}(E;E')} \;\; \mbox{for some universal constant} \; C.$$ 
An \textit{extension operator} is a linear map $T:L^{m,p}(E;E') \rightarrow \LR$, such that $T(f,\vP)$ is an extension of $(f,\vP)$ for every $(f,\vP) \in L^{m,p}(E;E')$. We say that $T$ is \textit{bounded} if 
$$\|T\| := \sup \left\{ \|T(f,\vec{P})\|_{L^{m,p}(\R^n)}: \|(f,\vec{P})\|_{L^{m,p}(E;E')} \leq 1 \right\} \leq C.$$

Let $\lambda$ be a linear map from $L^{m,p}(E;E')$ into either $V=\R$ or $V= \cP$. The \textit{depth of $\lambda$} (denoted by $\sp(\lambda)$) is the smallest non-negative integer $d$, such that there exist points $x_1,\ldots,x_r \in E$ and $y_{1},\ldots,y_s \in E'$ with $r + s = d$, multi-indices $\alpha_{1},\ldots,\alpha_s \in \cM$ and elements $v_1,\ldots,v_{d} \in V$, such that
\begin{align*}
\lambda(f,\vP) = f(x_1) \cdot v_1 + \cdots + f(x_r) \cdot v_r + \partial^{\alpha_1} P_{y_1}(y_1) \cdot v_{r+1} + & \cdots + \partial^{\alpha_{s}} P_{y_{s}}(y_{s}) \cdot v_{r+s} \\ 
& \mbox{for all} \; (f,\vP) \in L^{m,p}(E;E').
\end{align*}
We now introduce the notion of ``assisted bounded depth'' linear maps on $L^{m,p}(E;E')$.
\begin{defn}
Let $V = \R$ or $V = \cP$. Given a collection of linear functionals $\Omega \subset [\LE]^*$ and an integer $d \geq 0$, we say that a linear map $\lambda: L^{m,p}(E;E') \rightarrow V$ has $\Omega$-assisted depth $d$ if there exist $\omega_1,\ldots,\omega_d  \in \Omega$, elements $v_1,\ldots,v_d \in V$ and a linear map $\widetilde{\lambda}: L^{m,p}(E;E') \rightarrow V$ with $\sp(\widetilde{\lambda}) \leq d$, such that
\begin{align*}
\lambda(f,\vec{P}) = \omega_1(f) \cdot v_1 + \cdots + \omega_{d}(f) \cdot v_{d} + \widetilde{\lambda}(f,\vP) \;\;\; \mbox{for all} \; (f,\vP) \in L^{m,p}(E;E').
\end{align*}

Let $U \subset \R^n$ be a domain. We say that a linear map $T : L^{m,p}(E;E') \rightarrow L^{m,p}(U)$ has $\Omega$-assisted depth $d$ if the linear map 
$$ (f,\vP) \in L^{m,p}(E;E')  \longmapsto J_x(T(f,\vec{P})) \in \cP \;\; \mbox{has} \; \Omega\mbox{-assisted depth} \; d, \; \mbox{for all} \; x \in U.$$

If $T$ (resp. $\lambda$) has $\Omega$-assisted depth $d$ for some constant $d$ that depends only on $m$, $n$ and $p$, then we say that $T$ (resp. $\lambda$) has $\Omega$-assisted bounded depth.
\end{defn}
For a finite subset $E \subset \R^n$ and $E'$ empty, we have introduced the notion of an $\Omega$-assisted bounded depth map $T : L^{m,p}(E) \rightarrow L^{m,p}(U)$ above and also in the introduction. It is not very difficult to see that these definitions are equivalent; we leave this as an exercise for the reader.

Let $E \subset \R^n$ and $x \in \R^n$. Then we define
$$\sigma(x,E) = \left\{ P \in \cP :  \exists \;  \phi \in \LR \; \mbox{with} \;  J_{x}\phi = P, \; \phi|_E=0,\; \mbox{and} \; \|\phi\|_{\LR} \leq 1 \right\}.$$
Thus, $\sigma(x,E)$ is a centrally-symmetric ($P \in \cP \implies -P \in \cP$ ) convex subset of $\cP$. 

\subsection{Preliminaries}

For each $x \in \R^n$, $\delta>0$, we define a norm on $\cP$ by
\begin{equation}\label{scalednorm}
\left|P\right|_{x,\delta} := \Bigl(\sum_{|\alpha| \leq m-1} \lvert \partial^\alpha P(x)\rvert^p \delta^{n + (|\alpha|-m)p} \Bigr)^{1/p}  \qquad (P \in \cP).
\end{equation}
For $x' \in \R^n$, we have the Taylor expansion
$$\partial^\alpha P(x) = \sum_{|\gamma| \leq m-1-|\alpha|} \frac{1}{\gamma!} \partial^{\alpha + \gamma} P(x') \cdot (x-x')^\gamma \qquad (|\alpha| \leq m-1).$$
Thus, the norms defined in \eqref{scalednorm} satisfy the inequality
\begin{equation}\label{ti0}
 \left|P\right|_{x,\delta} \leq C' \left|P\right|_{x',\delta} \qquad (x,x' \in \R^n, |x-x'| \leq C \delta).
\end{equation}

We will consider the following result in several settings.

\noindent\underline{The Sobolev Inequality for $U$:} For every $F \in L^{m,p}(U)$, we have
\begin{equation} \label{SET} \lvert \partial^\alpha (J_y F - F)(x) \rvert \leq C \cdot \lvert x - y \rvert^{m-|\alpha| - n/p} \cdot \|F\|_{L^{m,p}(U)} \quad (x,y \in U, \; |\alpha| \leq m-1). \end{equation}
For the appropriate class of domains $U$, the constant $C$ in \eqref{SET} will be universal. The usual proof of the Sobolev Inequality for $U = \R^n$ establishes \eqref{SET} when $U$ is any cube $Q \subset \R^n$. By applying a linear transformation, we obtain \eqref{SET} also when $U$ is an axis-parallel rectangular box that is non-degenerate (i.e., whose sidelengths differ at most by a universal constant factor). 

Finally, we consider the case when $U= R_1 \cup R_2$ is the union of two such rectangular boxes; we assume that $R_1$ and $R_2$ have an interior point in common. Take $x,y \in U$ and fix some function $F \in L^{m,p}(U)$. If $x$ and $y$ lie in the same box $R_i$ then \eqref{SET} follows from the Sobolev Inequality for $R_i$. Thus, we may assume that $x$ and $y$ lie in different boxes. Let $z \in R_1 \cap R_2$ be any point that satisfies
$$|x-z| \leq |x-y| \;\; \mbox{and} \;\;  |y-z| \leq |x-y|.$$
(Here, we exploit the axis-parallel structure of $R_1$ and $R_2$.) For any $\alpha$ with $|\alpha| \leq m-1$, since $|\alpha| + n/p - m < 0$, we have
\begin{align}
\lvert \partial^\alpha (J_y F - J_x F)(z) \rvert \cdot |x-y|^{|\alpha| + n/p - m} \leq \;\;& \lvert \partial^\alpha (J_y F - J_z F)(z) \rvert \cdot |y-z|^{|\alpha| + n/p - m} \notag{} \\
+ &\lvert \partial^\alpha (J_x F - J_z F)(z) \rvert \cdot |x-z|^{|\alpha| + n/p - m}, \notag{}
\end{align}
which is bounded by $C' \|F\|_{L^{m,p}(R_1 \cup R_2)}$, thanks to the Sobolev Inequality for $R_i$ ($i=1,2$). Inserting this inequality on the right-hand side of \eqref{ti0}, where we have set $P = J_y F - J_x F$, $\delta = |x-y|$ and $x' = z$, we obtain \eqref{SET}. This completes the proof of the Sobolev Inequality for $U = R_1 \cup R_2$.

\section{Plan for the Proof}\label{sec_plan}
\setcounter{equation}{0}

Let $E \subset \R^n$ (finite) and $z \in \R^n$ be given. Our main goal is to prove the following theorem.
\begin{extthm}
There exist a linear map $T : L^{m,p}(E;z) \rightarrow \LR$, a map $M: L^{m,p}(E;z) \rightarrow \R_+$ and a collection of linear functionals $\Omega \subset [\LE]^*$, that satisfy the following properties:
\begin{align}
&\mathbf{(E1)} \; T \; \mbox{is an extension operator}.\notag{}\\
&\mathbf{(E2)} \; \|(f,P)\|_{L^{m,p}(E;z)} \leq \|T(f,P)\|_{\LR} \leq C \cdot \|(f,P)\|_{L^{m,p}(E;z)}, \;\; \mbox{and} \notag{} \\
&\mathbf{(E3)} \;C^{-1} \cdot M(f,P) \leq \|T(f,P)\|_{\LR} \leq C \cdot M(f,P), \;\; \mbox{for each} \; (f,P) \in L^{m,p}(E;z).\notag{}\\
&\mathbf{(E4)} \; \sum_{\omega \in \Omega} \sp(\omega) \leq C \cdot \#(E).\notag{} \\
&\mathbf{(E5)} \;T \; \mbox{has} \; \Omega\mbox{-assisted bounded depth}.\notag{} \\
&\mathbf{(E6)} \; \mbox{There exists a collection of linear functionals} \; \Xi\subset [L^{m,p}(E;z)]^*, \; \mbox{such that} \notag{}\\
&\qquad\qquad \mathbf{(a)} \; \mbox{each functional in}\; \Xi \; \mbox{has}\; \Omega\mbox{-assisted bounded depth},\notag{}\\
&\qquad\qquad \mathbf{(b)} \; \# (\Xi) \leq C \cdot \#(E), \; \mbox{and} \notag{} \\
&\qquad\qquad \mathbf{(c)} \; M(f,P) = \left(\sum_{\xi \in \Xi} |\xi(f,P)|^p \right)^{1/p} \;\;\; \mbox{for each} \; (f,P) \in L^{m,p}(E;z). \notag{}
\end{align}
Here, $C$ depends only on $m$, $n$ and $p$.
\end{extthm}
Theorem \ref{thm1} (for finite $E$) and Theorems \ref{thm2}, \ref{thm3} follow easily from the above extension theorem. (See section \ref{sec_pf}.)

\subsection{Order Relation on Labels}

To prove the Extension Theorem for $(E,z)$ we proceed by induction on the ``shape'' of the convex subsets $\sigma(x,E) \subset \cP$, where $x$ ranges over $E$. The shape of a single convex subset $\sigma \subset \cP$ will be defined in terms of a label $\cA \subset \cM$. We use the following order relation on labels.

\begin{defn} \label{order}
Given distinct elements $\alpha=(\alpha_1,\ldots,\alpha_n)$, $\beta=(\beta_1,\ldots,\beta_n) \in \cM$, let $k \in \{0,\ldots,n\}$ be maximal subject to the condition $\sum_{i=1}^k \alpha_i \neq \sum_{i=1}^k \beta_i$. We write $\alpha < \beta$ if
$$\sum_{i=1}^k \alpha_i < \sum_{i=1}^k \beta_i.$$
Given distinct labels $\cA$, $\oA \subset \cM$, we write $\cA < \oA$ if the minimal element of the symmetric difference $\cA \Delta \oA$ (under the order $<$ on elements) lies in $\cA$.
\end{defn}

The next lemma is immediate from the definition.
\begin{lem}\label{lem_order} The following properties hold:
\begin{itemize}
\item If $\alpha, \beta \in \cM$ and $|\alpha| < |\beta|$ then $\alpha < \beta$. 
\item If $\cA \subsetneq \oA \subset \cM$ then $\oA < \cA$. In particular, the empty set is maximal and the whole set $\cM$ is minimal under the order on labels.
\end{itemize}
\end{lem}

\subsection{Polynomial Bases}\label{poly_basis}
In this section, we define the notion of a labeled basis for a symmetric convex subset $\sigma \subset \cP$. 

\begin{defn}\label{basis2}
Given $\cA \subset \cM$, $x \in \R^n$, $\epsilon > 0$ and $\delta > 0$, we say that $(P_\alpha)_{\alpha \in \cA} \subset \cP$ forms an $(\cA,x,\epsilon,\delta)$-basis for $\sigma$ if 
\begin{align*}
&\mathbf{(B1)} \; P_\alpha \in \epsilon \cdot \delta^{n/p + |\alpha| - m} \cdot \sigma, \qquad \; \mbox{for} \; \alpha \in \cA;
\\
&\mathbf{(B2)} \; \partial^\beta P_\alpha(x) = \delta_{\alpha \beta},\qquad\qquad \;\;\; \mbox{for}\; \alpha, \beta \in \cA; \; \mbox{and}\\
&\mathbf{(B3)} \; \lvert\partial^\beta P_\alpha(x)\rvert \leq \epsilon \cdot \delta^{|\alpha|-|\beta|}, \qquad \mbox{for} \; \alpha \in \cA,\; \beta \in \cM \; \mbox{with}  \; \beta > \alpha.
\end{align*}
\end{defn}
\begin{rek} \label{rem1} 
The above definition is monotone in the parameters $(\epsilon, \delta)$ in the following sense: Suppose that $(P_\alpha)_{\alpha \in \cA}$ forms an $(\cA,x,\epsilon,\delta)$-basis for some convex subset $\sigma \subset \cP$. If $\epsilon' \geq \epsilon$ and $0 < \delta' \leq \delta$ then $(P_\alpha)_{\alpha \in \cA}$ forms an $(\cA,x,\epsilon',\delta')$-basis for $\sigma$.
\end{rek}

\subsection{The Main Lemma}\label{theml}
Fix a collection of multi-indices $\cA \subset \cM$. We prove the following by induction with respect to the label $\cA$.
\begin{ml}
There exists a constant $\epsilon=\epsilon(\cA)$, depending only on $\cA$, $m$, $n$ and $p$, such that the following holds. Let $E \subset \R^n$ and $z \in \R^n$ satisfy $2 \leq \#(E \cup \{z\}) < \infty$, and assume that
\begin{align} \sigma(x,E) \; \mbox{contains an} \; &(\cA,x,\epsilon,\delta_{E,z})\mbox{-basis for every} \; x \in E,  \notag{}\\
& \mbox{where} \; \delta_{E,z} := 10 \cdot \diam(E \cup \{z\}). \label{type11}
\end{align}
Then the Extension Theorem for $(E,z)$ holds.
\end{ml} 

Note that condition \eqref{type11} in the Main Lemma for $\emptyset$ holds vacuously. Thus the Main Lemma for $\cA=\emptyset$ implies the Extension Theorem for $(E,z)$, whenever $\#(E \cup\{z\}) \geq 2$. The Extension Theorem for $(E,z)$ is obvious when $\#(E \cup\{z\}) \leq 1$ (e.g., see Lemma \ref{small} in the next section). Thus we have reduced the proof of the Extension Theorem for $(E,z)$ to the task of proving the Main Lemma for $\cA$, for each $\cA \subset \cM$. 

We proceed by induction and establish the following.

\noindent \underline{Base Case:} The Main Lemma for $\cM$ holds.

\noindent \underline{Induction Step:} Let $\cA \subset \cM$ with $\cA \neq \cM$. Suppose that the Main Lemma for $\cA'$ holds for each $\cA' < \cA$. Then the Main Lemma for $\cA$ holds.

To prove the Main Lemma for $\cM$, we define $\epsilon(\cM) = 1$ and take $\alpha = \beta = 0$ in \textbf{(B2)} from the definition of $(\cM,x,1,\delta)$-basis. Thus, for $x \in E$, we have $P^x_0(x) = 1$. On the other hand, from \textbf{(B1)} we have $P^x_0(x) = 0$. This contradiction shows that $E$ is empty, hence we cannot have $\#(E \cup \{z\}) \geq 2$. Consequently, the Main Lemma for $\cM$ holds vacuously.

\subsection{Small Extension Problems}
\begin{lem}\label{small}
Suppose that $\#(E \cup \{z\}) \leq 2$. Then the Extension Theorem for $(E,z)$ holds.
\end{lem}
\begin{proof} 
If $E$ is empty then the Extension Theorem for $(E,z)$ holds with $(T,M,\Omega, \Xi)$ defined by
$$T(P) = P \; \mbox{and} \; M(P) = 0 \;\; \mbox{for each} \; P \in \cP; \;\;\; \Omega = \emptyset \; \mbox{and} \;  \Xi = \emptyset.$$

Suppose that $E$ is non-empty. Since the Extension Theorem for $(E,z)$ is equivalent to the Extension Theorem for $(E \setminus \{z\},z)$, it suffices to assume that $E = \{\hx\}$ for some $\hx \in \R^n$ and that $z \in \R^n \setminus E$. 

Fix data $f : E=\{\widehat{x}\} \rightarrow \R$ and $P \in \cP$. Define the polynomial $R \in \cP$ that satisfies\begin{equation}
R(\widehat{x}) = f(\widehat{x}); \; \mbox{and} \; \partial^\alpha R(\widehat{x}) = \partial^\alpha P(\widehat{x}) \;\; \mbox{for all} \; \alpha \in \cM \setminus \{ 0 \}. \label{uuu}
\end{equation}
(Note that $R$ is determined uniquely and linearly from $(f,P)$.) Take a cutoff function $\theta \in C^\infty(\R^n)$ that satisfies
\begin{align}
\label{p001} & \theta \; \mbox{is supported on}\; B(\widehat{x},\Delta/2), \; \mbox{where} \; \Delta := |z-\widehat{x}|/2; \\
\label{p002} & \theta \equiv 1 \; \mbox{in a neighborhood of} \; \widehat{x};\; \mbox{and}\\ 
\label{p003} & |\partial^\alpha \theta| \leq C \lvert z - \widehat{x}\rvert^{-|\alpha|} \; \mbox{when} \; |\alpha| \leq m.
\end{align}
(Note that $\theta$ necessarily vanishes in a neighborhood of the point $z$.) Define the linear map $T : L^{m,p}(E;z) \rightarrow \LR$ by
$$T(f,P) = P + \theta (R - P), \;\; \mbox{for each} \; (f,P).$$
Then \eqref{uuu}, \eqref{p001} and \eqref{p002} imply that $T$ is an extension operator, i.e., that \textbf{(E1)} holds. 

Since $T(f,P)$ agrees with an $(m-1)^\rst$ degree polynomial outside of $B(\hx,|z-\hx|/2)$, the Leibniz rule and \eqref{p003} show that
\begin{align}\|T(f,P)\|_{\LR}^p & = \|T(f,P)\|_{L^{m,p}(B(\hx,|\hx-z|/2))}^p \lesssim \sum_{|\alpha| \leq m} \lvert \partial^\alpha(P-R)(\widehat{x})\rvert^p \cdot \lvert z - \widehat{x}\rvert^{n + (|\alpha|-m)p}\notag{} \\
& \oeq{\eqref{uuu}} |P(\widehat{x}) - f(\widehat{x})|^p \cdot \lvert z - \widehat{x}\rvert^{n - mp}. \label{ruf}
\end{align}
From the Sobolev inequality, we have
\begin{align}
|P(\widehat{x}) - f(\widehat{x})|^p \cdot \lvert z - \widehat{x}\rvert^{n - mp} \leq C \cdot \|F\|^p_{\LR}, \;\; \mbox{for} \; & \mbox{every} \; F \in \LR \; \mbox{that satisfies} \notag{} \\
&F(\widehat{x}) = f(\widehat{x}) \; \mbox{and} \; J_{z}(F) = P. \label{fef1}
\end{align} 
Taking the infimum with respect to $F$ in \eqref{fef1}, yields
$$|P(\widehat{x}) - f(\widehat{x})|^p \cdot \lvert z - \widehat{x}\rvert^{n - mp} \leq C \cdot \|(f,P)\|^p_{L^{m,p}(E;z)}.$$
Since $T$ is an extension operator, the definition of the trace seminorm shows that
$$\|(f,P)\|_{L^{m,p}(E;z)} \leq \|T(f,P)\|_{\LR}.$$
The last two inequalities and \eqref{ruf} suffice to prove \textbf{(E2)}. If we define $M(f,P) = |P(\widehat{x}) - f(\widehat{x})| \cdot |z - \widehat{x}|^{n/p - m}$ then the same inequalities establish \textbf{(E3)}. Finally, \textbf{(E4)}, \textbf{(E5)} and \textbf{(E6)} follow easily by taking $\Omega = \emptyset$ and $\Xi = \bigl\{\xi : \xi(f,P) = [P(\hx) - f(\hx)] \cdot |z-\hx|^{\frac{n}{p}-m}\bigr\}$. This concludes the proof of Lemma \ref{small}.
\end{proof}

\subsection{Technical Lemmas}

In this section we present two technical lemmas used in our proof of the \underline{Induction Step}; their proofs involve nothing but the most elementary linear algebra, though they are a bit involved (the reader may wish to omit the proofs of Lemmas \ref{techlem} and \ref{weaktostrong} on first reading). Recall that $D$ stands for the dimension of the space of polynomials $\cP$.

\begin{lem}
\label{techlem}
There exist universal constants $c_1 \in (0,1)$ and $C_1 \geq 1$ so that the following holds. Suppose we are given the following data:
\begin{itemize}
\item Real numbers $\epsilon_1 \in (0,c_1]$ and $\epsilon_2 \in (0,\epsilon_1^{2D+2}]$.
\item A lengthscale $\delta>0$.
\item A collection of multi-indices $\cA \subset \cM$.
\item Two finite subsets $\emptyset \neq E_1 \subset E_2 \subset \R^n$, with $\diam(E_2) \leq 10 \delta$.
\item A family of polynomials $(\widetilde{P}_\alpha^x)_{\alpha \in \cA}$ that form an $(\cA,x,\epsilon_2,\delta)$-basis for $\sigma(x,E_1)$, for each $x \in E_2$, and satisfy
\begin{equation}
\label{p6} \max \bigl\{ \lvert \partial^\beta \widetilde{P}^{x}_\alpha(x) \rvert \delta^{|\beta|-|\alpha|}  : \; x \in E_2, \; \alpha \in \cA, \; \beta \in \cM \bigr\} \geq \epsilon_1^{-D-1}.
\end{equation}
\end{itemize}
Then there exists $\oA < \cA$ so that $\sigma(x,E_1)$ contains an $(\oA,x,C_1 \cdot \epsilon_1,\delta)$-basis for every $x \in E_2$.
\end{lem}
\begin{proof}
By rescaling it suffices to assume that $\delta = 1$. Let $c_1$ be a sufficiently small constant, to be determined later. Our hypothesis tells us that $\epsilon_1 \leq c_1$, $\epsilon_2 \leq \epsilon_1^{2D+2}$ and that $(\widetilde{P}_\alpha^x)_{\alpha \in \cA}$ forms an $(\cA,x,\epsilon_2,1)$-basis for $\sigma(x,E_1)$, for each $x \in E_2$. That is,
\begin{align}
\label{p3} & \widetilde{P}_\alpha^x \in \epsilon_2 \cdot \sigma(x,E_1);\\
\label{p4} & \partial^\beta \widetilde{P}_\alpha^x(x) = \delta_{\alpha \beta} && (\alpha,\beta \in \cA); \; \mbox{and}\\
\label{p5} & |\partial^\beta \widetilde{P}_\alpha^x(x)| \leq \epsilon_2 && (\alpha \in \cA, \;\beta \in \cM, \; \beta > \alpha).
\end{align}
For each $\alpha \in \cA$, we define $Z_\alpha = \max \left\{ |\partial^\beta \widetilde{P}^{x}_\alpha(x)|: \; x \in E_2, \; \beta \in \cM \right\}$. Then hypothesis \eqref{p6} is equivalent to $\max\{Z_\alpha : \alpha \in \cA\} \geq \epsilon_1^{-D-1}$. Let $\oa \in \cA$ be the minimal index with $Z_{\oa} \geq \epsilon_1^{-D-1}$. Thus,
\begin{equation}
\label{stuff1} Z_\alpha < \epsilon_1^{-D-1}, \;\; \mbox{for all} \; \alpha \in \cA \; \mbox{with} \; \alpha < \oa,
\end{equation}
and there exist $x_0 \in E_2$ and $\beta_0 \in \cM$ with
\begin{equation}
\label{p6c} \epsilon_1^{-D-1} \leq Z_{\oa} = |\partial^{\beta_0} \widetilde{P}^{x_0}_{\oa}(x_0)|.
\end{equation}
Thus, \eqref{p4} and \eqref{p5} imply that $\beta_0 \neq \oa$ and $\beta_0 \leq \oa$, respectively. Therefore, $\beta_0 <\oa$. By definition of $Z_{\oa}$, we also have
\begin{equation}
\label{p6a} |\partial^\beta \widetilde{P}^{y}_{\oa}(y)| \leq |\partial^{\beta_0} \widetilde{P}^{x_0}_{\oa}(x_0)|, \;\; \mbox{for all} \; y \in E_2 \; \mbox{and} \; \beta \in \cM.
\end{equation}

Let the elements of $\cM$ between $\beta_0$ and $\oa$ be ordered as follows:
$$\beta_0 < \beta_1 < \cdots < \beta_k = \oa.$$
Note that $k + 1 \leq \# \cM = D$. Define 
$$a_i = |\partial^{\beta_i} \widetilde{P}^{x_0}_\oa(x_0)|,\;\; \mbox{for} \; i=0,\ldots,k.$$ 
Then, \eqref{p4} and \eqref{p6c} imply that $a_k = 1$ and  $a_0 \geq \epsilon_1^{-D-1}$. Choose $r \in \{0,\ldots,k\}$ with $a_r \epsilon_1^{-r} = \max\{a_l \epsilon_1^{-l} : 0 \leq l \leq k\}$. Note that $a_0 \geq \epsilon_1^{-D-1} > a_k \epsilon_1^{-k}$ which implies $r \neq k$. Moreover, we have
\begin{equation}\label{stuffs} a_r \geq \epsilon_1^D a_0 \;\; \mbox{and} \; a_r \geq \epsilon_1^{-1} a_i\; \mbox{for} \; i=r+1,\ldots,k.
\end{equation}

Define $\ob = \beta_r \in \cM$. Then, \eqref{stuffs} states that
\begin{equation}
\label{p7} |\partial^{\ob} \widetilde{P}^{x_0}_{\oa}(x_0)| \geq \epsilon_1^D |\partial^{\beta_0} \widetilde{P}^{x_0}_{\oa}(x_0)| \ogeq{\eqref{p6c}} \epsilon_1^{-1}.
\end{equation}
Also, \eqref{p5} and \eqref{p7} imply that 
$$|\partial^\beta \widetilde{P}^{x_0}_{\oa}(x_0)| \leq \epsilon_2 \leq 1 \leq \epsilon_1 |\partial^{\ob} \widetilde{P}^{x_0}_{\oa}(x_0)| \qquad (\beta \in \cM, \;\beta > \oa).$$ 
Meanwhile, for $\ob < \beta \leq \oa$, \eqref{stuffs} implies that $|\partial^\beta \widetilde{P}^{x_0}_{\oa}(x_0)| \leq \epsilon_1 |\partial^{\ob} \widetilde{P}^{x_0}_{\oa}(x_0)|$. Thus, we have
\begin{equation}
\label{p8} |\partial^\beta \widetilde{P}^{x_0}_{\oa}(x_0)| \leq \epsilon_1 |\partial^{\ob} \widetilde{P}_{\oa}(x_0)| \qquad (\beta \in \cM, \; \beta > \ob).
\end{equation}

By \eqref{p7} we have $|\partial^{\ob} \widetilde{P}^{x_0}_{\oa}(x_0)|> 1$. Hence, \eqref{p4} and \eqref{p5} imply that
\begin{equation}
\label{p9} \ob < \oa \; \mbox{and} \; \ob \notin \cA.
\end{equation}
Define $\widetilde{P}^{x_0}_{\ob} = \widetilde{P}^{x_0}_{\oa} / \partial^{\ob} \widetilde{P}^{x_0}_{\oa}(x_0)$, which satisfies
\begin{align}
\label{p10} & \widetilde{P}^{x_0}_{\ob} \in \epsilon_2 \cdot \sigma(x_0,E_1), \qquad \mbox{thanks to} \; \eqref{p3} \; \mbox{and} \; |\partial^{\ob} \widetilde{P}^{x_0}_{\oa}(x_0)|> 1;\\
\label{p13} & |\partial^\beta \widetilde{P}^{x_0}_{\ob}(x_0)| \leq \epsilon_1 \qquad (\beta \in \cM, \; \beta > \ob), \qquad \mbox{thanks to} \; \eqref{p8}; \\
\label{p14} & |\partial^\beta \widetilde{P}^{x_0}_{\ob}(x_0)| \leq \epsilon_1^{-D} \quad\; (\beta \in \cM), \qquad \mbox{thanks to} \; \eqref{p6a} \; \mbox{and}\; \eqref{p7}; \; \mbox{and} \\
\label{p14a} & \partial^{\ob} \widetilde{P}^{x_0}_{\ob}(x_0) = 1.
\end{align}

By \eqref{p10}, there exists $\varphi_{\ob} \in \LR$ with 
\begin{align}
\label{p901} &J_{x_0} \varphi_{\ob} = \widetilde{P}^{x_0}_{\ob}; \\
\label{p902} &\varphi_{\ob}|_{E_1} = 0; \; \mbox{and} \\
\label{p903} &\|\varphi_{\ob}\|_{\LR} \leq \epsilon_2.
\end{align}
Fix an arbitrary point $y \in E_2$ and define $\widetilde{P}^y_{\ob} := J_y \varphi_{\ob}$. Then \eqref{p902}-\eqref{p903} and the definition of $\sigma(\cdot,\cdot)$ imply that
\begin{equation}
\label{p15} \widetilde{P}^y_\ob \in \epsilon_2 \cdot \sigma(y,E_1).
\end{equation}
Since $\widetilde{P}^y_\ob = J_y \varphi_\ob$, we have
\begin{align}\label{p16pre}
|\partial^\beta \widetilde{P}^y_\ob(y) &- \delta_{\beta \ob}| \leq  |\partial^\beta \varphi_\ob(y) - \partial^\beta \widetilde{P}^{x_0}_{\ob}(y)| \\
& + \left|\sum_{|\gamma| \leq m-1-|\beta|} \frac{1}{\gamma!} \partial^{\beta+\gamma} \widetilde{P}^{x_0}_\ob(x_0) \cdot (y-x_0)^\gamma -  \delta_{\beta \ob} \right| \qquad (\beta \in \cM).\notag{} 
\end{align}
Since $x_0,y \in E_2$ and $\diam(E_2) \leq 10\delta = 10$ we have $|x_0-y|\leq 10$. Therefore, the first term on the right-hand side of \eqref{p16pre} is bounded by $C \|\varphi_\ob\|_{\LR}$; this uses \eqref{p901} and the  Sobolev inequality (\ref{sec_not}.\ref{SET}) . In turn this is bounded by $C \epsilon_2$, thanks to \eqref{p903}. If $\beta > \ob$ then $\beta + \gamma > \ob$ for each multi-index $\gamma$ with $|\gamma| \leq m-1-|\beta|$. In this case \eqref{p13} and $|x_0-y| \leq 10$ imply that the second term on the right-hand side of \eqref{p16pre} is controlled by $C \epsilon_1$. If $\beta = \ob$ then \eqref{p14a} implies that the second term from \eqref{p16pre} equals
$$\left|\sum_{0 < |\gamma| \leq m-1-|\beta|} \frac{1}{\gamma!} \partial^{\beta+\gamma} \widetilde{P}^{x_0}_\ob(x_0) \cdot (y-x_0)^\gamma \right|,$$
which is bounded by $C \epsilon_1$, again by \eqref{p13} and $|x_0 - y| \leq 10$. Thus we have argued that
\begin{equation}
\label{p16} |\partial^\beta \widetilde{P}^y_\ob(y) - \delta_{\beta \ob}| \leq C \epsilon_1 \qquad (\beta \in \cM,\; \beta \geq \ob).
\end{equation}
Finally, if $\beta \in \cM$ then \eqref{p14} and $|x_0-y| \leq 10$ imply that the second term on the right-hand side of \eqref{p16pre} is controlled by $C \epsilon_1^{-D}$. Therefore,
\begin{align}
|\partial^\beta \widetilde{P}^y_\ob(y)| &\leq |\partial^\beta \widetilde{P}^y_\ob(y) - \delta_{\beta\ob}| + 1  \notag{}\\
& \leq C \epsilon_1 + C \epsilon_1^{-D} + 1 \leq C' \epsilon_1^{-D} \qquad ( \beta \in \cM). \label{p17}
\end{align}

Define $\oA = \{\alpha \in \cA : \alpha < \ob\} \cup \{\ob\}$. Then \eqref{p9} implies that the minimal element of $\cA \Delta \oA$ is $\ob$. Thus, we have $\overline{\cA} < \cA$. For each $\alpha \in \oA \setminus \{\ob\}$, we have $\alpha < \ob$; hence, \eqref{p9} implies that $\alpha < \oa$. Thus, \eqref{stuff1} and the definition of $Z_\alpha$ imply that
\begin{align}
& \label{p20a} |\partial^\beta \widetilde{P}^y_\alpha(y)| \leq \epsilon_1^{-D-1} \qquad (\alpha \in \oA \setminus \{\ob\},\; \beta \in \cM).
\end{align}

Define
$$P^y_\ob := \widetilde{P}^y_\ob - \sum_{\gamma \in \oA \setminus \{\ob\}} \partial^\gamma \widetilde{P}^y_{\ob}(y) \widetilde{P}^y_\gamma.$$
Notice that $\oA \setminus \{\ob\} \subset \cA$; thus \eqref{p4} implies that
$$\partial^\alpha P^y_\ob(y) = \partial^\alpha \widetilde{P}^y_\ob(y) - \sum_{\gamma \in \oA \setminus \{\ob\}} \partial^\gamma \widetilde{P}^y_\ob(y) \delta_{\alpha \gamma} = 0 \qquad (\alpha \in \oA \setminus \{\ob\}).$$
Also, \eqref{p15},\eqref{p17} and \eqref{p3} imply that
$$P^y_\ob \in (\epsilon_2 + C\epsilon_1^{-D} \epsilon_2) \cdot \sigma(y,E_1) \subseteq C \epsilon_1^{-D} \epsilon_2 \cdot \sigma(y,E_1).$$
Since $\ob$ is the maximal element of $\oA$, it follows that for any $\beta \geq \ob$ and any $\gamma \in \oA \setminus \{\ob\}$, we have $\beta > \gamma$. Thus, \eqref{p16}, \eqref{p17} and \eqref{p5} imply that
\begin{equation}
\label{eq21a}
|\partial^\beta P^y_\ob(y) - \delta_{\beta \ob}| \leq C (\epsilon_1 + \epsilon_1^{-D} \epsilon_2) \qquad( \beta \in \cM, \;\beta \geq \ob).
\end{equation}
Finally, \eqref{p17} and \eqref{p20a} imply that
$$|\partial^\beta P^y_\ob(y)| \leq C\epsilon_1^{-2D-1} \qquad (\beta \in\cM).$$

Recall that $\epsilon_2 \leq \epsilon_1^{D+1}$ and $\epsilon_1 \leq c_1$. We now fix $c_1$ to be a small universal constant, so that \eqref{eq21a} yields $\partial^\ob P^y_\ob(y) \in [1/2,2]$. We then define $\hP^y_{\ob} = P^y_\ob / \partial^\ob P^y_\ob(y)$. The above four lines give that
\begin{align}
\label{p22} & \hP^y_\ob \in C\epsilon_1^{-D} \epsilon_2 \cdot \sigma(y,E_1); \\
\label{p23} & \partial^\beta (\hP^y_\ob)(y) = \delta_{\beta \ob} \qquad \; (\beta \in \oA);\\
\label{p24} & |\partial^\beta \hP^y_\ob(y)| \leq C(\epsilon_1 + \epsilon_1^{-D} \epsilon_2) \qquad ( \beta \in \cM, \; \beta > \ob); \; \mbox{and}\\
\label{p24a} & |\partial^\beta \hP^y_\ob(y)| \leq C \epsilon_1^{-2D-1} \qquad (\beta \in \cM).
\end{align}

For each $\alpha \in \oA \setminus \{\ob\}$ we define $\hP^y_\alpha = \widetilde{P}^y_\alpha - \left[\partial^{\ob} (\widetilde{P}^y_\alpha)(y)\right] \hP^y_\ob$. Note that $\alpha < \ob$, and hence $|\partial^{\ob} (\widetilde{P}^y_\alpha)(y) | \leq \epsilon_2 \leq 1$, thanks to \eqref{p5}. From \eqref{p3} and \eqref{p22}, we now obtain
\begin{equation}
\label{p26} \hP^y_\alpha \in C' \epsilon_1^{-D}\epsilon_2 \cdot \sigma(y,E_1).
\end{equation} 
From \eqref{p5} and \eqref{p24a}, we have
\begin{align}
 \left| \partial^\beta (\hP^y_\alpha)(y) \right| &\leq \left| \partial^\beta (\widetilde{P}^y_\alpha)(y)\right| + \left|\partial^{\ob} (\widetilde{P}^y_\alpha)(y)\right|\cdot \left| \partial^\beta (\hP^y_\ob)(y) \right| \leq \epsilon_2  + \epsilon_2 \cdot C \epsilon_1^{-2D-1} \notag{}\\
& \leq C' \epsilon_1^{-2D-1} \epsilon_2 \qquad ( \beta \in \cM, \; \beta > \alpha). \label{p28}
\end{align}
Recall that $\oA = \{\alpha \in \cA : \alpha < \ob\} \cup \{\ob\}$. From \eqref{p4} and \eqref{p23}, we have
\begin{align}
\partial^{\beta} \hP^y_\alpha(y) &= \partial^\beta (\widetilde{P}^y_\alpha)(y) - \left[\partial^{\ob} (\widetilde{P}^y_\alpha)(y)\right] \partial^\beta (\hP^y_\ob)(y)  \notag{}\\
&= \left\{
\begin{array}{ll}
 \delta_{\alpha \beta} - \left[\partial^{\ob}(\widetilde{P}^y_\alpha)(y) \right] \delta_{\beta \ob} \quad\quad\;\;\;= \delta_{\alpha \beta}& \mbox{for}\; \beta \in \cA, \; \beta < \ob \\
\partial^{\ob} (\widetilde{P}^y_\alpha)(y) - \left[\partial^{\ob} (\widetilde{P}^y_\alpha)(y)\right] \delta_{\ob \;\ob} = 0 & \mbox{for} \; \beta = \ob\\
\end{array}
\right. \notag{} \\
& = \delta_{\alpha \beta} \qquad (\beta \in \oA). \label{p27} 
\end{align}

By now varying the point $y \in E_2$, we deduce from \eqref{p22}-\eqref{p24} and \eqref{p26}-\eqref{p27} that $\sigma(y,E_1)$ contains an $(\oA,y,C \cdot [\epsilon_1 + \epsilon_1^{-2D-1} \epsilon_2],1)$-basis for each $y \in E_2$. Since $\epsilon_2 \leq \epsilon_1^{2D+2}$, the conclusion of Lemma \ref{techlem} is immediate.
\end{proof}

\begin{defn}[Near-triangular matrices]
Let $S \geq 1$, $\epsilon \in (0,1)$ and $\cA \subset \cM$ be given. A matrix $B = (B_{\alpha\beta})_{\alpha,\beta \in \cA}$ is called $(S, \epsilon)$ \textit{near-triangular} if 
\begin{align*}
& \lvert B_{\alpha\beta} - \delta_{\alpha \beta}\rvert \leq \epsilon \quad (\alpha, \beta \in \cA, \; \alpha \leq \beta); \; \mbox{and}\\ 
& \lvert B_{\alpha \beta}\rvert \leq S \qquad\quad \;\;(\alpha,\beta \in \cA).
\end{align*}
\end{defn}

\begin{lem}
\label{weaktostrong}
Given $R \geq 1$, there exist constants $c_2>0$, $C_2 \geq 1$ depending only on $R$, $m$, $n$, so that the following holds. Suppose we are given $\epsilon_2 \in (0,c_2]$, $x \in \R^n$, a symmetric convex subset $\sigma \subset \cP$ and a family of polynomials $(P_\alpha)_{\alpha \in \cA} \subset \cP$, such that
\begin{subequations}
\begin{align}
\label{trans1} & \;\;\;\; P_\alpha \in \epsilon_2 \cdot \sigma \qquad\qquad\quad (\alpha \in \cA); \\
\label{trans2} &\lvert \partial^\beta P_\alpha(x) - \delta_{\alpha \beta}\rvert \leq \epsilon_2  \quad\;\;\;  (\alpha \in \cA, \beta \in \cM, \beta \geq \alpha); \; \mbox{and}\\
\label{trans3} & \;\;\;\; \lvert \partial^\beta P_\alpha(x) \rvert \leq R \qquad\quad\;\; ( \alpha \in \cA, \beta \in \cM).
\end{align}
\end{subequations}
Then there exists a $(C_2,C_2 \epsilon_2)$ near-triangular matrix $B=(B_{\alpha \beta})_{\alpha,\beta \in \cA}$, so that
\begin{equation}
\label{low} \widehat{P}_\alpha := \sum_{\beta \in \cA} B_{\alpha \beta} P_\beta \qquad (\alpha \in \cA)
\end{equation}
forms an $(\cA,x,C_2 \epsilon_2,1)$-basis for $\sigma$. Furthermore, $\lvert\partial^\beta \widehat{P}_\alpha(x)\rvert \leq C_2$ for every $\alpha \in \cA$ and every $\beta \in \cM$.
\end{lem}
\begin{proof}
Let $c_2 \in (0,1)$ and $C_2 \geq 1$ be constants depending only on $R$, $m$ and $n$, that will be determined later. Define the matrix $D$ by setting $D_{\alpha \beta} = \partial^\beta P_\alpha(x)$ for all $\alpha,\beta \in \cA$. From \eqref{trans2} and \eqref{trans3} it is immediate that $D$ is $(R,\epsilon_2)$ near-triangular. Since $\epsilon_2 \leq c_2$, by fixing a small enough constant $c_2$ determined by $R$, $m$, $n$, we can ensure that $D$ is invertible and that $B := D^{-1} = (B_{\alpha\beta})_{\alpha,\beta \in \cA}$ is $(C',C'\epsilon_2)$ near-triangular, for some constant $C'$ determined by $R$, $m$, $n$. 

For each $\alpha \in \cA$ we define $\widehat{P}_\alpha$ as in \eqref{low}. Since $\sigma$ is a symmetric convex set, from $|B_{\alpha \beta}| \leq C'$ and \eqref{trans1} we deduce that
$$\widehat{P}_\alpha \in C'' \epsilon_2 \cdot \sigma, \;\;\; \mbox{where} \; C'' = C''(R,m,n).$$
Since $B = D^{-1}$, we have
$$\partial^\gamma \widehat{P}_\alpha(x) = \sum_{\beta \in \cA} B_{\alpha \beta} D_{\beta \gamma} = \delta_{\alpha \gamma} \qquad ( \gamma \in \cA).$$
Finally, for each $\gamma \in \cM$ with $\gamma > \alpha$, we write
$$|\partial^\gamma \widehat{P}_\alpha(x)| \leq \Biggl\lvert \sum_{\substack{\beta \in \cA\\ \beta \leq \alpha}} B_{\alpha \beta} \partial^\gamma P_\beta(x) \Biggr\rvert + \Biggl\lvert \sum_{\substack{\beta \in \cA\\ \beta > \alpha}} B_{\alpha \beta} \partial^\gamma P_\beta(x)\Biggr\rvert.$$ 
From \eqref{trans2}, \eqref{trans3} and the fact that $B$ is $(C',C'\epsilon_2)$ near-triangular, it follows that each summand is dominated by $C' \cdot R \cdot \epsilon_2$. Hence,
$$|\partial^\gamma \widehat{P}_\alpha(x)| \leq \hat{C} \epsilon_2 \quad  (\gamma \in \cM, \; \gamma> \alpha), \quad \mbox{where} \; \hat{C} = \hat{C}(R,m,n).$$
Taking a large enough constant $C_2$ determined by $R$, $m$, $n$, we have shown that $(\widehat{P}_\alpha)_{\alpha \in \cA}$ forms an $(\cA,x,C_2 \epsilon_2,1)$-basis for $\sigma$. Finally, \eqref{trans3} and $|B_{\alpha \beta}| \leq C'$ imply the last conclusion of the lemma.
\end{proof}

\section{The Inductive Hypothesis}\label{sec_ind}
\setcounter{equation}{0}

Let $\cA \subsetneq \cM$. Here, we start on the proof of the \underline{Induction Step}. We make the inductive assumption that the Main Lemma for $\cA'$ holds for each $\cA' < \cA$. Put
$$\epsilon_0 = \min \bigl\{ \epsilon(\cA') : \cA' < \cA\bigr\},$$
with $\epsilon(\cA')$ as in the Main Lemma for $\cA'$. From the remark in section \ref{poly_basis}, we have:
\begin{itemize}
\item[\textbf{(IH)}] Let $\widehat{\cA} < \cA$. Let $\hE \subset \R^n$ and $\hz \in \R^n$ satisfy $2 \leq \#(\hE \cup \{\hz\}) < \infty$, and assume that
\begin{align} \label{eq_600} \sigma(x,\hE) \; \mbox{contains an} &\; (\widehat{\cA},x,\epsilon_0,\delta_{\hE,\hz})\mbox{-basis for every} \; x \in \hE, \\
& \mbox{where} \; \delta_{\hE,\hz} = 10 \cdot \diam(\hE \cup \{\hz\}). \notag{}
\end{align}
Then the Extension Theorem for $(\hE,\hz)$ holds.
\end{itemize}

Let us start on the proof of the Main Lemma for $\cA$. The value of the universal constant $ \epsilon(\cA)$ is determined later in the paper. We now assume that $\epsilon = \epsilon(\cA)$  is less than a small enough universal constant (the ``small $\epsilon$ assumption'').

Fix $E \subset \R^n$ and $z \in \R^n$ with $2 \leq \#(E \cup \{z\}) <\infty$, and such that condition (\ref{sec_plan}.\ref{type11}) from the Main Lemma for $\cA$ holds. By rescaling and translating $E$ and $\{z\}$, we can arrange that
\begin{align} &\delta_{E,z} := 10 \cdot \diam(E \cup \{z\}) = 1 \; \mbox{and}\notag{}\\
&E \cup \{z\} \subset \frac{1}{8}Q^\circ \; \mbox{where} \; Q^\circ := (0,1]^n \;. \label{cube} 
\end{align}
(Note that $Q^\circ$ is a dyadic cube as per our notation.)

Suppose that there exists $\oA < \cA$ such that $\sigma(x,E)$ contains an $(\oA,x,\epsilon_0,1)$-basis for every $x \in E$. Then the Extension Theorem for $(E,z)$ holds in view of \textbf{(IH)}. Having reached the conclusion of the Main Lemma for $\cA$ in this case, we now assume that
\begin{align}
\mbox{for every} \; \oA < \cA, &\; \mbox{there exists} \; x \in E, \notag{}\\
&\;\mbox{such that} \; \sigma(x,E) \; \mbox{\underline{does not} contain an} \;  (\oA,x,\epsilon_0,1)\mbox{-basis}.\label{d1}
\end{align}

If $\#(E) \leq 1$ then the Extension Theorem for $(E,z)$ holds (see Lemma \ref{small}). Thus, we may assume that
\begin{equation} \label{atleasttwo}
\#(E) \geq 2.
\end{equation}

The assumptions \eqref{cube}-\eqref{atleasttwo} on $(E,z)$ will remain in place until the end of section \ref{sec_pf}, when we prove Theorem \ref{thm1} for finite $E$, and establish Theorems \ref{thm2}, \ref{thm3}.

\subsection{Auxiliary Polynomials} \label{aux}
Placing $\delta_{E;z} = 1$ into condition (\ref{sec_plan}.\ref{type11}) from the Main Lemma for $\cA$, we obtain $(\widetilde{P}^x_\alpha)_{\alpha \in \cA}$ that forms an $(\cA,x,\epsilon,1)$-basis for $\sigma(x,E)$, for every $x \in E$. The goal of this subsection is to exhibit a similar basis for $\sigma(x,E)$ at every other point $x \in Q^\circ$. 

As a consequence of \eqref{d1}, we will show that
\begin{equation} \label{a3}
\lvert \partial^\beta \widetilde{P}_\alpha^x(x)\rvert \leq C \qquad (x \in E, \; \alpha \in \cA, \; \beta \in \cM).
\end{equation}
To prove \eqref{a3}, we fix a universal constant $\epsilon_1 \leq \min \{c_1, \epsilon_0/C_1\}$, where $c_1$ and $C_1$ are the constants from Lemma \ref{techlem}. For the sake of contradiction, suppose that \eqref{a3} fails to hold with $C = \epsilon_1^{-D-1}$:
\begin{equation} \label{cont2}\max \bigl\{ \lvert \partial^\beta \widetilde{P}_\alpha^x(x) \rvert : x \in E, \alpha \in \cA, \beta \in \cM \bigr\} > \epsilon_1^{-D-1}.\end{equation}
We may assume that $\epsilon \leq \epsilon_1^{2D+2}$. Then the hypotheses of Lemma \ref{techlem} hold with parameters
$$\left(\epsilon_1,\epsilon_2,\delta,\cA,E_1,E_2,(\widetilde{P}^x_\alpha)_{\alpha \in \cA,x \in E_2} \right)  := \left(\epsilon_1,\epsilon,1,\cA,E,E,(\widetilde{P}^x_\alpha)_{\alpha \in \cA,x \in E} \right).$$
(Here, the left-hand side denotes the local parameters of Lemma \ref{techlem}, and the right-hand side is defined as in the paragraph above. Notice that $\diam(E) \leq 10$  and $E \neq \emptyset$ follow from \eqref{cube} and \eqref{atleasttwo}, respectively.) Thus we find $\oA < \cA$ so that $\sigma(x,E)$ contains an $(\oA,x,C_1 \epsilon_1,1)$-basis for each $x \in E$. Since $C_1 \epsilon_1 \leq \epsilon_0$, this contradicts \eqref{d1}, which concludes our proof of \eqref{a3}.

Fix $x_0 \in E$ and $\alpha \in \cA$. Since $\widetilde{P}^{x_0}_\alpha \in \epsilon \cdot \sigma(x_0,E)$, there exists $\varphi_\alpha \in \LR$ with
\begin{align}
\label{p101} & \varphi_\alpha = 0 \; \mbox{on} \; E; \\
\label{p102} & J_{x_0} (\varphi_\alpha) = \widetilde{P}_\alpha^{x_0}; \; \mbox{and} \\
\label{p103} & \|\varphi_\alpha\|_{\LR} \leq \epsilon.
\end{align}
In view of \eqref{a3} and \textbf{(B2)},\textbf{(B3)} from the definition of $(\widetilde{P}_\alpha^{x_0})_{\alpha \in \cA}$ being an $(\cA,x_0,\epsilon,1)$-basis, we have
\begin{align}
&\; \lvert \partial^{\beta} \varphi_\alpha(x_0) \rvert = \lvert \partial^{\beta} \widetilde{P}^{x_0}_\alpha(x_0) \rvert \leq C \qquad\;\; \mbox{for} \; \beta \in \cM; \label{a30}\\
&\; \; \partial^\alpha \varphi_\alpha(x_0) = \partial^\alpha \widetilde{P}^{x_0}_\alpha(x_0) = 1;  \; \mbox{and}\label{p104} \\
&\;\lvert \partial^{\beta+\gamma} \varphi_\alpha(x_0) \rvert = \lvert \partial^{\beta+\gamma} \widetilde{P}^{x_0}_\alpha(x_0) \rvert \leq \epsilon \;\;\; \mbox{for} \; \beta > \alpha, \; |\gamma| \leq m-1-|\beta|, \; \mbox{and} \notag{} \\
& \;\qquad\qquad\qquad\qquad\qquad\qquad\qquad\quad \mbox{for} \; \beta = \alpha, \; 0 < |\gamma| \leq m-1-|\beta| \label{p105}\\
& \;\; (\mbox{since in either case we have } \; \beta + \gamma > \alpha). \notag{}
\end{align}

For each $x \in Q^\circ$, define $\hP_\alpha^x = J_x(\varphi_\alpha)$. Then \eqref{p101} and \eqref{p103} yield
\begin{equation} \label{t_1} \hP^x_\alpha \in \epsilon \cdot \sigma(x,E).
\end{equation}
For each $\beta \in \cM$, by the Sobolev inequality (notice that $|x_0 - x| \leq \diam(Q^\circ) = 1$), we have
\begin{align*}
\left| \partial^\beta \hP^x_\alpha(x) - \delta_{\alpha \beta} \right| &= \left|\partial^\beta \varphi_\alpha(x) - \delta_{\alpha \beta}\right| \leq \left|\partial^\beta J_{x_0}(\varphi_\alpha)(x) - \delta_{\alpha \beta}\right| + \left| \partial^\beta \left[J_{x_0} (\varphi_\alpha) - \varphi_\alpha \right] (x) \right| \\
& \lesssim \Bigl| \sum_{|\gamma| \leq m- 1 -|\beta|} \frac{1}{\gamma!}\partial^{\beta+\gamma}\varphi_\alpha(x_0)(x-x_0)^\gamma - \delta_{\alpha \beta} \Bigr| + \|\varphi_\alpha\|_{\LR}. 
\end{align*}
Thus, \eqref{p103},\eqref{p104},\eqref{p105} show that
\begin{equation} \label{t_2}\lvert\partial^\beta \hP^x_\alpha(x)- \delta_{\alpha \beta}\rvert \leq \hat{C} \epsilon \qquad (\beta \in \cM, \beta \geq \alpha),
\end{equation}
while \eqref{p103},\eqref{a30} imply that
\begin{align} \label{t_3}
\lvert\partial^\beta \hP^x_\alpha(x)\rvert \leq \lvert\partial^\beta \hP^x_\alpha(x) - \delta_{\alpha \beta}\rvert + 1 \leq \hat{C} + \epsilon + 1 \qquad (\beta \in \cM).
\end{align}

It follows from \eqref{t_1},\eqref{t_2},\eqref{t_3} that (\ref{sec_plan}.\ref{trans1}),(\ref{sec_plan}.\ref{trans2}),(\ref{sec_plan}.\ref{trans3}) hold for $(\hP^x_\alpha)_{\alpha \in \cA}$, with $R$ equal to a universal constant, $\epsilon_2 = \hat{C} \epsilon$ and $\sigma = \sigma(x,E)$. We may assume that $\hat{C} \epsilon \leq c_2$, with the constant $c_2$ as in Lemma \ref{weaktostrong}. We have verified the hypotheses of Lemma \ref{weaktostrong}; hence for each $x \in Q^\circ$ there exists a $(C',C'\epsilon)$ near-triangular matrix $A^x = (A^x_{\alpha \beta})_{\alpha,\beta \in \cA}$, such that
\begin{align} 
& \notag{}P^x_\alpha :=  \sum_{\beta \in \cA} A^x_{\alpha \beta} \cdot J_x(\varphi_{\beta}) \; \mbox{satisfies:}\\
& (P^x_\alpha)_{\alpha \in \cA} \; \mbox{is an} \; (\cA,x,C'\epsilon,1)\mbox{-basis for} \; \sigma(x,E), \; \mbox{and} \label{mainprops}\\
&\qquad\qquad \lvert\partial^\beta P^x_\alpha(x)\rvert \leq C', \;\; \mbox{for every} \; \beta \in \cM. \notag{}
\end{align}
That is,
\begin{subequations}
\begin{align} \label{mainprops1}
& P^x_\alpha \in C'\epsilon \cdot \sigma(x,E) &&(x \in Q^\circ,\; \alpha \in \cA);\\
& \partial^\beta P^x_\alpha(x) = \delta_{\alpha \beta} && (x\in Q^\circ, \;\alpha,\beta \in \cA);\label{mainprops2}\\
& \lvert\partial^\beta P^x_\alpha(x)\rvert \leq C'\epsilon && (x \in Q^\circ,\; \alpha \in \cA, \; \beta \in \cM, \; \beta>\alpha); \; \mbox{and}\label{mainprops3} \\
& \lvert\partial^\beta P^x_\alpha(x)\rvert \leq C' &&(x \in Q^\circ,\; \alpha \in \cA, \;\beta \in \cM). \label{mainprops4}
\end{align}
\end{subequations}
Here and elsewhere, $\partial^\beta P^x_\alpha(x)$ denotes the value of $\partial^\beta_y P^x_\alpha(y)$ at $y=x$, not the $\beta^\th$ derivative of the function $x \mapsto P^x_\alpha(x)$.

\subsection{Reduction to Monotonic $\cA$}

The notion of monotonic labels played a key r\^{o}le in the study of $C^m$ extension problems. It continues to be crucial for us here.
\begin{defn}[Monotonic labels]\label{mon}
A collection of multi-indices $\cA \subset \cM$ is monotonic if 
$$\alpha \in \cA \;\; \mbox{and} \;\; \lvert\gamma\rvert \leq m-1- \lvert\alpha\rvert \; \implies \; \alpha + \gamma \in \cA.$$
If the above property fails, we say that $\cA$ is non-monotonic.
\end{defn}

In this section we deduce the monotonicity of $\cA$ using assumption \eqref{d1} and condition (\ref{sec_plan}.\ref{type11}) from the Main Lemma for $\cA$.

We proceed by contradiction and assume that $\cA$ is non-monotonic. In this setting, we prove
\begin{equation*} \circledast \;\; \mbox{There exists} \; \oA < \cA \; \mbox{so that} \; \sigma(x,E) \; \mbox{contains an}  \; (\oA,x,\epsilon_0,1)\mbox{-basis for each} \; x \in E. \end{equation*} 
This contradicts \eqref{d1}; thus, our proof of the \underline{Induction Step} will be reduced to the case of monotonic $\cA$.

In order to construct $\oA$, we exploit the non-monotonicity of $\cA$ and choose $\alpha_0 \in \cA$ and $\gamma \in \cM$ with 
\begin{equation} \label{nonmon} 
0 < \lvert\gamma\rvert \leq m-1-\lvert\alpha_0\rvert \;\; \mbox{and} \;\; \oa := \alpha_0 + \gamma \in \cM \setminus \cA.
\end{equation}
We then define $\oA = \cA \cup \{\oa\}$. Note that the minimal (only) element of $\oA \Delta \cA$ is $\oa$, which is a member of $\oA$. Hence, $\oA < \cA$ by definition of the order. Also, note that $\alpha_0 < \oa$.

For each $y \in Q^\circ$, we have defined $(P^y_\alpha)_{\alpha \in \cA}$ that satisfy (\ref{mainprops})$-$(\ref{mainprops4}). For each $y \in E$, we now define
\begin{equation} \label{eq_601}P^y_{\overline{\alpha}} = P^y_{\alpha_0} \odot_{y} q^y,\; \; \mbox{where} \;\; q^y(x):=\frac{\alpha_0!}{\overline{\alpha}!}(x-y)^\gamma.
\end{equation} 
Expanding out this product, we have
\begin{equation*}
P^y_{\oa}(x) = \frac{\alpha_0!}{\overline{\alpha}!} \sum_{|\omega| \leq m - 1 -|\gamma|} \frac{1}{\omega!} \partial^\omega P^y_{\alpha_0}(y) (x-y)^{\omega + \gamma}.
\end{equation*}
Note that $\omega=\alpha_0$ arises in the sum above, thanks to \eqref{nonmon}. Also, the terms with $\omega + \gamma > \oa = \alpha_0 + \gamma$ correspond precisely to $\omega >\alpha_0$, by definition of the order. The following properties are now immediate. \begin{subequations}
\begin{align}
\label{dia1}
&\partial^{\overline{\alpha}} P^y_{\overline{\alpha}}(y) = 1, \;\;\; \mbox{thanks to} \; (\ref{mainprops2}) \; \mbox{with} \; \alpha=\beta=\alpha_0.\\
\label{dia2}
&\lvert\partial^\beta P^y_{\overline{\alpha}}(y)\rvert \leq C' \epsilon \quad (\beta \in \cM, \beta > \overline{\alpha}), \;\; \mbox{thanks to} \; (\ref{mainprops3}) \; \mbox{with} \; \alpha = \alpha_0. \\
\label{dia3}
&\lvert\partial^\beta P^y_{\oa}(y)\rvert \leq C' \quad\;\; (\beta \in \cM), \;\;\; \mbox{thanks to} \;(\ref{mainprops4}).
\end{align}

Next, we establish that
\begin{equation}
\label{i204}
P^y_{\overline{\alpha}} \in C' \epsilon \cdot \sigma(y,E).
\end{equation}
\end{subequations}
From (\ref{mainprops1}) (with $\alpha = \alpha_0$), we find a function $\varphi \in \LR$ with 
\begin{align}
\label{p202} &\varphi = 0 \; \mbox{on}\; E; \\
\label{p203} &J_{y} (\varphi) = P^y_{\alpha_0}; \;\mbox{and}\\
\label{p201} &\|\varphi\|_{\LR} \leq C\epsilon.
\end{align}

Consider $x \in Q^\circ$ and $\beta \in \cM$ with $\beta > \alpha_0$. Since $x$ and $y$ belong to the unit cube $Q^\circ$, we have $|x-y| \leq \delta_{Q^\circ} = 1$. Thus, \eqref{p203} and the Sobolev inequality yield
\begin{align}
\notag{}
\lvert \partial^\beta \varphi(x) \rvert \leq \lvert \partial^\beta P^y_{\alpha_0}(x)\rvert &+ \lvert\partial^\beta \varphi(x) - \partial^\beta P^y_{\alpha_0}(x)\rvert \leq  \lvert\partial^\beta P^y_{\alpha_0}(x)\rvert + C \|\varphi\|_{\LR} \\
& = \Bigl|\sum_{|\gamma| \leq m-1-|\beta|} \frac{1}{\gamma!}\partial^{\beta+\gamma} P^y_{\alpha_0}(y) (x-y)^\gamma\Bigr| + C \|\varphi\|_{\LR}.\notag{}
\end{align}
Note that $\beta > \alpha_0 \implies \beta + \gamma > \alpha_0$. Thus, (\ref{mainprops3}),\eqref{p201} imply that
\begin{equation}
\label{i205}
\lvert\partial^\beta \varphi(x)\rvert \leq C' \epsilon \qquad (x \in Q^\circ, \; \beta\in \cM, \beta > \alpha_0).
\end{equation}

Choose a cutoff function $\theta \in C^\infty_0(Q^\circ)$ with 
\begin{align}
\label{p301} &\theta \equiv 1 \; \mbox{in a neighborhood of} \; E, \; \mbox{and} \\
\label{p302} &|\partial^\alpha \theta| \lesssim 1 \; \mbox{when} \; |\alpha|\leq m.
\end{align} 
(This cutoff exists in view of (\ref{cube}).) Define
$$\widehat{\varphi} = P^y_{\oa} + \theta (q^y \varphi - P^y_{\oa}).$$ 
Since $y \in E$, by \eqref{p202},\eqref{p203},\eqref{p301}, we have
\begin{equation}
\label{dada}
J_{y}\widehat{\varphi} = P^y_{\alpha_0} \odot_y q^y = P^y_{\overline{\alpha}} \quad \mbox{and} \quad \widehat{\varphi}|_E = 0.
\end{equation}
Because $\theta$ is compactly supported on $Q^\circ$, the function $\widehat{\varphi}$ agrees with an $(m-1)^\rst$ degree polynomial outside $\supp(\theta) \subset Q^\circ$. Also, since the $m^\th$ order derivatives of $P^y_{\oa} \in \cP$ vanish, we find that
\begin{equation*} 
\|\widehat{\varphi}\|_{\LR} = \|\widehat{\varphi}\|_{L^{m,p}(Q^\circ)} = \|\theta \cdot (q^y\varphi - P^y_{\oa})\|_{L^{m,p}(Q^\circ)} \oeq{\eqref{eq_601},\eqref{p203}}  \|\theta \cdot (q^y \varphi - J_{y} \left[ q^y \varphi \right])\|_{L^{m,p}(Q^\circ)}.
\end{equation*}
Thus, by \eqref{p302} and the Sobolev inequality, we have
\begin{equation} \label{i206} \|\widehat{\varphi}\|_{\LR}  \lesssim \sum_{|\alpha| \leq m} \|\partial^\alpha ( q^y \varphi - J_{y} \left[ q^y \varphi \right])\|_{L^p(Q^\circ)} \lesssim \|q^y \varphi\|_{L^{m,p}(Q^\circ)}.
\end{equation}

Consider $x \in Q^\circ$ and some multi-index $\beta$ with $|\beta|=m$. Since $q^y$ is a polynomial of degree $|\gamma|$, we have
\begin{equation}
\label{i207} \partial^\beta \left[q^y \varphi\right](x) =  \sum_{ \omega + \omega' = \beta} \frac{\beta!}{\omega! \omega' !} \partial^\omega q^y(x) \partial^{\omega'} \varphi(x) = \sum_{\substack{\omega + \omega' = \beta \\ |\omega| \leq |\gamma|}} \frac{\beta!}{\omega! \omega' !} \partial^\omega q^y(x) \partial^{\omega'} \varphi(x). \\
\end{equation}
If $\omega + \omega' = \beta$ and $|\omega| \leq |\gamma|$ then $|\omega'| \geq |\beta| - |\gamma| = m - |\gamma|$. Thus, from \eqref{nonmon}, we have $|\omega'| > |\alpha_0|$. By definition of the order on $\cM$, either [$\omega=0$ and $\omega'=\beta$] or [$\omega' \in \cM$ and $\omega' > \alpha_0$]. Using either $\|q^y\|_{L^{\infty}(Q^\circ)} \leq C$ (see the definition of $q^y$ in \eqref{eq_601}) or \eqref{i205} to bound the corresponding summand from \eqref{i207}, we obtain
$$|\partial^\beta [q^y \varphi](x)| \leq  C' \cdot \lvert \partial^\beta \varphi(x)\rvert + C' \cdot \epsilon \quad (x \in Q^\circ,|\beta|=m).$$
Taking $p^\th$ powers, integrating over $Q^\circ$ and maximizing with respect to the multi-indices $\beta$ with $|\beta|=m$, we obtain
$$\|\widehat{\varphi}\|_{\LR}^p \olesssim{\eqref{i206}} \|q^y \varphi\|^p_{L^{m,p}(Q^\circ)} \lesssim \|\varphi\|_{L^{m,p}(Q^\circ)}^p + \epsilon^p \olesssim {\eqref{p201}} \epsilon^p.$$
Together with \eqref{dada}, this proves \eqref{i204}.

Now (\ref{mainprops1})-(\ref{mainprops4}) and \eqref{dia1}-\eqref{i204} imply that $(P^y_\alpha)_{\alpha \in \oA}$ satisfies (\ref{sec_plan}.\ref{trans1})-(\ref{sec_plan}.\ref{trans3}) with $R$ equal to a universal constant, $\epsilon_2 = C' \epsilon$ and $\sigma = \sigma(y,E)$. We may assume that $C' \epsilon \leq c_2$, where the constant $c_2$ comes from Lemma \ref{weaktostrong}. Then the hypotheses of Lemma \ref{weaktostrong} hold, hence $\sigma(y,E)$ contains an $(\oA,y,C'' \epsilon,1)$-basis for some universal constant $C'' \geq 1$.  

Finally, we may assume that $\epsilon \leq \epsilon_0/C''$, in which case $\sigma(y,E)$ contains an $(\oA, y, \epsilon_0,1)$-basis. Since $y \in E$ was arbitrary this concludes the proof of $\circledast$. As mentioned already, this contradicts (\ref{d1}). Thus we may assume that $\cA \subsetneq \cM$ is a monotonic set. This property will be called upon much later to prove Proposition \ref{special} (but nowhere else).

\section{The CZ Decomposition}\label{sec_cz}
\setcounter{equation}{0}

To start, we make two definitions.

\begin{defn}[OK cubes]
A dyadic cube $Q \subset Q^\circ$ is OK if either $\#(3Q \cap E) \leq 1$ or there exists $\oA < \cA$ such that $\sigma(x,E \cap 3Q)$ contains an $(\oA,x,\epsilon_0,30\delta_Q)$-basis for every $x \in E \cap 3Q$.
\end{defn}

\begin{defn}[Calder\'on-Zygmund cubes]
A dyadic cube $Q \subset Q^\circ$ is CZ if $Q$ is OK and all dyadic cubes $Q' \subset Q^\circ$ that properly contain $Q$ are not OK.
\end{defn}
The collection of CZ cubes will be denoted by
$$\CZ = \{Q_1,\ldots,Q_{\nu_{\max}}\}.$$

\begin{lem} \label{lem_fin}
The Calder\'on-Zygmund cubes $\CZ$ form a non-trivial finite partition of $Q^\circ$ into dyadic cubes. In particular, each $Q \in \CZ$ has a (unique) dyadic parent $Q^+ \subset Q^\circ$. 
\end{lem}
\begin{proof}
Recall our assumption that $E \subset \R^n$ is finite with $\#(E) \geq 2$ (see (\ref{sec_ind}.\ref{atleasttwo})). Therefore, 
$$\delta := \inf \left\{|x-y| : x,y \in E \; \mbox{distinct} \right\} \in (0,\infty).$$ 
Let $Q \subset Q^\circ$ be some dyadic cube with sidelength $\delta_Q \leq \delta/4$. It follows that $\#(E \cap 3Q) \leq 1$, hence $Q$ is OK. Since there are finitely many dyadic subcubes of $Q^\circ$ with sidelength greater or equal to $\delta/4$, it follows that $\CZ$ is a finite partition of $Q^\circ$.

Suppose that $Q^\circ$ is OK. Since $E \subset Q^\circ$, this means that either 
\begin{itemize}
\item[(a)] $\#(E) \leq 1$ or 
\item[(b)] $\sigma(x,E)$ contains an $(\cA',x,\epsilon_0,30 \delta_{Q^\circ})$-basis for each $x \in E$ and some $\cA' < \cA$.
\end{itemize}
Since $\delta_{Q^\circ}=1$, we can replace $30 \delta_{Q^\circ}$ by $1$ in (b), while retaining validity (see the remark in section \ref{poly_basis}). Therefore, (a) contradicts $\#(E) \geq 2$, while (b) contradicts (\ref{sec_ind}.\ref{d1}).

It follows that $Q^\circ$ is not OK. Therefore, $\CZ$ is not the trivial partition $\{Q^\circ\}$.
\end{proof}

Two cubes $Q_{\nu}, Q_{\nu'} \in \CZ$ are called ``neighbors'' if their closures satisfy $\cl(Q_\nu)  \cap \cl(Q_{\nu'}) \neq \emptyset$. (In particular, any CZ cube is neighbors with itself.) We denote this relation by $Q_\nu \leftrightarrow Q_{\nu'}$ or $\nu \leftrightarrow \nu'$.

Suppose that $\#(E \cap 3Q^+) \leq 1$ for some $Q \in \CZ$. Then by definition $Q^+$ is OK, which contradicts $Q \in \CZ$. Thus, we have
\begin{equation}
\#(E \cap 9Q) \geq \#(E \cap 3Q^+) \geq 2 \;\; \mbox{for each}\; Q \in \CZ. \label{nonempty}
\end{equation}
(Here, we are using the fact that $Q^+ \subset 3Q$ for any cube $Q \subset \R^n$.)

\begin{lem}[Good Geometry] \label{gg}
If $Q,Q' \in \CZ$ satisfy $Q \leftrightarrow Q'$, then $\frac{1}{2} \delta_Q \leq \delta_{Q'} \leq 2 \delta_Q$.
\end{lem}
\begin{proof}
For the sake of contradiction, suppose that $Q,Q' \in \CZ$ satisfy
\begin{equation*}
\cl(Q) \cap \cl(Q') \neq \emptyset \;\; \mbox{and} \;\; 4 \delta_Q \leq \delta_{Q'}.
\end{equation*} 
Since $Q^+$,$Q'$ are dyadic cubes, it follows that $3Q^+ \subset 3 Q'$. Since $Q'$ is OK, either 
\begin{itemize}
\item[(a)] $\#(3Q' \cap E) \leq 1$ or 
\item[(b)] $\sigma(x,E \cap 3Q')$ contains an $(\oA,x,\epsilon_0,30\delta_{Q'})$-basis for each $x \in 3Q' \cap E$ and some $\oA < \cA$.
\end{itemize}

Note that (a) implies that $\#(3Q^+ \cap E) \leq 1$, which contradicts \eqref{nonempty}. Thus it remains to consider (b).

By definition of $\sigma(\cdot,\cdot)$, we have $\sigma(x,E \cap 3Q') \subset \sigma(x,E \cap 3Q^+)$ for every $x \in E \cap 3Q^+$. Therefore, (b) implies that $\sigma(x,E \cap 3Q^+)$ contains an $(\oA,x,\epsilon_0,30 \delta_{Q'})$-basis for each $x \in 3Q^+ \cap E$. Since $\delta_{Q^+} \leq \delta_{Q'}$, we may replace $\delta_{Q'}$ by $\delta_{Q^+}$ in the previous statement, while retaining validity. It follows that $Q^+$ is OK, which contradicts $Q \in \CZ$.
\end{proof}

\begin{lem}[More Good Geometry]\label{moregg} For each $Q \in \CZ$, the following properties hold.
\begin{itemize}
\item If $Q' \in \CZ$ is such that $(1.3)Q' \cap (1.3)Q\neq \emptyset$ then $Q \leftrightarrow Q'$. Consequently, each point $x \in \R^n$ belongs to at most $C(n)$ of the cubes $(1.3)Q$ with $Q \in \CZ$.
\item If $Q' \in \CZ$ is such that $(1.1)Q' \cap 10Q \neq \emptyset$ then $\delta_{Q'} \leq 50 \delta_{Q}$.
\item If $Q' \in \CZ$ is such that $(1.1)Q' \cap 100Q \neq \emptyset$ then $\delta_{Q'} \leq 10^3 \delta_{Q}$.
\item If $\cl(Q) \cap \partial Q^\circ \neq \emptyset$ then $\delta_Q \geq \frac{1}{20} \delta_{Q^\circ}$.
\end{itemize}
\end{lem}

\begin{proof}
Let $Q \in \CZ$ be fixed. We start with the first bullet point.  Suppose that $Q' \in \CZ$ does not neighbor $Q$. Put $\delta = \max\{\delta_Q/2,\delta_{Q'}/2\}$. Then Lemma \ref{gg} implies that $\dist(Q,Q') \geq \delta$ (the CZ cubes that neighbor the larger cube have sidelength at least $\delta$, providing a buffer between $Q$ and $Q'$ of this width). Thus,
$$(1.3)Q \cap (1.3)Q' \subset B(Q,(0.3)\delta) \cap B(Q',(0.3)\delta) = \emptyset.$$
This completes the proof of the first statement of the first bullet point. The second statement is immediate from Lemma \ref{gg}. We pass to the second bullet point.

For the sake of contradiction, suppose that there exists $Q' \in \CZ$ with $(1.1)Q' \cap 10Q \neq \emptyset$ and $\delta_{Q'} > 50 \delta_{Q}$. Note that
$$(1.1)Q' \cap 10Q \neq \emptyset \implies \dist(Q,Q') \leq (0.05) \delta_{Q'} + (4.5) \delta_{Q}.$$ 
Since $\delta_{Q'} > 50 \delta_Q$ it follows that $\dist(Q,Q') < (0.05 + .09) \delta_{Q'} < (0.15) \delta_{Q'}.$ Therefore, $(1.3) Q \cap (1.3) Q' \neq \emptyset$. However, this contradicts $\delta_{Q'} > 50 \delta_{Q}$ in view of the first bullet point and Lemma \ref{gg}. This completes the proof of the second bullet point.

The proof of the third bullet point is analogous to the above.

Finally, we prove the fourth bullet point. For the sake of contradiction, suppose that $\cl (Q) \cap \partial Q^\circ \neq \emptyset$ and $\delta_Q < \delta_{Q^\circ}/20$. Therefore, $ 9Q \subset \R^n \setminus (1/8)Q^\circ$. Because $E \subset (1/8)Q^\circ$, we have $9Q \cap E = \emptyset$. However, this contradicts (\ref{nonempty}). This proves the fourth bullet point, and completes the proof of Lemma \ref{moregg}
\end{proof}

\section{Paths to Keystone Cubes I}
\setcounter{equation}{0}

\label{PTKC}

In this section, constants called $c_{G},c,c^{\prime },C,C^{\prime }$, etc.
depend only on the dimension $n$. They are \textquotedblleft
controlled\textquotedblright\ constants. The lower case letters denote small (controlled) constants while the
upper case letters denote large (controlled) constants. $A$ is a constant to be picked
later. We assume $A$ is greater than a large enough controlled constant (\textquotedblleft large $A$
assumption\textquotedblright ).

We derive our main proposition in the following setting.
\begin{itemize}
\item We are given a CZ decomposition: $\R^n$ is partitioned into a collection $CZ$ of
dyadic cubes, such that $Q,Q^{\prime }\in CZ$ and $(1+10c_{G})Q\cap
(1+10c_{G})Q^{\prime }\neq \emptyset $ $\implies $ $\frac{1}{64}\delta
_{Q}\leq \delta _{Q^{\prime }}\leq 64\delta _{Q}$ (\textquotedblleft good
geometry\textquotedblright ).
\item We are given a finite set $E\subset \R^n$ with cardinality $N=\#(E) \geq 2$.
\item The CZ cubes are related to the set $E$ as follows: $\#(E\cap 9Q)\geq
2 $ for each $Q\in CZ$.

\end{itemize}

\begin{defn}
\label{keystonecube1}A cube $Q \in CZ$ is called a keystone cube if for
any $ Q^{\prime }\in CZ$ with $Q^{\prime }\cap 100 Q\not=\emptyset $, we
have $\delta _{Q^{\prime }}\geq \delta _{Q}$.
\end{defn}

This section is devoted to a proof of the following proposition:

\begin{prop}
\label{pathassignment} We can find a subset $CZ_{\spec}\subset CZ$,
and an assignment to each $Q\in CZ$ of a finite sequence $\mathcal{S}%
_{Q}=(Q_1,Q_{2},\ldots ,Q_{L})$ of CZ cubes (with length $L$ depending on $Q$), such that the following properties hold.

\begin{enumerate}[(i)]
\item $CZ_{\spec}$ contains at most $C \cdot N$ distinct cubes.

\item $Q_1=Q$, $Q_L$ is a keystone cube; $(1+c_G)Q_l \cap (1+c_G) Q_{l+1} \neq \emptyset$
for every $1 \leq l < L$; and 
$$\delta_{Q_k} \leq C \cdot(1-c)^{k - l}\delta_{Q_l} \;\; \mbox{for} \; 1 \leq l \leq k \leq L.$$

\item Let $Q,Q^{\prime }\in CZ \setminus CZ_{\spec}$, and let $%
\mathcal{S}_Q = (Q_1,\ldots,Q_L)$, $\mathcal{S}_{Q^{\prime }} = (Q^{\prime
}_1,\ldots,Q^{\prime }_{L^{\prime }})$. 

If $(1+c_G) Q \cap (1+c_G)Q^{\prime
}\neq \emptyset$, then $Q_L= Q^{\prime }_{L^{\prime }}$.
\end{enumerate}
\end{prop}

We begin the proof of Proposition \ref{pathassignment}.

\begin{defn}[Clusters] A subset $S\subset E$ is a cluster if it satisfies $%
\#(S)\geq 2$ and $\mbox{dist}(S,E \setminus S)\geq A^{3}\mbox{diam}(S)$,
where the left-hand side $=\infty $ if $E \setminus S=\emptyset $.
\end{defn}

\begin{defn}[Representatives] For each cluster $S$ we pick a representative $%
x(S) \in S$. We write $Q(S)$ to denote the CZ cube containing $x(S)$.
\end{defn}

\begin{defn}[Halos] For each cluster $S$, we define the halo of $S$ to be 
\begin{equation*}
H(S)=\left\{ x\in \mathbb{R}^{n}:A\cdot \mbox{diam}(S)<\lvert x-x(S)\rvert<A^{-1}%
\mbox{dist}(S,E\setminus S)\right\} .
\end{equation*}%
Again, $\mbox{dist}(S,E\setminus S)=\infty $ if $S=E$.
\end{defn}

\begin{defn}[Interstellar cubes] $Q\in CZ$ is interstellar if $\mbox{diam}%
(A^{10}Q\cap E)\leq A^{-10}\delta _{Q}$ and $(1+3c_{G})Q\cap
E=\emptyset $.
\end{defn}

We recall the Well Separated Pairs Decomposition \cite{CC}.

\begin{thm}[Well Separated Pairs Decomposition]
Let $E \subset \R^n$ be finite. There exists a list of non-empty Cartesian products $E_{1}^{\prime }\times E_{1}^{\prime \prime }, E_{2}^{\prime }\times E_{2}^{\prime \prime }, \ldots, E_{\nu _{\max }}^{\prime }\times E_{\nu _{\max}}^{\prime \prime }$, each contained in $E \times E \setminus \{(x,x) : x \in E\}$, such that the following properties hold.
\begin{itemize}
\item For each $\nu=1,\ldots,\nu_{\max}$, we have $\mbox{diam}(E_{\nu }^{\prime })+\mbox{diam}(E_{\nu }^{\prime
\prime })\leq 10^{-6}\mbox{dist}(E_{\nu }^{\prime },E_{\nu }^{\prime \prime
})$.
\item Each pair $(x^{\prime },x^{\prime \prime })\in E \times E$ with $x^\prime \neq x^{\prime \prime}$
belongs to precisely one $E_\nu^\prime \times E_\nu^{\prime \prime}$.
\item $\nu_{\max} \leq C N$.
\end{itemize}
\end{thm}

Fix a representative $(x_\nu^\prime, x_\nu^{\prime \prime}) \in E_\nu^\prime \times E_\nu^{\prime \prime}$ for each $\nu=1,\ldots,\nu_{\max}$. Then for any $(x^\prime,x^{\prime \prime}) \in E \times E$ with $x^\prime \neq x^{\prime \prime}$, we have $(x^\prime, x^{\prime \prime}) \in E_\nu^\prime \times E_\nu^{\prime \prime}$ for some $\nu$. For that $\nu$, we have
\begin{align*}
\lvert x_\nu^\prime - x^{\prime}\rvert + \lvert x_\nu^{\prime \prime} - x^{\prime \prime}\rvert &\leq \diam(E_\nu^\prime) + \diam(E_\nu^{\prime \prime}) \\
& \leq 10^{-6} \dist(E_\nu^\prime, E_\nu^{\prime \prime}) \leq 10^{-6} \lvert x^\prime - x^{\prime \prime}\rvert,\\
\mbox{and similarly}& \\
\lvert x_\nu^\prime - x^\prime\rvert + \lvert x_\nu^{\prime \prime} - x^{\prime \prime}\rvert &\leq 10^{-6} \lvert x^\prime_\nu - x^{\prime \prime}_\nu\rvert.
\end{align*}

Using the Well Separated Pairs Decomposition, we prove the following.

\begin{lem}
\label{lemA} The number of non-interstellar cubes $Q \in CZ$ is at most $%
C(A)\cdot N$; here, $C(A)$ depends only on $A$ and on the dimension $n$.
\end{lem}

\begin{lem}
\label{lemB} The number of distinct clusters is at most $C N$.
\end{lem}

\begin{proof}[Proof of Lemma \protect\ref{lemA}]
The non-interstellar cubes are $Q\in CZ$ such that $(1+3c_{G})Q\cap
E\neq \emptyset $ or $\mbox{diam}(A^{10}Q\cap E)>A^{-10}\delta _{Q}$.

For each $x\in E$, there are at most $C$ distinct CZ cubes $Q$ such that 
$x\in (1+3c_{G})Q$. Therefore, there are at most $CN$ distinct CZ cubes $Q$
such that $(1+3c_{G})Q\cap E\neq \emptyset $.

Next, we consider $Q\in CZ$ such that $\mbox{diam}(A^{10}Q\cap
E)>A^{-10}\delta _{Q}$. For such $Q$, we can find $x^{\prime },x^{\prime
\prime }\in E\cap A^{10}Q$ such that $\left\vert x^{\prime }-x^{\prime
\prime }\right\vert >A^{-10}\delta _{Q}$. For some $\nu $, we have 
\begin{equation*}
\left\vert x_{\nu }^{\prime }-x^{\prime }\right\vert +\left\vert x_{\nu
}^{\prime \prime }-x^{\prime \prime }\right\vert \leq 10^{-6}\left\vert
x^{\prime }-x^{\prime \prime }\right\vert \leq CA^{10}\delta _{Q}.
\end{equation*}
Hence, $\lvert x_{\nu }^{\prime }-x_{\nu }^{\prime \prime } \rvert \geq \frac{1}{2}%
A^{-10}\delta _{Q}$, yet $x_{\nu }^{\prime },x_{\nu }^{\prime \prime }\in
CA^{10}Q$.

Therefore, the number of $Q\in CZ$ such that $\mbox{diam}(A^{10}Q\cap
E)>A^{-10}\delta _{Q}$ is at most the sum over all $\nu =1,\ldots ,\nu
_{\max }$ of 
\begin{equation*}
\left[ \mbox{the number of distinct dyadic cubes}\;\;Q\;\mbox{such that}%
\;x_{\nu }^{\prime },x_{\nu }^{\prime \prime }\in CA^{10}Q\;\mbox{and}%
\;\lvert x_{\nu }^{\prime }-x_{\nu }^{\prime \prime } \rvert\geq \frac{1}{2}%
A^{-10}\delta _{Q}.\right]
\end{equation*}

For each fixed $\nu$, the quantity in the square brackets is at most $C(A)$;
and the number of distinct $\nu=1,\ldots,\nu_{\max}$ is at most $CN$.

Thus, there are at most $C(A)N$ distinct $Q\in CZ$ such that $\mbox{diam}%
(A^{10}Q\cap E)>A^{-10}\delta _{Q}$.

The proof of Lemma \ref{lemA} is complete.
\end{proof}

\begin{proof}[Proof of Lemma \protect\ref{lemB}]
Let $S$ be a cluster. Fix $x^{\prime },x^{\prime \prime }\in S$ such that $%
|x^{\prime }-x^{\prime \prime }|=\mbox{diam}(S)$. Then any point $y\in
E\setminus S$ satisfies $\left\vert x^{\prime }-y\right\vert \geq
A^{3}|x^{\prime }-x^{\prime \prime }|$. Fix $\nu $ such that 
\begin{equation*}
\left\vert x_{\nu }^{\prime }-x^{\prime }\right\vert +\left\vert x_{\nu
}^{\prime \prime }-x^{\prime \prime }\right\vert \leq 10^{-6}\left\vert
x^{\prime }-x^{\prime \prime }\right\vert
\end{equation*}%
and 
\begin{equation*}
\left\vert x_{\nu }^{\prime }-x^{\prime }\right\vert +\left\vert x_{\nu
}^{\prime \prime }-x^{\prime \prime }\right\vert \leq 10^{-6}\lvert x_{\nu
}^{\prime }-x_{\nu }^{\prime \prime }\rvert.
\end{equation*}%
Then 
\begin{equation*}
(1-10^{-6})\lvert x^{\prime }-x^{\prime \prime } \rvert \leq \lvert x_{\nu }^{\prime }-x_{\nu
}^{\prime \prime }\rvert \leq \left( 1+10^{-6}\right) \lvert x^{\prime }-x^{\prime
\prime }\rvert.
\end{equation*}%
Any $y\in S$ satisfies 
\begin{align*}
\lvert y-x_{\nu }^{\prime }\rvert & \leq \lvert y-x^{\prime }\rvert +\lvert x^{\prime }-x_{\nu }^{\prime
} \rvert \leq \mbox{diam}(S)+10^{-6}\lvert x^{\prime }-x^{\prime \prime }\rvert \\
& =\lvert x^{\prime }-x^{\prime \prime }\rvert +10^{-6}\lvert x^{\prime }-x^{\prime \prime
}\rvert \leq 2\lvert x_{\nu }^{\prime }-x_{\nu }^{\prime \prime }\rvert .
\end{align*}%
On the other hand, any $y\in E\setminus S$ satisfies 
\begin{align*}
\lvert y-x_{\nu }^{\prime }\rvert & \geq \lvert y-x^{\prime }\rvert -\lvert x_{\nu }^{\prime }-x^{\prime
}\rvert \geq A^{3}\lvert x^{\prime }-x^{\prime \prime }\rvert -10^{-6}\lvert x^{\prime }-x^{\prime
\prime }\rvert \\
&>A\lvert x_{\nu }^{\prime
}-x_{\nu }^{\prime \prime }\rvert.
\end{align*}%
Consequently, $S=E\cap B(x_{\nu }^{\prime },10\lvert x_{\nu }^{\prime }-x_{\nu
}^{\prime \prime }\rvert )$ for some $\nu $.

Thus, every cluster $S$ arises as $E\cap B(x_{\nu }^{\prime },10\lvert x_{\nu
}^{\prime }-x_{\nu }^{\prime \prime }\rvert )$ for some $\nu =1,\ldots ,\nu _{\max
}$. Since $\nu _{\max }\leq CN$, there can be at most $CN$ distinct
clusters. The proof of Lemma \ref{lemB} is complete.
\end{proof}

\begin{lem}
\label{lem1} If $Q$ is an interstellar cube, then $(1+c_G) Q \subset H(S)$
for some cluster $S$.
\end{lem}

\begin{proof}
Let $S = E \cap A^{10} Q$. Then $\#(S) \geq \#(E \cap 9Q) \geq 2$; also $%
\mbox{diam}(S) \leq A^{-10} \delta_Q$. Since $S$ intersects $9Q$, it follows
that $S \subseteq 13Q$. On the other hand, $E \setminus S \subseteq 
\mathbb{R}^n \setminus A^{10} Q$. Consequently, $\mbox{dist}(S,E \setminus
S) \geq \mbox{dist}(13Q,\mathbb{R}^n \setminus A^{10}Q) \geq c A^{10}
\delta_Q$. Thus, $\mbox{dist}(S,E \setminus S) \geq c A^{20} \mbox{diam}%
(S) $, proving that $S$ is a cluster.

Next, let $x\in (1+c_{G})Q$. We know that $(1+3c_{G})Q\cap E=\emptyset $
since $Q$ is interstellar. Hence, $\lvert x-x(S) \rvert\geq c\delta _{Q}$ since $x(S)\in
S\subseteq E$. Therefore, $\lvert x-x(S)\rvert \geq cA^{10}\mbox{diam}(S)$. On the
other hand, $x(S)\in S\subseteq 13Q$, and $x\in (1+c_{G})Q$. It follows that 
$\lvert x-x(S) \rvert \leq C\delta _{Q}\leq C^{\prime}A^{-10}\mbox{dist}(S,E\setminus S)$%
.

Thus, $cA^{10} \mbox{diam}(S) \leq \lvert x-x(S) \rvert \leq C^{\prime}A^{-10} \mbox{dist}%
(S,E \setminus S)$. We have shown that each $x \in (1+c_G) Q$ belongs to $%
H(S)$.
\end{proof}

\begin{lem}
\label{lem2} Let $S$ be a cluster. Let $Q \in CZ$, and let $x \in (1+c_G)Q
\cap H(S)$. Then 
\begin{equation*}
\frac{1}{2} A^{-1} \Bigl[ \lvert x-x(S)\rvert + \delta_{Q(S)} \Bigr] \leq \delta_Q \leq
A \Bigl[ \lvert x-x(S)\rvert  + \delta_{Q(S)}\Bigr].
\end{equation*}
\end{lem}

\begin{proof}
We know that $A\cdot \mbox{diam}(S)<\lvert x-x(S) \rvert <A^{-1}\mbox{dist}%
(S,E\setminus S)$. In particular, 
\begin{equation*}
S\subset B\bigl( x,\lvert x-x(S)\rvert +\mbox{diam}(S)\bigr) \subset B\left(
x,[1+A^{-1}] \cdot \lvert x-x(S)\rvert \right) \subset B\bigl( x,2\lvert x-x(S) \rvert \bigr)
\end{equation*}%
and 
\begin{equation*}
\mbox{dist}(x,S)\geq \lvert x-x(S) \rvert-\mbox{diam}(S)>[1-A^{-1}] \cdot \lvert x-x(S)\rvert\geq \frac{1}{%
2}\lvert x-x(S)\rvert.
\end{equation*}%
Also, \ 
\begin{equation*}
\mbox{dist}(x,E\setminus S)\geq \mbox{dist}\bigl( x(S),E\setminus
S \bigr)-\lvert x-x(S)\rvert \geq \mbox{dist}(S,E\setminus S)-\lvert x-x(S)\rvert>[A-1] \cdot\lvert x-x(S)\rvert.
\end{equation*}%
In particular, 
\begin{equation*}
\mbox{dist}(x,E\setminus S)>[A-1]\cdot \lvert x-x(S)\rvert >2\lvert x-x(S)\rvert \geq \mbox{dist}(x,S).
\end{equation*}%
Therefore, $\mbox{dist}(x,E)=\mbox{dist}(x,S)\geq \frac{1}{2} \lvert x-x(S)\rvert.$
On the other hand, $%
E\cap 9Q\neq \emptyset $ since $Q \in CZ$; and $x\in (1+c_{G})Q$%
. Therefore, $\mbox{dist}(x,E)\leq C\delta _{Q}$. It follows that $\lvert x-x(S) \rvert\leq 2\mbox{dist}(x,E)\leq
C^{\prime }\delta _{Q}$.

Next, we check that $\delta_{Q(S)} \leq A \delta_Q$. In fact, suppose that $%
\delta_{Q(S)} > A \delta_Q$. Then also $\delta_{Q(S)} > c A\lvert x-x(S)\rvert$. Since $%
x(S) \in Q(S)$ by definition, it follows that $x \in (1+c_G) Q(S)$. On the
other hand, $x \in (1+c_G)Q$. Therefore, $\delta_{Q(S)}$ and $\delta_Q$
differ by at most a factor of $64$, thanks to the good geometry of the $CZ$
cubes. This contradicts our assumption that $\delta_{Q(S)} > A \delta_Q$,
completing the proof that $\delta_{Q(S)} \leq A \delta_Q$.

We now know that $\frac{1}{2} A^{-1} \bigl[\delta_{Q(S)} + \lvert x-x(S) \rvert\bigr] \leq
\delta_Q.$

Next, we show that $\delta _{Q}\leq A \bigl[ \delta _{Q(S)}+ \lvert x-x(S)\rvert \bigr]$. Indeed,
suppose that 
$$\delta _{Q}>A\bigl[ \delta _{Q(S)}+\lvert x-x(S)\rvert \bigr].$$ 
Since $x\in
(1+c_{G})Q $ and $A^{-1}\delta _{Q}>\lvert x-x(S)\rvert$, it follows that $x(S)\in
(1+2c_{G})Q$. On the other hand, $x(S)\in Q(S)$. By the good
geometry of the $CZ$ cubes, the side lengths $\delta _{Q}$ and $\delta
_{Q(S)}$ can differ at most by a factor of $64$. This contradicts our
assumption that $\delta _{Q}>A\delta _{Q(S)}$, completing the proof that $%
\delta _{Q}\leq A\bigl[ \delta _{Q(S)}+ \lvert x-x(S)\rvert \bigr]$.
\end{proof}

\begin{lem}
\label{lem3} For any two distinct clusters $S,S^{\prime }$, the halos $%
H(S),H(S^{\prime })$ are disjoint.
\end{lem}

\begin{proof}
Suppose $x\in H(S)\cap H(S^{\prime })$, with $S$ and $S^{\prime }$ distinct
clusters. Then 
$$A\cdot \mbox{diam}(S)<\lvert x-x(S) \rvert<A^{-1}\mbox{dist}%
(S,E\setminus S).$$ 
Let $R_{S}=2|x-x(S)|$. Then $$S\subset B\bigl(
x,\lvert x-x(S)\rvert+\mbox{diam}(S)\bigr) \subset B\left( x,[1+A^{-1}] \cdot \lvert x-x(S) \rvert \right)
\subset B(x,R_{S}).$$ 
For $y\in E\setminus S$, we have 
\begin{align*}
\lvert y - x \rvert& \geq \mbox{dist}(y,S)-\mbox{dist}(x,S)\geq \mbox{dist}%
(E\setminus S,S)-\mbox{dist}(x,S)\geq \mbox{dist}(E\setminus
S,S)-\lvert x-x(S) \rvert \\
& \geq \lbrack A-1] \cdot \lvert x-x(S)\rvert>2\lvert x-x(S)\rvert=R_{S}.
\end{align*}%
Therefore, $(E\setminus S)\cap B(x,R_{S})=\emptyset $. Since we observed
that $S\subset B(x,R_{S})$ and since $S\subset E$, we conclude that $%
S=B(x,R_{S})\cap E$.

Similarly, $S^{\prime }= B(x,R_{S^{\prime }}) \cap E$, where $R_{S^{\prime
}} = 2\lvert x-x(S^{\prime })\rvert$. It follows that $S \subset S^{\prime }$ or $%
S^{\prime }\subset S$. Without loss of generality, we may suppose $S \subset
S^{\prime }$. Since $S$ and $S^{\prime }$ are distinct, we can find $y \in
S^{\prime }\setminus S \subset E \setminus S$.

Note that $\mbox{diam}(S^{\prime }) \geq \lvert y-x(S) \rvert \geq \mbox{dist}(E
\setminus S,S)$. Since $x \in H(S^{\prime })$, we have $\lvert x-x(S^{\prime })\rvert >
A \cdot \mbox{diam}(S^{\prime })$, hence 
\begin{eqnarray*}
\lvert x-x(S)\rvert &>& A \cdot \mbox{diam}(S^{\prime }) - \lvert x(S)-x(S^{\prime })\rvert \geq A
\cdot \mbox{diam}(S^{\prime }) - \mbox{diam}(S^{\prime }) \\
&&(\text{since } x(S) \in S \subset S^{\prime }\text{ and } x(S^{\prime })
\in S^{\prime }) \\
&\geq& \frac{1}{2} A \cdot \mbox{diam}%
(S^{\prime }).
\end{eqnarray*}
On the other hand, since $x \in H(S)$, we have 
\begin{equation*}
\lvert x-x(S)\rvert \leq A^{-1} \mbox{dist}(S,E \setminus S) \leq A^{-1} \mbox{dist}%
(S,S^{\prime} \setminus S) \leq  A^{-1} \mbox{diam}(S^{\prime }).
\end{equation*}

Thus, $\lvert x-x(S)\rvert \geq \frac{1}{2} A \cdot \mbox{diam}(S^{\prime })$, and $%
\lvert x-x(S)\rvert \leq A^{-1} \mbox{diam}(S^{\prime })$. This contradiction shows
that we cannot have $x \in H(S) \cap H(S^{\prime })$.
\end{proof}

\begin{lem}
\label{lem4} Let $S$ be a cluster. Let $x,x^{\prime }\in H(S)$, and let $%
Q,Q^{\prime }\in CZ$, with $x \in Q$ and $x^{\prime }\in Q^{\prime }$. If $%
\lvert x-x^{\prime} \rvert \leq A^{-2} \lvert x-x(S) \rvert$, then $(1+c_G) Q \cap (1+c_G) Q^{\prime }\neq
\emptyset$.
\end{lem}

\begin{proof}
By Lemma \ref{lem2}, we have $\lvert x^{\prime} - x \rvert \leq A^{-2} \left[ \lvert x-x(S)\rvert + \delta_{Q(S)}%
\right] \leq 2 A^{-1} \delta_Q$. Since $x \in Q$, it follows that $x^{\prime
}\in (1+c_G)Q$. Also $x^{\prime }\in Q^{\prime }\subset (1+c_G) Q^{\prime }$.
\end{proof}

\begin{lem}
\label{lem5} Fix a cluster $S$. Let $x,x^{\prime }\in H(S)$, with $\lvert x-x(S) \rvert
\geq \lvert x^{\prime }-x(S) \rvert$. Then there exist a finite sequence of points $%
x_1,x_2,\ldots,x_L \in H(S)$, and a positive integer $L_*$, with the
following properties:

\begin{itemize}
\item $x_1=x$ and $x_L = x^{\prime }$.

\item $\lvert x_{l+1} - x(S) \rvert \leq \lvert x_l - x(S) \rvert$ for $l=1,\ldots,L-1$.

\item $\lvert x_l - x_{l+1}\rvert \leq A^{-2} \lvert x_l - x(S) \rvert$ for $l=1,\ldots,L-1$.

\item $\lvert x_{l+L_*} - x(S)\rvert \leq (1-A^{-3}) \lvert x_l - x(S) \rvert$ for $1 \leq l \leq
L-L_*$.

\item $L_* \leq A^3$.
\end{itemize}
\end{lem}

\begin{proof}
Define a point $\tilde{x}\in \R^n$ such that 
\begin{equation*}
\tilde{x}-x(S)=\lvert x-x(S) \rvert\cdot \frac{x^{\prime }-x(S)}{\lvert x^{\prime }-x(S)\rvert }.
\end{equation*}%
Then $\tilde{x}\in H(S)$ since $\lvert \tilde{x}-x(S) \rvert=\lvert x-x(S)\rvert$ and $x\in H(S)$.
We have 
\begin{equation}
\tilde{x}-x(S)=T\cdot [x^{\prime }-x(S)]\;\;\mbox{for some }T\geq 1.  \label{ex}
\end{equation}%
We pick points $x_{1},\ldots ,x_{L_{1}}$ in $\partial
B\bigl(x(S),\lvert x-x(S)\rvert \bigr)$ such that 
\begin{equation}
\begin{aligned}
&x_{1}=x;\;x_{L_{1}}=\tilde{x}; \\
&\lvert x_{l+1}-x_{l}\rvert \leq 
A^{-2}\lvert x-x(S) \rvert=A^{-2} \lvert x_{l}-x(S)\rvert \;\;\mbox{for}\;1\leq l<L_{1};\;\mbox{and}%
\;L_{1}\leq \frac{1}{2}A^{3}.
\end{aligned} \label{pound}
\end{equation}

We then pick positive real numbers $T_{L_1}, T_{L_1 + 1}, \ldots,T_{L-1},T_L$
with the properties: 
\begin{align*} &T_{L_1} = T \; \mbox{ as in \eqref{ex}}; \; T_L=1; \; T_{l+1} =
(1-A^{-3})T_l \; \mbox{for} \; L_1 \leq l \leq L-2; \; \mbox{and} \\
&(1-A^{-3})T_{L-1} \leq T_{L} \leq T_{L-1}.
\end{align*}

Define the points $x_{L_1+1},\ldots,x_L \in \mathbb{R}^n$ by setting 
\begin{equation}  \label{ex2}
x_l - x(S) = T_l \cdot [x^{\prime }- x(S)] \;\; \mbox{for} \;
l=L_1+1,\ldots,L.
\end{equation}
Note that \eqref{ex2} holds for $l=L_1$ also, and that $\lvert x_l - x(S)\rvert =
T_l \cdot \lvert x^{\prime }-x(S) \rvert$. We have $1 \leq T_l \leq T$ for each $%
l=L_1,\ldots,L$; therefore $\lvert x^{\prime }-x(S) \rvert \leq \lvert x_l - x(S)\rvert \leq
\lvert x-x(S) \rvert$ for each such $l$. Since $x,x^{\prime }\in H(S)$, it follows that
each $x_l$ also belongs to $H(S)$.

We have 
\begin{equation}  \label{pound2}
\lvert x_l-x_{l+1}\rvert \leq A^{-3} \lvert x_l - x(S)\rvert \;\; \mbox{for} \; L_1 \leq l \leq L-1;
\end{equation}
\begin{equation*}
\lvert x_{l+1}-x(S)\rvert \leq \lvert x_l - x(S) \rvert \;\; \mbox{for} \; L_1 \leq l \leq L-1; \;\mbox{and}
\end{equation*}
\begin{equation*}
\lvert x_{l+1} - x(S) \rvert = (1-A^{-3}) \lvert x_l-x(S) \rvert \;\; \mbox{for} \; L_1 \leq l \leq
L-2.
\end{equation*}

We have now defined $x_1,\ldots,x_L$. Note that $x_1 = x$ and $x_L =
x^{\prime }$, which is the first bullet point in Lemma \ref{lem5}.

We know that $\lvert x_{l+1} - x(S)\rvert = \lvert x_l - x(S)\rvert$ for $1 \leq l \leq L_1-1$,
and $\lvert x_{l+1}-x(S)\rvert \leq \lvert x_l - x(S)\rvert$ for $L_1 \leq l \leq L-1$; therefore $%
\lvert x_{l+1}-x(S)\rvert \leq \lvert x_l - x(S)\rvert$ for $1 \leq l \leq L-1$. This establishes
the second bullet point of Lemma \ref{lem5}. Also, $\lvert x_{l+1} - x(S)\rvert  \leq
(1-A^{-3})\lvert x_l-x(S)\rvert $ for $L_1 \leq l \leq L-2$. The last two estimates
together show that $\lvert x_{l + (L_1+3)} - x(S)\lvert  \leq (1-A^{-3})\lvert x_l - x(S)\rvert $
for $1 \leq l \leq L - (L_1 + 3)$. Here, $L_1 + 3 \leq \frac{1}{2}A^3 + 3 <
A^3$. Thus, we have proven the last two bullet points of Lemma \ref{lem5}.
The third bullet point in Lemma \ref{lem5} is immediate from \eqref{pound}
and \eqref{pound2}. We have verified all the conclusions of Lemma \ref{lem5}.
\end{proof}

\begin{lem}
\label{lem6} Fix a cluster $S$. Let $Q,Q^{\prime }\in CZ$, and let $%
x,x^{\prime }\in H(S)$, with $x \in Q$, $x^{\prime }\in Q^{\prime }$. Assume 
\begin{equation*}
\lvert x-x(S)\rvert  \geq \lvert x^{\prime }-x(S)\rvert  \geq \delta_{Q(S)}.
\end{equation*}

Then there exist cubes $Q_1,Q_2,\ldots,Q_L \in CZ$, with the following
properties.

\begin{itemize}
\item $Q_1=Q$ and $Q_L = Q^{\prime }$.

\item $(1+c_G) Q_l \cap (1+c_G) Q_{l+1} \neq \emptyset$ for all $%
l=1,2,\ldots,L-1$.

\item $\delta_{Q_k} \leq C(A) (1-c(A))^{k-l} \delta_{Q_l}$ for $1
\leq l \leq k \leq L$; here, $0<c(A)<1$ and $C(A)>0$ are constants
depending only on $A$ and on the dimension $n$.
\end{itemize}
\end{lem}

\begin{proof}
Pick a sequence $x_1,x_2,\ldots,x_L \in H(S)$ and an integer $L_*$, as in
Lemma \ref{lem5}. For each $l$, let $Q_l$ be the CZ cube
containing $x_l$. In particular, $Q_1=Q$ and $Q_L = Q^{\prime }$, since $x_1
= x \in Q$ and $x_L = x^{\prime }\in Q^{\prime }$. For each $l=1,2,\ldots,L-1$ , we have $\lvert x_{l+1}-x_l\rvert  \leq A^{-2} \lvert x_l - x(S)\rvert $. Since also $%
x_l \in Q_l$ and $x_{l+1} \in Q_{l+1}$, Lemma \ref{lem4} tells us that $%
(1+c_G) Q_l \cap (1+c_G)Q_{l+1} \neq \emptyset$. In particular, $%
\delta_{Q_{l+1}}$ and $\delta_{Q_l}$ differ by at most a factor of $64$.

Lemma \ref{lem5} gives $\lvert x_{l+1} - x(S)\rvert  \leq \lvert x_l - x(S)\rvert $ for $%
l=1,2,\ldots,L-1$; hence, $\lvert x_l-x(S)\rvert  \geq \lvert x^{\prime }-x(S)\rvert  \geq
\delta_{Q(S)}$. Hence, by Lemma \ref{lem2} (and the fact that $x_l \in Q_l$%
), $\delta_{Q_l}$ differs by at most a factor of $2A$ from $%
\delta_{Q(S)} + \lvert x_l - x(S)\rvert $, which in turn differs by at most a factor of $%
2$ from $\lvert x_l-x(S)\rvert $.

Therefore, 
\begin{equation*}
\frac{1}{4} A^{-1} \lvert x_l - x(S)\rvert  \leq \delta_{Q_l} \leq 4A \lvert x_l-x(S)\rvert ,\;\; %
\mbox{for each} \;l=1,2,\ldots,L.
\end{equation*}

The fourth bullet point of Lemma \ref{lem5} now gives 
\begin{equation*}
\delta_{Q_{L_* j}} \leq (4A)^2 \cdot(1-A^{-3})^{j-i} \delta_{Q_{L_*
i}} \;\; \mbox{for} \; 1 \leq L_* i \leq L_* j \leq L.
\end{equation*}

Since also $\delta_{Q_{l+1}}$ and $\delta_{Q_l}$ differ by at most a factor
of $64$ ($1 \leq l \leq L-1$), and since $1 \leq L_* \leq A^3$, it follows
that 
\begin{equation*}
\delta_{Q_k} \leq C(A) \cdot (1-c(A))^{k-l} \delta_{Q_l} \;\; %
\mbox{for} \; 1 \leq l \leq k \leq L.
\end{equation*}
\end{proof}

Fix a cluster $S$. For each $Q \in CZ$ such that $Q \cap H(S) \neq \emptyset$
we fix a point $x(Q,S) \in Q \cap H(S)$.

\begin{lem}
\label{lem7} Given $R \geq \delta_{Q(S)}$, there are at most $C A^{2n}
\left(R/\delta_{Q(S)} \right)^n$ distinct cubes $Q \in CZ$ such that $Q \cap
H(S) \neq \emptyset$ and $\lvert x(Q,S) - x(S)\rvert  \leq R$.
\end{lem}

\begin{proof}
For each $Q$ as in the statement of the lemma, we know from Lemma \ref{lem2}
that 
$$\delta_Q \leq A \cdot \left[ \lvert x(Q,S)-x(S)\rvert  + \delta_{Q(S)}\right] \leq
A \cdot[2R],$$ 
and therefore $Q \subset B(x(Q,S),CAR) \subset
B(x(S),C^{\prime }AR)$.

On the other hand, the CZ cubes are pairwise disjoint, and each CZ cube such that $Q \cap H(S) \neq \emptyset$ has
volume at least $\left( 2A\right) ^{-n}\delta _{Q(S)}^{n}$, by Lemma \ref%
{lem2}. The conclusion of Lemma \ref{lem7} follows at once.
\end{proof}

We say that $Q \in CZ$ is \textit{privileged} for the cluster $S$ (or \emph{$S$-privileged}), provided that $%
Q \cap H(S) \neq \emptyset$ and $\lvert x(Q,S)-x(S)\rvert  \leq \delta_{Q(S)}$.
According to Lemma \ref{lem7}, there are at most $C A^{2n}$ privileged cubes
for a given cluster $S$. 

Moreover, Lemma \ref{lem7} shows that, if there are CZ
cubes $Q$ such that 
\begin{equation}Q\cap H(S)\neq \emptyset \; \mbox{and} \; \lvert x(Q,S) - x(S)\rvert  > \delta_{Q(S)}, \label{fef2}
\end{equation}
then there exists $\hat{Q}_{S}\in CZ$ such that $\hat{Q}_{S}\cap H(S)\neq \emptyset $, $\lvert x(\hat{Q}%
_{S},S)-x(S)\rvert >\delta _{Q(S)}$, and $\lvert x(\hat{Q}_{S},S)-x(S)\rvert  \leq \lvert x(Q,S)-x(S)\rvert $ for any $Q\in CZ$ for which \eqref{fef2} holds.

For each such cluster $S$, we pick such a $\hat{Q}_S$.

\begin{lem}
\label{lem8} Let $S$ be a cluster, and let $Q \in CZ$. Suppose $Q \cap H(S)
\neq \emptyset$, and suppose $Q$ is not privileged for $S$. Then there
exists a finite sequence of cubes $Q_1,Q_2,\ldots,Q_L \in CZ$ with the
following properties:

\begin{itemize}
\item $Q_1 = Q$ and $Q_L = \hat{Q}_S$.

\item $(1+c_G) Q_l \cap (1+c_G) Q_{l+1} \neq \emptyset$ for all $%
l=1,2,\ldots,L-1$.

\item $\delta_{Q_k} \leq C(A) \cdot (1-c(A))^{k -l} \delta_{Q_l}$
for $1 \leq l \leq k \leq L$; here, $0<c(A)<1$ and $C(A)>0$ are
constants depending only on $A$ and on the dimension $n$.
\end{itemize}
\end{lem}

\begin{proof}
Lemma \ref{lem6} applies, with $Q$ as in the present lemma, $Q^{\prime }= 
\hat{Q}_S$, $x=x(Q,S)$, $x^{\prime }=x(\hat{Q}_S,S)$.
\end{proof}

\begin{lem}
\label{lem9} Let $Q,Q^{\prime }\in CZ$ be interstellar cubes, and suppose $%
(1+c_G)Q \cap (1+c_G)Q^{\prime }\neq \emptyset$. Then there exists a cluster 
$S$ such that: $(1+c_G) Q \subset H(S)$ and $(1+c_G)Q^{\prime }\subset H(S)$%
; and for any cluster $S^{\prime }\neq S$, the cubes $(1+c_G) Q$ and $%
(1+c_G)Q^{\prime }$ are both disjoint from $H(S^{\prime })$.
\end{lem}

\begin{proof}
Immediate from Lemmas \ref{lem1} and \ref{lem3}.
\end{proof}

\begin{lem}
\label{lem10} Let $Q \in CZ$. Then there exists a finite sequence $%
Q_1,Q_2,\ldots,Q_L$ of $CZ$ cubes, such that

\begin{itemize}
\item $Q_1 = Q$ and $Q_L$ is a keystone cube.

\item $(1+c_G) Q_l \cap (1+c_G)Q_{l+1} \neq \emptyset$ for $l=1,2,\ldots,L-1$.

\item $\delta_{Q_k} \leq C(A) \cdot(1-c(A))^{k-l} \delta_{Q_l}$ for $%
1 \leq l \leq k \leq L$.
\end{itemize}
\end{lem}

\begin{proof}
Let $x \in \R^n$, and let $Q^x$ be the CZ cube containing $x$. Then any CZ cube that meets $(1+c_G)Q^x$ is a dyadic cube of sidelength at least $\delta_{Q^x}/64$. Hence, $x$ has a neighborhood that meets only finitely many CZ cubes. Consequently, every compact set meets only finitely many CZ cubes.

If $Q$ is a keystone cube, then the conclusion of the lemma holds with $L=1$ and $Q_1 = Q$. 

Suppose $Q$ is not a keystone cube. Then there exist cubes $Q' \in CZ$ with
\begin{equation} \label{hum}
Q' \cap 100Q \neq \emptyset \; \mbox{and} \; \delta_{Q'} \leq \frac{1}{2}\delta_Q.
\end{equation} 
There are only finitely many such $Q'$. Pick $Q^1 \in CZ$ that satisfies \eqref{hum} and has minimal distance to $Q$ among cubes satisfying \eqref{hum}. Let $\mathfrak{s} : [0,1] \rightarrow \R^n$ be an affine map with
\begin{equation} \label{eq_49}
\mathfrak{s}(0) \in \cl(Q), \; \mathfrak{s}(1) \in \cl(Q^1) \; \mbox{and} \; \lvert \mathfrak{s}(1) - \mathfrak{s}(0) \rvert = \dist(Q^1,Q).
\end{equation}
Since $Q^1$ meets $100Q$ and the point $\mathfrak{s}(1) \in \cl(Q^1)$ has a minimal distance to $Q$, it follows that $\mathfrak{s}(1) \in \cl(100Q)$. Also, since $\mathfrak{s}(0) \in \cl(Q)$, we have $\mathfrak{s}((0,1)) \subset 100 Q$.

The bounded set $\mathfrak{s}((0,1))$ meets only finitely many CZ cubes $Q^{1,1},\ldots,Q^{1,K}$. Thus, 
\begin{align}
& Q^{1,k} \cap 100 Q \neq \emptyset, \; \mbox{and} \label{eq_50} \\
& \dist(Q^{1,k},Q) < \dist(Q^1,Q) \;\; \mbox{for} \; k=1,\ldots,K. \label{eq_51}
\end{align}
(Here, we use \eqref{eq_49} to prove \eqref{eq_51}.) Reordering the cubes $Q^{1,1},\ldots,Q^{1,K}$ if necessary, we can arrange that
\begin{equation}
\begin{aligned}
 &(1+c_G) Q^{1,1} \cap (1+c_G) Q \neq \emptyset \; \mbox{and} \; (1+c_G) Q^{1,K} \cap (1+c_G) Q^1 \neq \emptyset; \; \mbox{and} \\
&(1+c_G) Q^{1,k} \cap (1+c_G) Q^{1,k+1} \neq \emptyset \;\; \mbox{for} \; k=1,\ldots, K-1.
\end{aligned}
\label{eq_53}
\end{equation}

From \eqref{eq_51} and the definition of $Q^1$, each $Q^{1,k}$ must not satisfy \eqref{hum}, hence $\delta_{Q^{1,k}} \geq \delta_Q$, thanks to \eqref{eq_50}. Since each $Q^{1,k}$ meets $100Q$, good geometry shows that $\delta_{Q^{1,k}} \leq C' \delta_Q$, as in the proof of Lemma \ref{moregg}. Therefore,
\begin{equation}
\label{eq_54} \delta_{Q^{1,k}} \simeq \delta_Q \; \mbox{for} \; k=1,\ldots,K, \;\; \mbox{and} \; K \leq C.
\end{equation}

We call the cube $Q^1 \in CZ$ a \emph{junior partner} of $%
Q $; any sequence $( Q^{1,1}, \ldots, Q^{1,K})$ of CZ cubes that satisfies \eqref{eq_53} and \eqref{eq_54} is said to \emph{join} $Q$ with $Q^1$. Since $Q^1$ meets $100 Q$ and satisfies $\delta_{Q^1} \leq \frac{1}{2}\delta_Q$, we have
\begin{equation*}
C'' \cdot Q^1 \subset C'' \cdot Q \; \text{ whenever } Q^1 \text{ is a junior
partner to } Q.
\end{equation*}
Now, either $Q^1$ is a keystone cube or it has a junior partner 
$Q^2$. In the latter case, either $Q^2$ is a
keystone cube or it has a junior partner $Q^3$.
Continue in this way, either forever, or until we arrive at a keystone cube.

If the above process continued indefinitely, then we would have a sequence
of $CZ$ cubes $Q^1,Q^2,Q^3, \ldots$ with each $Q^{j+1}$ being a junior
partner to $Q^{j}$. That would imply that $\delta_{Q^{j+1}}\leq \frac{1}{%
2}\delta_{Q^{j}}$ and $C'' \cdot Q^{j+1} \subset C'' \cdot Q^{j}$ for each $j$.
Thus, the cubes $Q^j$ would shrink to a single point as $j
\rightarrow \infty$; however, this contradicts the fact that every point has
a neighborhood that meets only finitely many CZ cubes. Thus, the above process of successively
passing to junior partners must stop after finitely many steps. Accordingly,
starting from any $Q \in CZ$, we obtain a finite sequence $Q^1,Q^2,\ldots,
Q^{J}$ of $CZ$ cubes such that $Q^{j+1}$ is a junior partner of $%
Q^j$ for $1\leq j \leq J-1$, and $Q^{J}$ is a keystone
cube. We now join each $Q^{j}$ with $Q^{j+1}$ through a sequence of $CZ$
cubes. Concatenating these
sequences, we obtain a sequence satisfying the conclusions of Lemma \ref{lem10}.
\end{proof}

\begin{lem}
\label{lem11} There exists a set $CZ_{\spec}$, consisting of at most $%
C(A)\cdot N$ distinct CZ cubes, for which the following holds. 
We can associate to each $Q \in CZ$ a finite sequence $\cS_Q =
(Q_1,Q_2,\ldots,Q_L)$ of CZ cubes, with the following properties.

\begin{itemize}
\item $Q_1 = Q$ and $Q_L$ is a keystone cube.

\item $(1+c_G) Q_l \cap (1+c_G)Q_{l+1} \neq \emptyset$ for $l=1,2,\ldots,L-1$.%

\item $\delta_{Q_k} \leq C(A) \cdot(1-c(A))^{k-l} \delta_{Q_l}$ for $%
1 \leq l \leq k \leq L$.

\item Let $Q,Q^{\prime }\in
CZ \setminus CZ_{\spec}$, and suppose $(1+c_G) Q \cap (1 +
c_G)Q^{\prime }\neq \emptyset$. Let $\cS_Q=(Q_1,\ldots,Q_L) $ and $\cS%
_{Q^{\prime }} = (Q^{\prime }_1,\ldots,Q^{\prime }_{L^{\prime
}})$ be the finite sequences of CZ cubes associated to $Q$ and to $Q^{\prime
} $, respectively. Then $Q_L = Q^{\prime }_{L^{\prime }}$.
\end{itemize}
\end{lem}

\begin{proof}
First, we define the collection $CZ_{\spec}$. It consists of all
non-interstellar cubes $Q\in CZ$, together with all $Q\in CZ$ that are
privileged for some cluster $S$. If $Q \in CZ_{\spec}$, then we say that $Q$ is ``special.'' We have seen that
there are at most $C(A)N$ non-interstellar cubes and at most $CN$ distinct clusters (see Lemmas \ref{lemA}, \ref{lemB}). 
Since there are at most $C(A)$ privileged cubes for each given cluster,
it follows that $CZ_{\spec}$ consists of at most $C(A)N$ distinct
CZ cubes.

Next, we define the sequence $\cS_Q = (Q_1,Q_2,\ldots Q_L)$ for each CZ cube 
$Q$. If $Q \in CZ_{\spec}$, we just pick any finite sequence as in
Lemma \ref{lem10}. Then $\cS_Q$ satisfies the first three bullet points in the
statement of Lemma \ref{lem11}. It remains to define the $\cS_Q$ for all $Q
\in CZ \setminus CZ_{\spec}$, and to prove that our $\cS_Q$ and $%
CZ_{\spec}$ have the properties asserted in Lemma \ref{lem11}.

For each cluster $S$ such that there exist cubes $Q \in CZ$
that are not privileged for $S$ but that meet $H(S)$, we have picked out a cube $\hat{Q}_S$ in the discussion
following Lemma \ref{lem7}. Applying Lemma \ref{lem10} to $\hat{Q}_S$, we
obtain a finite sequence $\hat{Q}^1_S,\hat{Q}^2_S,\ldots,\hat{Q}^{L(S)}_S$
of CZ cubes such that

\begin{itemize}
\item $\hat{Q}^1_S = \hat{Q}_S$, $\hat{Q}^{L(S)}_S$ is a keystone cube.

\item $(1+c_G) \hat{Q}^l_S \cap (1+c_G) \hat{Q}^{l+1}_S \neq \emptyset$ for $%
l=1,2,\ldots,L(S)-1$.

\item $\delta _{\hat{Q}_{S}^{\nu }}\leq C\left( A\right) \cdot (1-c\left(
A\right) )^{\nu -\mu }\delta _{\hat{Q}_{S}^{\mu }}$ for $1\leq \mu \leq \nu
\leq L(S)$.
\end{itemize}

For each $Q \in CZ$ that is not privileged for $S$ but that meets $H(S)$, we define a sequence $%
Q_1,Q_2,\ldots,Q_L$ as in Lemma \ref{lem8}. Thus, $Q_1=Q$, $Q_L=\hat{Q}_S$, $%
(1+c_G)Q_l \cap (1+c_G) Q_{l+1} \neq \emptyset$ for $1 \leq l \leq L-1$; and 
$\delta_{Q_k} \leq C(A) \cdot (1-c(A))^{k-l} \delta_{Q_l}$ for $1
\leq l \leq k \leq L$.

Unless $Q$ is special (in which case we have already defined $\cS(Q)$), we
then define $\cS(Q)$ to be the sequence 
\begin{equation*}
\cS(Q)=\left( Q_{1},Q_{2},\ldots ,Q_{L},\hat{Q}_{S}^{1},\hat{Q}%
_{S}^{2},\ldots ,\hat{Q}_{S}^{L(S)}\right) .
\end{equation*}%
This is well-defined, since each $Q\in CZ\backslash CZ_{\spec}$ is
as above for one and only one $S$ (see Lemma \ref{lem9} with $Q^{\prime }=Q$%
).

Since $Q_L = \hat{Q}_S = \hat{Q}^1_S$, one checks easily that $\cS(Q)$
satisfies the first three bullet points in the statement of Lemma \ref{lem11}.

Moreover, if $Q$ and $Q^{\prime }$ are any two non-special $CZ$ cubes that
meet $H(S)$, then the finite sequences $\cS(Q)$ and $\cS(Q^{\prime })$ both
end with the finite sequence $\hat{Q}^1_S,\hat{Q}^1_S,\ldots,\hat{Q}%
^{L(S)}_S $. In particular, the sequences $\cS(Q)$ and $\cS(Q^{\prime })$
both end with the same cube, namely $\hat{Q}^{L(S)}_S$.

The above observation applies to any $Q ,Q^{\prime }\in CZ \setminus CZ_{%
\spec}$ such that $(1+c_G)Q \cap (1+c_G)Q^{\prime }\neq \emptyset$.
Indeed, Lemma \ref{lem9} gives a cluster $S$ such that $(1+c_G)Q,(1+c_G)Q^{%
\prime }\subset H(S)$. The cluster $S$ admits non-special cubes $Q^{\prime
\prime }$ that intersect $H(S)$; indeed, we may take $Q^{\prime \prime }= Q$
or $Q^{\prime }$. Hence $\cS(Q)$ and $\cS(Q^{\prime })$ end with the same
cube, by the observation in the preceding paragraph.

The conclusions of Lemma \ref{lem11} are now obvious.
\end{proof}

\begin{proof}[Proof of Proposition \protect\ref{pathassignment}]
We simply take $A$ in Lemma \ref{lem11} to be a large enough constant
determined by the dimension $n$.
\end{proof}

\section{Paths to Keystone Cubes II}\label{sec_PTKC2}
\setcounter{equation}{0}

We place ourselves back in the setting of section \ref{sec_cz}.  In particular, $\CZ$ is a dyadic decomposition of the cube $Q^\circ = (0,1]^n$. We define the collection of keystone cubes for $\CZ$ by
\begin{equation}\CZ_{\ky} = \bigl\{Q \in \CZ : \delta_{Q'} \geq \delta_Q \; \mbox{for every} \; Q' \in \CZ \; \mbox{that meets} \; 100Q \bigr\}, \label{keystonecube2} \end{equation}
which will also be denoted by
$$\CZ_{\ky} = \{Q^\sharp_1,\ldots,Q^\sharp_{\mu_{\max}}\}.$$ 
\begin{lem}\label{moregeom}
For each $\mu=1,\ldots,\mu_{\max}$, there are at most $C'$ indices $\mu' \in \{1,\ldots,\mu_{\max}\}$ with $10 Q^\sharp_{\mu'} \cap 10 Q^\sharp_\mu \neq \emptyset$.
\end{lem}
\begin{proof}
Suppose that $Q^\sharp, \hQ^\sharp \in \CZ_{\ky}$ satisfy $10 Q^\sharp \cap 10\hQ^\sharp \neq \emptyset$. Without loss of generality, we may assume that $\delta_{Q^\sharp} \geq \delta_{\hQ^\sharp}$. Therefore, $100 Q^\sharp \cap \hQ^\sharp \neq \emptyset$. By the definition of keystone cubes, we have $\delta_{\hQ^\sharp} \geq \delta_{Q^\sharp}$. Thus, $\delta_{Q^\sharp} = \delta_{\hQ^\sharp}$ whenever $10 Q^\sharp \cap 10\hQ^\sharp \neq \emptyset$. The conclusion of Lemma \ref{moregeom} follows immediately.
\end{proof}

By applying Proposition \ref{pathassignment} to the current setting we prove the following.

\begin{prop}
\label{keygeom} To each cube $Q_\nu \in \CZ$ ($\nu=1,\ldots,\nu_{\max}$), we can assign a finite sequence $%
\cS_\nu=(Q_{\nu,1},Q_{\nu,2},\ldots,Q_{\nu,L_\nu})$ of cubes from $\CZ$, such that the following properties hold.

\begin{itemize}

\item[\textbf{(K1)}] $Q_{\nu,1}=Q_\nu$ and $Q_{\nu,L_\nu}$ is a keystone cube for $\CZ$; $Q_{\nu,l} \leftrightarrow Q_{\nu,l+1}$ $\;\;$ ($1 \leq l \leq L_\nu-1$); and 
$$\delta_{Q_{\nu,k}} \leq C\cdot (1-c)^{k-l}\delta_{Q_{\nu,l}} \;\;\; (1 \leq l \leq k \leq L_\nu).$$

\item[\textbf{(K2)}]  Let $\mathcal{K}:\CZ \rightarrow \CZ$ be defined by 
$\key(Q_\nu)=Q_{\nu,L_\nu}$. Then 
\begin{equation*}
\# \left\{ (\nu, \nu') \in \{1,\ldots,\nu_{\max}\}^2: Q_\nu \leftrightarrow Q_{\nu'} \; \mbox{and} \;
\key(Q_\nu) \neq \key(Q_{\nu'}) \right\} \leq C\cdot N.
\end{equation*}

\item[\textbf{(K3)}] $\key(Q_\nu) = Q_{\nu}$ for any $Q_\nu \in \CZ_{\ky}$.

\end{itemize}
\end{prop}
\begin{proof}
First we embed $\CZ$ (a dyadic decomposition of $Q^\circ$) into a dyadic decomposition $CZ^+$ of the whole $\R^n$. 

Define the collection of dyadic cubes
\begin{equation}
\label{ocz} 
\begin{aligned}
\overline{CZ} =&\; \big[ \; \{Q \subset \R^n \; \mbox{dyadic} : \; \delta_Q = 1, \; 3Q^+ \supset Q^\circ\} \\
 & \cup \{Q \subset \R^n \; \mbox{dyadic} : \; \delta_Q \geq 2, \; 3Q^+ \supset Q^\circ, \; 3Q \not\supset Q^\circ\} \;\big] \setminus \{Q^\circ\}.
\end{aligned}
\end{equation}
We establish the following claims.

\textbf{Claim 1:} $\overline{CZ} $ partitions $\R^n \setminus Q^\circ$ into dyadic cubes.

\textbf{Claim 2:} If $Q \in \overline{CZ} $ then $9 Q \supset Q^\circ$.

\textbf{Claim 3:} If $Q,Q' \in \overline{CZ}$ satisfy $\cl(Q) \cap \cl(Q') \neq \emptyset$, then $1/2 \leq \delta_Q/\delta_{Q'} \leq 2$.

\textbf{Claim 4:} If $Q \subset \R^n$ is dyadic and satisfies $\delta_Q =1$, $\cl(Q) \cap \partial Q^\circ \neq \emptyset$ and $Q \neq Q^\circ$, then $Q \in \overline{CZ}$.

\underline{Proof of Claim 1:} For each $x \in \R^n \setminus Q^\circ$, let $Q \subset \R^n$ be the smallest dyadic cube with $x \in Q$, $3Q^+  \supset Q^\circ$ and $\delta_{Q} \geq 1$. Then $Q \neq Q^\circ$, since $x \notin Q^\circ$. If $\delta_{Q} = 1$, then $Q \in \overline{CZ}$. On the other hand, if $\delta_{Q} \geq 2$, then $3Q \not\supset Q^\circ$ since $Q$ is minimal, hence also $Q \in \overline{CZ}$. In either case, $Q \in \overline{CZ}$ and $x \in Q$. Thus, $\overline{CZ}$ covers $\R^n \setminus Q^\circ$.

We now prove that the collection $\overline{CZ}$ is pairwise disjoint. For the sake of contradiction, suppose that $Q,Q' \in \overline{CZ}$ are distinct with $Q \cap Q' \neq \emptyset$. Since $Q,Q'$ are dyadic, either $Q \subsetneq Q'$ or $Q' \subsetneq Q$. Without loss of generality, $Q \subsetneq Q'$. Therefore, $Q^+ \subset  Q'$. It follows that $\delta_{Q'} \geq 2 \delta_Q \geq 2$. Also, since $Q \in \overline{CZ}$, we have $3Q^+ \supset Q^\circ$, hence $3Q' \supset Q^\circ$. Thus, $Q' \notin \overline{CZ}$, yielding the desired contradiction.

Obviously, each $Q \in \overline{CZ}$ is disjoint from $Q^\circ$, which completes the proof of Claim 1.

\underline{Proof of Claim 2:} Let $Q \in \overline{CZ}$. Then $9Q \supset 3 Q^+ \supset Q^\circ$.

\underline{Proof of Claim 3:} For the sake of contradiction, suppose that $Q, Q' \in \overline{CZ}$ satisfy $\cl(Q) \cap \cl(Q') \neq \emptyset$ and $\delta_{Q'} \geq 4 \delta_Q$. It follows that $3Q^+ \subset 3Q'$. Therefore, since $Q^\circ \subset 3Q^+$, we have $Q^\circ \subset 3Q'$. Note that $\delta_{Q'} \geq 4$. Therefore, $Q' \notin \overline{CZ}$, yielding the desired contradiction.

\underline{Proof of Claim 4:} For any cube $Q$ that satisfies $\cl(Q) \cap \partial Q^\circ \neq \emptyset$ and $\delta_Q \geq 1$, we have $3Q^+ \supset Q^\circ$. If $Q$ is also dyadic with $\delta_Q = 1$ and $Q \neq Q^\circ$, then we have $Q \in \overline{CZ}$ by definition.

Let us define 
$$CZ^+ = \CZ \cup \overline{CZ}.$$ From Claim 1 and the fact that $\CZ$ is a dyadic decomposition of $Q^\circ$, it follows that $CZ^+$ is a dyadic decomposition of $\R^n$.

The following is immediate from Claims 3,4, the Good Geometry of $\CZ$ (Lemma \ref{gg}) and the last bullet point in Lemma \ref{moregg}.
$$\mbox{If} \; Q, Q' \in CZ^+\; \mbox{satisfy} \; \cl(Q) \cap \cl(Q') \neq \emptyset, \; \mbox{then} \; \delta_Q/\delta_{Q'} \in [1/64,64].$$
Therefore,
\begin{equation} \label{gg0} 
\mbox{if} \; Q,Q' \in CZ^+ \; \mbox{satisfy} \; (1+10^{-5})Q \cap (1 + 10^{-5})Q' \neq \emptyset,\; \mbox{then} \; \cl(Q) \cap \cl(Q') \neq \emptyset.
\end{equation}
Thus, by the last two lines,
\begin{equation} \label{gga} \mbox{if} \; Q, Q' \in CZ^+\; \mbox{satisfy} \; (1+10^{-5})Q \cap (1 + 10^{-5})Q' \neq \emptyset, \; \mbox{then} \; \delta_Q/\delta_{Q'} \in [1/64,64].
\end{equation}
It follows that $CZ^+$ satisfies the first bullet point at the beginning of section \ref{PTKC} with $c_G := 10^{-6}$.

From (\ref{sec_cz}.\ref{nonempty}), we have $\#(E \cap 9Q) \geq 2$ for every $Q \in \CZ$. Also, from Claim 3 and $E \subset Q^\circ$ we have $\#(E \cap 9Q) = \#(E) \geq 2$ for every $Q \in \overline{CZ}$. Thus, $CZ^+$ and $E \subset \R^n$ satisfy the second and third bullet points at the beginning of section \ref{PTKC}.

Define the keystone cubes for $CZ^+$ as in section \ref{PTKC}:
$$CZ^+_{\ky} = \{Q \in CZ^+ : \delta_{Q'} \geq \delta_Q \; \mbox{for every} \; Q' \in CZ^+ \; \mbox{that meets} \; 100Q\}.$$
Since each $Q \in \overline{CZ}$ satisfies $\delta_Q \geq 1$ and $3Q^+ \supset Q^\circ$, there must exist cubes $Q' \in  \CZ$ with $Q' \cap 100Q \neq \emptyset$ and $\delta_{Q'} \leq \frac{1}{2} \delta_Q$. (Here, we also use the fact that $\CZ$ partitions $Q^\circ$ into cubes of sidelength $\leq \frac{1}{2}$.) Consequently,
\begin{equation}
CZ^+_{\ky} \subset \CZ_{\ky}. \label{keycontain}
\end{equation}

From Proposition \ref{pathassignment}, we find a subcollection $CZ^+_\spec \subset CZ^+$, and an assignment to each $Q_\nu \in \CZ$ ($\nu=1,\ldots,\nu_{\max}$) of a sequence $\widehat{\cS}_{\nu} = (\hQ_{\nu,1},\ldots,\hQ_{\nu,L_\nu})$ of cubes, such that the following properties hold.
\begin{itemize}
\item[(a)] $CZ^+_\spec$ contains at most $C \cdot N$ distinct cubes.
\item[(b)] $\hQ_{\nu,1}=Q_\nu$ and $\hQ_{\nu,L_\nu} \in CZ^+_{\ky}$; each $\hQ_{\nu,k}$ belongs to $CZ^+$;
\begin{equation} \label{label0}
(1+10^{-6})\hQ_{\nu,l} \cap (1+10^{-6}) \hQ_{\nu,l+1} \neq \emptyset \;\; \mbox{for every} \; 1 \leq l < L_\nu; \; \mbox{and}
\end{equation}
\begin{equation} \label{label1}
\delta_{\hQ_{\nu,k}} \leq C \cdot(1-c)^{k - l}\delta_{\hQ_{\nu,l}} \; 
\mbox{for} \; 1 \leq l \leq k \leq L_\nu.
\end{equation}

\item[(c)] If $Q_\nu,Q_{\nu'} \in \CZ \setminus CZ_{\spec}^+$ satisfy $(1+10^{-6}) Q_\nu \cap (1+10^{-6}) Q_{\nu'} \neq \emptyset$, then $\hQ_{\nu,L_\nu} = \hQ_{\nu',L_{\nu'}}$.
\end{itemize}

Suppose that $Q_\nu \in \CZ \setminus \CZ_{\ky}$ is such that $\widehat{\mathcal{S}}_\nu$ contains at least one cube from $\overline{CZ}$. In this case, we define $a(\nu)$ to be the first index $a \in \{1,\ldots,L_\nu\}$ such that $\hQ_{\nu,a} \in \overline{CZ}$, and define $b(\nu)$ to be the last index $b \in \{1,\ldots,L_\nu\}$ such that $\hQ_{\nu,b} \in \overline{CZ}$.

Note that $\hQ_{\nu,1} = Q_\nu \in \CZ$, and $\hQ_{\nu,L_\nu} \in \CZ$, thanks to (b) and \eqref{keycontain}. Therefore, $1 < a(\nu) \leq b(\nu) < L_\nu$. Since $\hQ_{\nu,a(\nu)} \in \overline{CZ}$ and $\hQ_{\nu,a(\nu)-1} \in \CZ$, by \eqref{gg0} and \eqref{label0} it follows that the cube $\hQ_{\nu,a(\nu)-1}$ must touch the boundary of $Q^\circ$. Likewise, $\hQ_{\nu,b(\nu)+1} \in \CZ$ must touch the boundary of $Q^\circ$.

In view of the last bullet point in Lemma \ref{moregg}, we may choose a sequence of $\CZ$ cubes $\mathcal{E}_\nu = (Q_{\nu,a(\nu)-1},\ldots,Q_{\nu,s(\nu)+1})$ of bounded length (i.e., $s(\nu) - a(\nu) \leq C$) that connects $\hQ_{\nu,a(\nu)-1}$ with $\hQ_{\nu,b(\nu)+1}$ such that each of the $Q_{\nu,k}$ touches the boundary of $Q^\circ$. That is, we can arrange for
\begin{align}
&Q_{\nu,a(\nu)-1} = \hQ_{\nu,a(\nu)-1} \; \mbox{and} \; Q_{\nu,s(\nu)+1} = \hQ_{\nu,b(\nu)+1}; \label{label2} \\
&(1+10^{-6})Q_{\nu,k} \cap (1+10^{-6}) Q_{\nu,k+1} \neq \emptyset \; \mbox{for all} \; a(\nu)-1 \leq k \leq s(\nu); \label{label3}\\
&Q_{\nu,k} \in \CZ \; \mbox{for each} \; k=a(\nu)-1, \ldots, s(\nu)+1;\; \mbox{and} \label{label4}\\
&s(\nu) - a(\nu) \leq C. \label{label5}
\end{align}
Define the sequence of $\CZ$ cubes
$$\cS_\nu = (\hQ_{\nu,1},\hQ_{\nu,2},\ldots ,\hQ_{\nu,a(\nu)-2},Q_{\nu,a(\nu)-1},\ldots,Q_{\nu,s(\nu)+1}, \hQ_{\nu,b(\nu)+2}, \ldots, \hQ_{\nu,L_\nu}).$$ 
Note that the starting cube $\hQ_{\nu,1}$ and the terminating cube $\hQ_{\nu,L_\nu}$ of the sequence $\mathcal{S}_\nu$ have not been changed, thanks to \eqref{label2}. Also, note that consecutive cubes from $\cS_\nu$ are neighbors (they have intersecting closures), due to \eqref{gg0},\eqref{label0},\eqref{label2},\eqref{label3}. Thanks to \eqref{label1},\eqref{label5} and the Good Geometry of the $\CZ$ cubes, the sequence $\cS_\nu$ satisfies the inequality from \textbf{(K1)}. Thus, we have defined $\cS_\nu$ for each $Q_\nu \in \CZ \setminus \CZ_{\ky}$ such that $\widehat{\cS}_\nu$ contains cubes from $\overline{CZ}$, and we have proven that $\cS_\nu$ satisfies \textbf{(K1)}.

For any cube $Q_\nu \in \CZ \setminus \CZ_{\ky}$ such that the path $\widehat{\mathcal{S}}_\nu$ contains no cubes from $\overline{CZ}$, we simply take $\mathcal{S}_\nu = \widehat{\cS}_\nu$. Then \textbf{(K1)} holds in view of (b).

For any cube $Q_\nu \in \CZ_{\ky}$, we take $L_\nu = 1$ and $\cS_\nu = (Q_{\nu,1}) = (Q_\nu)$. Clearly, \textbf{(K1)} holds in this case.

Thus, we have defined $\cS_\nu$ for each $\nu$, and established \textbf{(K1)} in all cases. Also, \textbf{(K3)} clearly holds. Thanks to Proposition \ref{pathassignment} and the fact that each $Q_\nu$ has boundedly many neighbors, we obtain
\begin{align*}
& \;\;\;\; \# \left\{ (Q_\nu,Q_{\nu'}) \in \left[\CZ \setminus \CZ_{\ky}\right]^2 : \; Q_\nu \leftrightarrow Q_{\nu'}, \;
Q_{\nu,L_\nu} \neq Q_{\nu',L_{\nu'}} \right\} \leq C \cdot \#(CZ_{\spec}^+).
\end{align*}
Note that $\#(CZ_{\spec}^+) \leq C \cdot N$ by definition. Similarly, we have
$$\# \left\{ (Q_\nu,Q_{\nu'}) \in \CZ_{\ky} \times \CZ: Q_\nu \leftrightarrow Q_{\nu'}\right\} \leq C \cdot \#(\CZ_{\ky}).$$
Thanks to (\ref{sec_cz}.\ref{nonempty}), to each $Q \in \CZ_{\ky}$ we may assign a point $y_Q \in E \cap 10 Q$. Lemma \ref{moregeom} shows that the preimage of each $y \in E$ has cardinality at most a universal constant. Thus, $\#(\CZ_{\ky}) \leq C \cdot N$, which together with the previous two lines establishes \textbf{(K2)}.

This completes the proof of Proposition \ref{keygeom}.
\end{proof}

Recall that $\{Q^\sharp_1,\ldots,Q^\sharp_{\mu_{\max}}\}$ denotes the keystone cubes. Using the map $\key : \CZ \rightarrow \CZ_{\ky}$ from Proposition \ref{keygeom} we produce a map on indices $\kappa : \{1,\ldots,\nu_{\max}\} \rightarrow \{1,\ldots,\mu_{\max}\}$ defined by $\key(Q_\nu) = Q^\sharp_{\kappa(\nu)}$.

\section{Representatives}\label{sec_rep}
\setcounter{equation}{0}
Since $Q_\nu$ is OK, the subset $E \cap 3Q_\nu$ lies on the zero set of a nondegenerate smooth function with small norm. Using this fact we prove the next result.

\begin{lem}
\label{pointprop}
For each $\nu=1,\ldots,\nu_{\max}$, there exists $\hx_\nu \in \frac{1}{2} Q_\nu$, such that $\dist(\widehat{x}_\nu,E) \geq c' \delta_{Q_\nu}$.
\end{lem}

\begin{proof}
Fix $ \nu \in \{1,\ldots,\nu_{\max}\}$. If $ E \cap (1/4)Q_\nu = \emptyset$, then we take $\hx_\nu$ to be the center of $Q_\nu$ and reach the desired conclusion. Thus, we may assume that $E \cap (1/4)Q_\nu \neq \emptyset$. Let $y_\nu \in E \cap (1/4)Q_\nu$ be fixed. 

If $\#(E \cap 3Q_\nu) \leq 1$ then the conclusion of the lemma is obvious. Thus we may assume that $\#(E \cap 3Q_\nu) \geq 2$. Thus, because $Q_\nu$ is OK there exists an $(\oA,y_\nu,\epsilon_0,30\delta_{Q_\nu})$-basis for $\sigma(y_\nu ,E \cap 3{Q_\nu})$, for some $\oA < \cA$. Since $\oA < \cA$ and the empty set is maximal under the order $<$, we may choose a multi-index $\alpha_0 \in \oA$. From the definition of an $(\oA,y_\nu,\epsilon_0,30\delta_{Q_\nu})$-basis (see \textbf{(B1)},\textbf{(B2)}), we have a polynomial
$$P_{\alpha_0} \in C\epsilon_0 \delta_{Q_\nu}^{n/p + |\alpha_0| - m} \cdot \sigma(y_\nu, E \cap 3{Q_\nu})\; \; \mbox{with} \;\; \partial^{\alpha_0} P_{\alpha_0}(y_\nu) = 1.$$ 
Choose $\alpha' \in \cM$ with
$$\lvert \partial^{\alpha'} P_{\alpha_0}(y_\nu) \rvert \delta_{Q_\nu}^{|\alpha'|} \geq \lvert \partial^{\beta} P_{\alpha_0}(y_\nu) \rvert \delta_{Q_\nu}^{|\beta|} \;\;\; \mbox{for all} \; \beta \in \cM.$$
Placing $\beta=\alpha_0$ in the inequality above, we obtain
$$\lvert \partial^{\alpha'} P_{\alpha_0}(y_\nu)\rvert \geq \lvert \partial^{\alpha_0} P_{\alpha_0}(y_\nu) \rvert \delta_{Q_\nu}^{|\alpha_0|-|\alpha'|} = \delta_{Q_\nu}^{|\alpha_0|-|\alpha'|}.$$ 
Define $P = \left[\partial^{\alpha'} P_{\alpha_0}(y_\nu) \right]^{-1} \cdot P_{\alpha_0}$. Then the above three lines imply that
\begin{align*}
& P \in C\epsilon_0 \delta_{Q_\nu}^{n/p + |\alpha'| - m} \cdot \sigma(y_\nu,E \cap 3{Q_\nu}); \; \partial^{\alpha'} P(y_\nu) = 1;\; \mbox{and} \\
& \lvert \partial^\beta P(y_\nu) \rvert \leq \delta_{Q_\nu}^{|\alpha'|-|\beta|} \quad \mbox{for all} \; \beta \in \cM. 
\end{align*}
By definition of $\sigma(y_\nu,E \cap 3{Q_\nu})$, there exists $\varphi \in \LR$ with 
\begin{subequations}
\begin{align}
\label{pp1} & \varphi = 0 \; \mbox{on} \; E \cap 3{Q_\nu}  \; \mbox{and} \; J_{y_\nu} (\varphi) = P; \\ 
& \label{pp2} \lvert \partial^{\beta} \varphi(y_\nu) \rvert \leq \delta_{Q_\nu}^{|\alpha'|-|\beta|} \;\; \mbox{for all} \; \beta \in \cM; \; \mbox{and} \\
& \label{pp3} \|\varphi\|_{\LR} \leq C \epsilon_0 \delta_{Q_\nu}^{n/p + |\alpha'| - m}. 
\end{align}
\end{subequations}

Applying Bernstein's inequality to the polynomial $P \in \cP$, we have
$$ \max_{B(y_\nu,\delta)} \lvert \partial^{\alpha'} P \rvert \lesssim \delta^{-|\alpha'|} \max_{B(y_\nu,\delta)} \lvert P\rvert \quad (\delta > 0).$$
Since $\partial^{\alpha'} P(y_\nu) = 1$, the left-hand side is bounded from below by $1$. Therefore, for every $\delta > 0$ there exists $x_\delta \in B(y_\nu,\delta)$ with $\lvert P(x_\delta) \rvert \geq c \cdot \delta^{|\alpha'|}.$ By the Sobolev inequality,
\begin{align*}\lvert (\varphi - P)(x_\delta)\rvert &= \lvert (\varphi - J_{y_\nu} (\varphi))(x_\delta)\rvert \lesssim \|\varphi\|_{\LR} \lvert x_\delta - y_\nu\rvert^{m-n/p} \\
& \lesssim \delta_{Q_\nu}^{n/p + |\alpha'| - m}\delta^{m-n/p}, \;\;\mbox{thanks to} \; \eqref{pp3} \; \mbox{and} \; |x_\delta - y_\nu| \leq \delta.
\end{align*}
Thus,
\begin{align*}
\lvert \varphi(x_\delta) \rvert \geq \lvert P(x_\delta) \rvert - \lvert (\varphi - P)(x_\delta) \rvert &\geq c \cdot \delta^{|\alpha'|} - C \delta_{Q_\nu}^{n/p + |\alpha'| - m} \delta^{m - n/p} \\
& = \delta^{|\alpha'|} \cdot \left[c - C \left(\frac{\delta}{\delta_{Q_\nu}} \right)^{m - |\alpha'| - n/p} \right].
\end{align*}
We now set $\delta = c_0 \delta_{Q_\nu}$, for some small universal constant $c_0<1/8$, so that $\lvert \varphi(x_\delta) \rvert \geq c'' \cdot \delta_{Q_\nu}^{|\alpha'|}$ and
$$ x_\delta \in B(y_\nu,\delta) \subset B(y_\nu,\delta_{Q_\nu}/8) \subset (1/2){Q_\nu}, \;\; \mbox{since} \;\; y_\nu \in (1/4){Q_\nu}.$$

Also, \eqref{pp2}, \eqref{pp3} and the Sobolev inequality imply that $\lvert \nabla \varphi \rvert \leq C \delta_{Q_\nu}^{|\alpha'|-1}$ on $3{Q_\nu}$. Since $\varphi=0$ on $E \cap 3{Q_\nu}$, we have
$$c'' \cdot \delta_{Q_\nu}^{|\alpha'|} \leq \lvert \varphi(x_\delta) \rvert = \lvert \varphi(x_\delta) - \varphi(x) \lvert \leq C \delta_{Q_\nu}^{|\alpha'|-1}\cdot \lvert x_\delta - x \rvert, \;\; \mbox{for every} \; x \in E \cap 3{Q_\nu}.$$
Hence, $\dist(x_\delta,E \cap 3{Q_\nu}) \geq c' \cdot \delta_{Q_\nu}$. Thus the conclusion of Lemma \ref{pointprop} holds with $\widehat{x}_\nu = x_\delta$.
\end{proof}

We have indexed the CZ cubes $\CZ$ and the subcollection of keystone cubes in an arbitrary manner.  Without loss of generality, we may put in place several new indexing assumptions. First, we may assume that $Q_1$ contains the point $z \in (1/8)Q^\circ$. If $Q_1$ happens to be keystone, we also assume that $Q^\sharp_1 = Q_1$. To summarize, we have
\begin{equation} \label{index_ass} z \in Q_1 \; \mbox{and if} \; Q_1 \; \mbox{is keystone, then} \; Q_1 = Q^\sharp_1.\end{equation}
We make no further assumptions on the indexing of $Q_\mu^\sharp$ or $Q_\nu$. We now define $\mu_{\min} = 2$ if $Q_1$ is keystone and $\mu_{\min} = 1$ otherwise. Thus, $\{Q^\sharp_{\mu_{\min}},\ldots,Q^\sharp_{\mu_{\max}}\} = \{Q_1^\sharp,\ldots,Q_{\mu_{\max}}^\sharp\} \setminus \{Q_1\}$.

For each $\nu=1,\ldots,\nu_{\max}$, we define the representative basepoint for the CZ cube $Q_\nu$:
\begin{equation*}
x_1 = z; \; \mbox{and}\; x_\nu = \hx_\nu \; \mbox{for} \; \nu = 2,\ldots, \nu_{\max}.
\end{equation*}
Similarly, for each  $\mu=1,\ldots,\mu_{\max}$, we define the representative basepoint for the keystone cube $Q^\sharp_\mu$:
\begin{equation*}
x^\sharp_\mu = x_\nu,\; \mbox{where} \; \nu \in \{1,\ldots,\nu_{\max}\} \; \mbox{is such that} \; Q^\sharp_\mu = Q_\nu.
\end{equation*}
Denote the collection of basepoints $E' := \{x_1,\ldots,x_{\nu_{\max}}\}$, and denote the collection of keystone basepoints $E^\sharp := \{x_1^\sharp,\ldots,x_{\mu_{\max}}^\sharp\}$.

\begin{lem} \label{basegeom} The following properties hold.
\begin{itemize}
\item $x_{1} \in Q_1$ and $x_\nu \in (1/2)Q_\nu$ \qquad for $\nu=2,\ldots,\nu_{\max}$.
\item $\dist(x_\nu,E) \gtrsim \delta_{Q_\nu}$ \qquad\qquad\qquad for $\nu=2,\ldots,\nu_{\max}$.
\item $x_\nu \in 0.99 Q^\circ$ \qquad\qquad\qquad\quad\;\; for $\nu=1,\ldots,\nu_{\max}$.
\item $|x_\nu-x_{\nu'}| \geq \delta_{Q_\nu}/8$ \qquad\qquad\quad\; for $\nu, \nu' = 1,\ldots, \nu_{\max}$.
\end{itemize}
\end{lem}
\begin{proof}
The first and second bullet points follow immediately from Lemma \ref{pointprop}. Now we pass to the third bullet point. From Lemma \ref{moregg}, we have $\delta_{Q_\nu} \geq \delta_{Q^\circ}/20$ for any $Q_\nu$ that touches the boundary of $Q^\circ$, and hence $(1/2)Q_\nu \subset 0.99 Q^\circ$ for such cubes.

On the other hand, if $\cl(Q_{\nu'}) \cap \partial Q^\circ = \emptyset$ then $\dist(Q_{\nu'},\partial Q^\circ) \geq \delta_{Q^\circ}/20$. (The cubes $Q_\nu$ that touch the boundary of $Q^\circ$ provide a buffer between $Q_{\nu'}$ and $Q^\circ$ of width $\delta_{Q^\circ}/20$.) Therefore, $Q_{\nu'} \subset 0.99 Q^\circ$ for such cubes. Thus, $x_\nu \in (1/2) Q_\nu \subset 0.99 Q^\circ$ for each $\nu=2,\ldots,\nu_{\max}$. Since $x_1 = z \in (1/8)Q^\circ$ by definition, we have established the third bullet point.

Finally, the fourth bullet point is an easy consequence of the first bullet point and the Good Geometry of $\CZ$. This concludes the proof of Lemma \ref{basegeom}.
\end{proof}

Henceforth a polynomial written $P_\nu$, $R_\nu$, etc., will always denote a jet at $x_\nu$. Similarly, for any subcollection $E'' = \{x_{\nu_1},\ldots,x_{\nu_s}\} \subset E'$, we naturally identify tuplets of polynomials $(P_{\nu_1},\ldots,P_{\nu_s})$, $(R_{\nu_1},\ldots,R_{\nu_s})$, etc., with Whitney fields on $E''$.

For $P \in \cP$, recall the norm defined in (\ref{sec_not}.\ref{scalednorm}):
$$\left|P\right|_{x,\delta} =  \left( \sum_{|\alpha| \leq m-1} |\partial^\alpha P(x)|^p \cdot \delta^{n + (|\alpha| - m)p} \right)^{1/p} \qquad (x\in \R^n, \; \delta >0).$$ 
For each $\nu=1,\ldots,\nu_{\max}$, we denote
$$\left|P\right|_\nu := \left|P\right|_{x_\nu,\delta_{Q_\nu}}.$$
Note that $|x^\sharp_{\kappa(\nu)} - x_{\nu}| \leq C' \delta_{Q_\nu}$, since $x_\nu \in Q_\nu$, $x_{\kappa(\nu)}^\sharp \in Q^\sharp_{\kappa(\nu)}$ and \textbf{(K1)} from Proposition \ref{keygeom} holds (recall that $\key(Q_\nu) = Q^\sharp_{\kappa(\nu)} $). Thus, (\ref{sec_not}.\ref{ti0}) yields
\begin{equation}\label{ti}
\left|P\right|_\nu \simeq \left| P \right|_{x^\sharp_{\kappa(\nu)}, \delta_{Q_\nu}} = \Bigl( \sum_{|\alpha| \leq m-1} \bigl\lvert \partial^\alpha P\bigl(x^\sharp_{\kappa(\nu)} \bigr) \bigr\rvert^p \cdot \delta_{Q_\nu}^{n + (|\alpha| - m)p} \Bigr)^{1/p}.
\end{equation}
Similarly, if $\nu' \in \{1,\ldots,\nu_{\max}\}$ is such that $|x_{\nu'}-x_{\nu}| \leq C \delta_{Q_\nu}$ and $\delta_{Q_\nu}/\delta_{Q_{\nu'}} \in [C^{-1},C]$, then by (\ref{sec_not}.\ref{ti0}) we have
\begin{equation} \label{eq_1111}
\left|P\right|_\nu \simeq \left|P\right|_{\nu'}. 
\end{equation}
In particular, \eqref{eq_1111} holds if $\nu \leftrightarrow \nu'$.

\section{A Partition of Unity}\label{sec_pu}
\setcounter{equation}{0}
Thanks to the properties of $Q_1,\ldots,Q_{\nu_{\max}}$ and $x_1,\ldots,x_{\nu_{\max}}$ established in Lemmas \ref{gg}, \ref{moregg} and \ref{basegeom}, there exists a partition of unity $\{\theta_\nu\}_{\nu=1}^{\nu_{\max}} \subset C^\infty(Q^\circ)$ that satisfies the following properties. (We leave the construction as an exercise for the interested reader.) 
\begin{align*}
&\mathbf{(POU1)}\; 0 \leq \theta_\nu \leq 1.\\
&\mathbf{(POU2)}\; \theta_\nu \; \mbox{vanishes} \; \mbox{on} \; Q^\circ \setminus (1.1)Q_\nu.
 \\
&\mathbf{(POU3)}\; |\partial^\alpha \theta_\nu| \leq C \delta_{Q_\nu}^{-|\alpha|} \; \mbox{whenever} \; |\alpha| \leq m.\\
&\mathbf{(POU4)}\; \theta_\nu \equiv 1 \; \mbox{near} \;x_\nu,\; \mbox{and} \; \theta_\nu \equiv 0 \; \mbox{near} \; \{ x_{\nu'} : \nu' = 1,\ldots, \nu_{\max}, \; \nu' \neq \nu \}.\\
&\mathbf{(POU5)}\; \sum_{\nu=1}^{\nu_{\max}} \theta_\nu \equiv 1 \; \mbox{on} \; Q^\circ.\\
\end{align*}

\begin{lem}[Patching Estimate]
\label{whittrick}
Let $G_\nu \in L^{m,p}(1.1Q_\nu)$ be given for each $\nu=1,\ldots,\nu_{\max}$, and define
$$G(x):= \sum_{\nu=1}^{\nu_{\max}} G_\nu(x) \theta_\nu(x) \quad (x \in Q^\circ).$$ 
Then
\begin{equation*}
\|G\|_{L^{m,p}(Q^\circ)}^p \lesssim \sum_{\nu=1}^{\nu_{\max}} \|G_\nu\|^p_{L^{m,p}(1.1Q_\nu)} + \sum_{\nu \leftrightarrow \nu'} \left|J_{x_{\nu}} (G_\nu) - J_{x_{\nu'}} (G_{\nu'}) \right|_\nu^p.
\end{equation*}
\end{lem}

\begin{proof}
Fix $\nu' \in \{1,\ldots,\nu_{\max}\}$. Since $\sum \theta_\nu = 1$ on $Q^\circ$, for each $x \in Q_{\nu'}$ we have
$$G(x) = G(x) - G_{\nu'}(x) \left[\sum_{\nu=1}^{\nu_{\max}} \theta_\nu(x)\right] + G_{\nu'}(x) = \sum_{\nu=1}^{\nu_{\max}} (G_{\nu} - G_{\nu'})(x) \theta_\nu(x) + G_{\nu'}(x).$$
For each $x \in Q_{\nu'}$ there are at most a bounded number of $\theta_\nu$ that do not vanish in a neighborhood of $x$ (because $\mbox{supp}(\theta_\nu) \subset (1.1)Q_\nu$ and from the Good Geometry of the CZ cubes). Thus, by taking $m^\th$ derivatives and integrating $p^\th$ powers over $Q_{\nu'}$, we obtain
\begin{align}
\label{r0}
\|G\|^p_{L^{m,p}(Q_{\nu'})} \lesssim \|G_{\nu'}\|^p_{L^{m,p}(Q_{\nu'})} + \sum_{|\alpha+\beta| = m} \sum_{\nu=1}^{\nu_{\max}} \int_{Q_{\nu'}} \left| \partial^\beta(G_\nu - G_{\nu'})(x) \right|^p \left|\partial^\alpha \theta_\nu(x)\right|^p dx.
\end{align}

First we consider a term from the sum in \eqref{r0} when $\nu$ is fixed, $|\beta| = m$ and $\alpha=0$. Since $|\theta_\nu| \leq 1$ and $\theta_\nu$ is supported in $(1.1)Q_\nu$, we have
\begin{equation}
\label{r1}
\int_{Q_{\nu'}} \left|\partial^\beta(G_\nu - G_{\nu'})(x)\right|^p \left|\theta_\nu(x)\right|^p dx \lesssim \|G_\nu\|^p_{L^{m,p}(1.1Q_\nu)} + \|G_{\nu'}\|^p_{L^{m,p}(Q_{\nu'})}.
\end{equation}
The remaining terms in the sum arise when $|\beta| \leq m-1$ and $|\alpha| =m-|\beta| \geq 1$. Since $|\partial^\alpha \theta_\nu| \lesssim \delta_{Q_\nu}^{-|\alpha|}$ and $\partial^\alpha \theta_\nu$ is supported in $(1.1)Q_\nu$, the sum of these terms is controlled by
\begin{equation}
\label{r2}
C \sum_{\substack{ 1 \leq \nu \leq \nu_{\max} \\ 1.1Q_\nu \cap Q_{\nu'} \neq \emptyset}} \sum_{|\beta| \leq m-1} \sup_{x \in 1.1Q_\nu \cap Q_{\nu'}} \left|\partial^\beta (G_\nu - G_{\nu'})(x)\right|^p \delta_{Q_\nu}^n \delta_{Q_\nu}^{-(m-|\beta|)p}.
\end{equation}
Because $|x-x_\nu| \lesssim \delta_{Q_\nu}$ whenever $x \in (1.1)Q_\nu$, and $|x - x_{\nu'}| \lesssim \delta_{Q_{\nu'}}$ whenever $x \in Q_{\nu'}$, the Sobolev inequality (\ref{sec_not}.\ref{SET})  implies that
\begin{align}
\left|\partial^\beta (G_\nu - G_{\nu'})(x)\right| &\lesssim \left|\partial^\beta (J_{x_\nu}(G_\nu) - J_{x_{\nu'}}(G_{\nu'}))(x)\right| + \|G_\nu\|_{L^{m,p}(1.1Q_\nu)} \delta_{Q_\nu}^{m-|\beta|-n/p} \notag{} \\
& + \|G_{\nu'}\|_{L^{m,p}(Q_{\nu'})} \delta_{Q_{\nu'}}^{m-|\beta|-n/p} \qquad (\beta \in \cM). \label{noname}
\end{align}
Now summing \eqref{noname} over all $\nu$ with $\nu \leftrightarrow \nu'$, and using Good Geometry of the CZ cubes, we find that \eqref{r2} is bounded by
\begin{equation*}
C' \sum_{\substack{1 \leq \nu \leq \nu_{\max} \\ \nu \leftrightarrow \nu'}} \left[ \sup_{x \in 1.1Q_\nu \cap Q_{\nu'}} \sum_{|\beta| \leq m-1} \delta_{Q_\nu}^{n-(m-|\beta|)p} \left|\partial^\beta (J_{x_\nu}(G_\nu) - J_{x_{\nu'}}(G_{\nu'}))(x)\right|^p  + \|G_\nu\|_{L^{m,p}(1.1Q_\nu)}^p \right]. 
\end{equation*}
(Recall that $\nu' \leftrightarrow \nu'$ and thus the term from \eqref{noname} that contains $\|G_{\nu'}\|_{L^{m,p}(1.1Q_{\nu'})}$ appears above as is required.) Applying (\ref{sec_not}.\ref{ti0}) (recall that $\left|\cdot\right|_\nu = \left|\cdot\right|_{x_\nu,\delta_{Q_\nu}}$), we may bound the above by
\begin{equation}
\label{r3}
C''\sum_{\substack{1 \leq \nu \leq \nu_{\max} \\ \nu \leftrightarrow \nu'}} \left[ \left|J_{x_\nu}(G_\nu) - J_{x_{\nu'}}(G_{\nu'})\right|_\nu^p  + \|G_\nu\|_{L^{m,p}(1.1Q_\nu)}^p \right].
\end{equation}
Hence, by inserting the bounds \eqref{r1} and \eqref{r3} in \eqref{r0}, and using that each CZ cube has boundedly many neighbors, we have
\begin{equation}
\label{r4} \|F\|_{L^{m,p}(Q_{\nu'})}^p \lesssim \sum_{\substack{ 1 \leq \nu \leq \nu_{\max} \\ \nu \leftrightarrow \nu'}}\left[\left|J_{x_\nu}(G_\nu) - J_{x_{\nu'}}(G_{\nu'})\right|_\nu^p + \|G_\nu\|_{L^{m,p}(1.1Q_\nu)}^p \right].
\end{equation}
We now sum \eqref{r4} over $\nu' = 1,\ldots,\nu_{\max}$. Since the CZ cubes partition $Q^\circ$, and every CZ cube has a bounded number of neighbors, we have
\begin{align*}
\|F\|_{L^{m,p}(Q^\circ)}^p &\lesssim \sum_{ \nu \leftrightarrow \nu'}\left[\left|J_{x_\nu}(G_\nu) - J_{x_{\nu'}}(G_{\nu'}) \right|_\nu^p + \|G_\nu\|_{L^{m,p}(1.1Q_\nu)}^p \right] \\
& \lesssim \sum_{\nu \leftrightarrow \nu'} \left|J_{x_\nu}(G_\nu) - J_{x_{\nu'}}(G_{\nu'})\right|_\nu^p + \sum_{\nu=1}^{\nu_{\max}} \|G_\nu\|_{L^{m,p}(1.1Q_\nu)}^p,
\end{align*}
as required.
\end{proof}

\begin{lem}
\label{patch2}
Let $0<a<1$, a cube $Q \subset \R^n$ and $F \in L^{m,p}(Q)$ be given. Then there exists $G \in \LR$ which depends linearly on $F$ and satisfies 
$$G = F \; \mbox{on} \; aQ \;\; \mbox{and} \;\; \|G\|_{\LR} \leq C \cdot (1-a)^{-m} \cdot \|F\|_{L^{m,p}(Q)}.$$ 
Moreover, suppose that $F=T(f,\vP)$ for some linear map $T: L^{m,p}(E_1;E_2) \rightarrow L^{m,p}(Q)$ with $\Omega$-assisted bounded depth, for some $\Omega \subset [L^{m,p}(E_1)]^*$. Then one may take $G=T'(f,\vP)$ for some linear map $T' : L^{m,p}(E_1;E_2) \rightarrow \LR$ with $\Omega$-assisted bounded depth.
\end{lem}
\begin{proof}
Let $y$ denote the center of $Q$, and fix a cutoff function $\theta \in C^\infty_0(Q)$ that satisfies 
\begin{align} 
\label{p401} & \quad\;\;\theta \equiv 1\; \mbox{on} \; aQ, \; \mbox{and} \\
\label{p402} &|\partial^\alpha \theta| \lesssim ((1-a)\delta_Q)^{-|\alpha|} \; \mbox{when} \; |\alpha| \leq m.
\end{align}
Define $G = \theta F + (1-\theta) \cdot J_y (F) \in \LR$. Then $G$ depends linearly on $F$, while \eqref{p401} implies that $G = F$ on $aQ$. 

Since $G$ matches an $(m-1)^{\rst}$ degree polynomial on the complement of $Q$, we have 
\begin{align*}
\|G\|^p_{\LR} &= \|G\|^p_{L^{m,p}(Q)} \lesssim \|F\|^p_{L^{m,p}(Q)} + \|(1 - \theta) \cdot (J_y (F) - F)\|^p_{L^{m,p}(Q)} \\
&\lesssim \|F\|^p_{L^{m,p}(Q)} + \sum_{|\beta| \leq m-1} \delta_{Q}^{n} \cdot ((1-a)\delta_Q)^{-(m-|\beta|)p} \cdot \sup_{x \in Q} | \partial^\beta (F - J_y (F))(x)|^p.
\end{align*}
(In the last inequality, the $m^\th$ order derivatives that fall on $(J_y(F)-F)$ have been raised to the $p^\th$ power and integrated over $Q$; these terms are incorporated into $\|F\|^p_{L^{m,p}(Q)}$.) Thus, using the Sobolev inequality, we have $\|G\|_{\LR} \lesssim (1-a)^{-m} \cdot \|F\|_{L^{m,p}(Q)}$. It remains to analyze this construction from the perspective of assisted bounded depth linear maps. 

Suppose that $F = T(f,\vP)$ depends linearly on some data $(f,\vP) \in L^{m,p}(E_1;E_2)$, where $T$ has $\Omega$-assisted bounded depth for some $\Omega \subset [L^{m,p}(E_1)]^*$. Then the function $G$ is given by the linear map $T'(f,\vP) := \theta T(f,\vP) + (1-\theta) \cdot J_y (T(f,\vP))$. Thus, for each point $x \in \R^n$ there exist linear maps $\psi_x, \phi_x : \cP \rightarrow \cP$, such that
\begin{equation} \label{jets0}
J_x[T'(f,\vP)] = \left\{
\begin{array}{lr}
\psi_x(J_x[T(f,\vP)]) + \phi_x(J_y[T(f,\vP)]) &: x \in Q. \\
\phi_x(J_y[T(f,\vP)]) &: x \notin Q. \\ 
\end{array}
\right.
\end{equation}
These maps are defined by $\psi_x(P) = J_x[\theta P]$ and $\phi_x(P) = J_x[(1-\theta) P]$ for each $P \in \cP$. Note that \eqref{jets0} follows from the containment $\supp(\theta) \subset Q$ and the observation that $ J_x[\theta H] = \psi_x(J_x H)$ and $J_x[(1 - \theta) H] = \phi_x(J_x H)$ for each function $H$ that is $C^{m-1}$ on a neighborhood of $x$.

Because $T$ has $\Omega$-assisted bounded depth, \eqref{jets0} implies that $T'$ has $\Omega$-assisted bounded depth. This concludes the proof of Lemma \ref{patch2}.
\end{proof}

\section{Local Extension Operators}\label{sec_loc}
\setcounter{equation}{0}
In this section we apply the inductive hypothesis \textbf{(IH)} to construct local extension operators for functions defined on subsets of $E \cap 3Q_\nu$.

Fix $\nu \in \{1,\ldots,\nu_{\max}\}$ and $\oE_\nu \subset E \cap 3Q_\nu$.

Suppose that $\#(\oE_\nu \cup \{x_\nu\}) \geq 3$. Since $\oE_\nu \subset 3Q_\nu$ and $x_\nu \in Q_\nu$ we have $10 \cdot \diam(\oE_\nu \cup \{x_\nu\}) \leq 30 \delta_{Q_\nu}$. Also, $\#(E \cap 3Q_\nu) \geq \#(\oE_\nu) \geq 2$. By definition, since $Q_\nu$ is OK and $\#(E \cap 3Q_\nu) \geq 2$, there exists $\cA_\nu < \cA$ such that for all $x \in E \cap 3Q_\nu$ we have that
\begin{align*}
&\sigma(x,E \cap 3Q_\nu) \; \mbox{contains an} \; (\cA_\nu,x,\epsilon_0,30 \delta_{Q_\nu})\mbox{-basis} \implies \; (\mbox{from} \; \sigma(x,E \cap 3Q_\nu) \subset \sigma(x,\oE_\nu)) \\
&\sigma(x,\oE_\nu) \; \mbox{contains an} \; (\cA_\nu,x,\epsilon_0,30 \delta_{Q_\nu})\mbox{-basis} \implies \; (\mbox{from the remark in section \ref{poly_basis}})\\
&\sigma(x,\oE_\nu) \; \mbox{contains an} \; (\cA_\nu,x,\epsilon_0,10\cdot \diam(\oE_\nu \cup \{x_\nu\}))\mbox{-basis}.
\end{align*}
Thus (\ref{sec_ind}.\ref{eq_600}) holds with $\hE = \oE_\nu$, $\hx = x_\nu$ and $\widehat{\cA} = \cA_\nu$. From the inductive hypothesis \textbf{(IH)}, it follows that the Extension Theorem for $(\oE_\nu,x_\nu)$ holds, as long as $\#(\oE_\nu \cup \{x_\nu\}) \geq 3$. 

On the other hand, if $\#(\oE_\nu \cup \{x_\nu\}) \leq 2$ then Lemma \ref{small} implies the Extension Theorem for $(\oE_\nu,x_\nu)$.

In any case, for each $\nu=1,\ldots,\nu_{\max}$, and each $\oE_\nu \subset E \cap 3Q_\nu$, there exists $(T_{\nu}, M_\nu, \Omega_\nu)$ with the following properties:

\begin{align*}
&\mathbf{(L1)}\; T_\nu:L^{m,p}(\oE_\nu;x_{\nu}) \rightarrow \LR \;\; \mbox{is a linear extension operator}.\\
&\mathbf{(L2)}\; \|(f_\nu,P_\nu)\|_{L^{m,p}(\oE_\nu,x_{\nu})} \leq \|T_\nu(f_\nu,P_\nu)\|_{\LR} \leq C \|(f_\nu,P_\nu)\|_{L^{m,p}(\oE_\nu,x_{\nu})}, \; \mbox{and}\\
&\mathbf{(L3)}\; C^{-1} M_\nu(f_\nu,P_\nu) \leq \|T_\nu(f_\nu,P_\nu)\|_{\LR} \leq C M_\nu(f_\nu,P_\nu) \;\; \mbox{for each} \; (f_\nu,P_\nu).\\
&\mathbf{(L4)}\; \sum_{\omega \in \Omega_\nu} \sp(\omega) \leq C \cdot \#[\oE_\nu].\\
&\mathbf{(L5)}\; T_\nu \; \mbox{has} \; \Omega_\nu\mbox{-assisted bounded depth}.\\
&\mathbf{(L6)} \; \mbox{There exists a collection of linear functionals} \; \Xi_\nu \subset [L^{m,p}(\oE_\nu;x_\nu)]^*, \; \mbox{so that} \notag{}\\
&\qquad\qquad \mathbf{(a)} \; \mbox{each functional in}\; \Xi_\nu \; \mbox{has}\; \Omega_\nu\mbox{-assisted bounded depth},\notag{}\\
&\qquad\qquad \mathbf{(b)} \; \# (\Xi_\nu) \leq C \cdot \#(\oE_\nu), \; \mbox{and} \notag{} \\
&\qquad\qquad \mathbf{(c)} \; M_\nu(f_\nu,P_\nu) = \left(\sum_{\xi \in \Xi_\nu} |\xi(f_\nu,P_\nu)|^p \right)^{1/p} \;\; \mbox{for each} \; (f_\nu,P_\nu). 
\end{align*}

We now outline the content of sections \ref{sec_key}-\ref{sec_inf}. In section \ref{sec_key}, we compute the jet $R^\sharp_\mu$ of a near optimal extension of $f|_{E \cap 9Q^\sharp_\mu}$ that is suitably consistent with the polynomial $P$. We can arrange that $R^\sharp_\mu$ depends linearly on the local data and $P$. We then define $R_1 = P$ and $R_\nu = R^\sharp_{\kappa(\nu)}$ for each $\nu=2,\ldots,\nu_{\max}$.

We must show that $\vR = (R_1,\ldots,R_{\nu_{\max}})$ is the jet on $E'$ of some near optimal extension of $(f,P)$. To do so we proceed by several steps. In section \ref{sec_aux}, we produce some new local estimates on the coefficients of the auxiliary polynomials $P^x_\alpha$. (Recall that the $P^x_\alpha$ are Taylor polynomials of linear combinations of the $\varphi_\alpha$, which vanish on $E$ and have small $\LR$ seminorm.) These new estimates complement the unit-scale estimates on $P^x_\alpha$ from section \ref{sec_ind}, and are localized to the lengthscale $\delta_Q$ for the cube $Q \in \CZ$ that contains $x$.

In section \ref{sec_ess}, through an estimate on the size of $\sigma(x^\sharp_\mu, E \cap 9Q^\sharp_\mu)$ (or dually, through an estimate on the appropriate trace semi-norm), we bound the error between our guess $R^\sharp_\mu \in \cP$ and any hypothetical jet $R'$ of a near optimal extension of $f|_{E \cap 9Q^\sharp_\mu}$. Thus, we establish that $R^\sharp_\mu \in \cP$ was almost uniquely determined from its constraints and near minimality.

In section \ref{sec_ac}, we prove a certain Sobolev inequality that bounds the variance of the derivatives of an $F \in L^{m,p}(\R^n)$ along the paths introduced in section \ref{sec_PTKC2}. This inequality and our assumption that the approximate monomials $\{\varphi_\alpha \}_{\alpha \in \cA}$ vanish on $E$ are then used to prove the existence of a near optimal extension of $(f,P)$ whose jet on $E'$ satisfies two additional constraints, termed Constancy Along Paths and Coherence with $P$. It will be the case that $\vR$ is Constant Along Paths and Coherent with $P$.

In section \ref{sec_mod}, we patch together the local extension operators to build an extension operator $T: L^{m,p}(E;E') \rightarrow \LR$; along the way we develop a formula for the $L^{m,p}(E;E')$ trace seminorm. In section \ref{sec_jets}, we use the results from sections \ref{sec_ess} and \ref{sec_ac}, and the formula for the trace seminorm from section \ref{sec_mod}, to prove that $\vR$ is the jet on $E'$ of some near optimal extension. At this point, we are ready to tackle the main theorems.

In section \ref{plstow}, we prove the Extension Theorem for $(E,z)$. In section \ref{sec_pf}, we prove Theorem 1 for finite $E$, as well as Theorems 2 and 3; we also prove an extension theorem for the inhomogeneous Sobolev space. Finally, in section \ref{sec_inf} we use a Banach limit to deduce Theorem \ref{thm1} for infinite sets from the finite case.

Until the end of section \ref{sec_pf}, we fix some (arbitrary) data $(f,P) \in L^{m,p}(E;z)$.

\section{Extension Near the Keystone Cubes} \label{sec_key}
\setcounter{equation}{0}

We begin with a simple lemma concerning the minimization of the $\ell^p$-norm subject to linear constraints. We will work in $R^{N_0+k}$ and denote a vector in $\R^{N_0 + k}$ by $(w,w')$ with $w \in \R^{N_0}$, $w' \in \R^k$.
\begin{lem}
\label{lpmin}
Given integers $N_0,k \geq0$ and linear functionals $\lambda_1,\ldots,\lambda_L$ on $\R^{N_0 + k}$, there exists a linear map $\xi: \R^{N_0} \rightarrow \R^k$, so that for each $w \in \R^{N_0}$ we have
$$\sum_{i=1}^L \lvert \lambda_i(w, \xi(w)) \rvert^p \leq C \inf_{w' \in \R^k} \left\{ \sum_{i=1}^L \lvert \lambda_i(w, w') \rvert^p \right\}, \;\; \mbox{with} \; C=C(k,p).$$
\end{lem}
\begin{proof}
It suffices to assume that $k=1$, for by iteratively applying this result we obtain the full version of the lemma. Write $\lambda_i(w,w') = \widehat{\lambda}_i(w) - a_i w'$, where $\widehat{\lambda}_i : \R^{N_0} \rightarrow \R$ is a linear functional and $a_i \in \R$ for each $i=1,\ldots,L$. The expression to be minimized is
$$\sum_{i=1}^L \left| \widehat{\lambda}_i(w) - a_i w' \right|^p.$$
By removing those terms that are independent of $w'$, we may assume that $a_i \neq 0$ for all $i=1,\ldots,L$, and rewrite the expression as
\begin{equation} \label{rf} \sum_{i=1}^L \left|\frac{\widehat{\lambda}_i(w)}{a_i} - w' \right|^p |a_i|^p.
\end{equation}
If $L \neq 0$, then a standard application of H\"older's inequality shows that 
$$w'  = \xi(w) := \frac{ \widehat{\lambda}_1(w) |a_1|^{p}/ a_1 + \cdots + \widehat{\lambda}_L(w) |a_L|^{p} / a_L}{ |a_1|^p + \cdots + |a_L|^p}$$ minimizes \eqref{rf} to within a factor of $2^{p+1}$. If $L=0$, we simply take $w'=0$.
\end{proof}

\begin{lem}
Let $Q^\sharp_\mu$ be a given keystone cube. Set $f^\sharp_{\mu} = f|_{E \cap 9Q^\sharp_\mu}$. There exists a polynomial $R_\mu^\sharp \in \cP$ with the following properties :
\begin{equation}
R_\mu^\sharp \; \mbox{depends linearly on} \; f_\mu^\sharp \; \mbox{and} \; P;
\end{equation}
\begin{equation}\label{jetop1} 
\partial^\alpha R_\mu^\sharp(x_{\mu}^\sharp) = \partial^\alpha P(x_{\mu}^\sharp) \; \mbox{for every} \; \alpha \in \cA;
\end{equation}
and
\begin{align} \label{jetop2}
&\|(f_\mu^\sharp,R^\sharp_\mu)\|_{L^{m,p}(E \cap 9 Q_\mu^\sharp; x_{\mu}^\sharp)} \leq \\
&\qquad C  \inf \left\{ \|(f_\mu^\sharp,R')\|_{L^{m,p}(E \cap 9 Q^\sharp_\mu;x_{\mu}^\sharp)}: R' \in \cP, \; \partial^\alpha R'(x_{\mu}^\sharp) = \partial^\alpha P(x_{\mu}^\sharp) \; \mbox{for every} \; \alpha \in \cA \right\}.\notag{}
\end{align}
\label{jetop}
\end{lem}

\begin{proof}
The main content of the proof consists of showing that the trace seminorm 
\begin{equation} \label{appx0} \|(f_\mu^\sharp,R')\|_{L^{m,p}(E \cap 9Q_\mu^\sharp ; x_{\mu}^\sharp)}\end{equation}
is given up to a universal constant factor by an expression of the type
\begin{equation}\label{appx} \left( \left| \lambda_1(f_\mu^\sharp,R') \right|^p + \cdots + \left| \lambda_{L}(f_\mu^\sharp,R')\right|^p \right)^{1/p},\end{equation}
where $\lambda_1,\ldots,\lambda_L$ are some linear functionals. 

Denote by $CZ^\sharp_\mu$  the collection of cubes $Q \in \CZ$ such that $Q \cap 100Q_\mu^\sharp \neq \emptyset$. We write $CZ_\mu^\sharp = \{Q_{\nu_1}, \ldots, Q_{\nu_K}\}$, where $Q_{\nu_1} = Q_\mu^\sharp$. Hence also $x_{\nu_1} = x^\sharp_\mu$.
By definition of the keystone cubes and Lemma \ref{moregg}, we have
\begin{equation}
\label{ls}
\delta_{Q^\sharp_\mu} \leq \delta_{Q_{\nu_i}} \leq 10^3 \delta_{Q^\sharp_\mu} \;\; \mbox{for each} \; i=1,\ldots,K,\; \mbox{and hence} \; K \leq C'.
\end{equation}

Let $R' \in \cP$, $P_{\nu_i} \in \cP$ ($i=2,\ldots,K$) be arbitrary polynomials. Set $P_{\nu_1} = R'$. For each $i = 1,\ldots, K$, we define the following objects:
\begin{itemize}
\item The subset $S_{\nu_i} = E \cap 9Q_\mu^\sharp \cap 3Q_{\nu_i}$ and the function $f_{\nu_i} = f|_{S_{\nu_i}}$. 
\item The map $M_{\nu_i} : L^{m,p}(S_{\nu_i}; x_{\nu_i}) \rightarrow \R_+$ that satisfies \textbf{(L1-6)} for the subset $S_{\nu_i} \subset E \cap 3Q_{\nu_i}$ and the representative $x_{\nu_i} \in Q_{\nu_i}$ (together with some linear map $T_{\nu_i} : L^{m,p}(S_{\nu_i}; x_{\nu_i}) \rightarrow \LR$ and some collection of linear functionals $\Omega_{\nu_i} \subset [L^{m,p}(S_{\nu_i})]^*$).
\end{itemize}

From \textbf{(L2)}, \textbf{(L3)}, \textbf{(L6)}, for any $i=1,\ldots K$, we have:
\begin{align} \label{eee}
M_{\nu_i}(f_{\nu_i},P_{\nu_i}) &= \left(\sum_{j=1}^{L_i} |\lambda_i^j(f_{\nu_i},P_{\nu_i})|^p \right)^{1/p} \simeq \|(f_{\nu_i},P_{\nu_i})\|_{L^{m,p}(S_{\nu_i};x_{\nu_i})} \\
& \mbox{for some linear functionals} \; \lambda_i^1,\ldots,\lambda_i^{L_i} \in [L^{m,p}(S_{\nu_i};x_{\nu_i})]^*. \notag{}
\end{align}
Thus, there exists $F_{\nu_i} \in \LR$ with
\begin{align} 
& F_{\nu_i} = f_{\nu_i} \; \mbox{on} \; S_{\nu_i} \; \mbox{and} \; J_{x_{\nu_i}} (F_{\nu_i}) = P_{\nu_i}; \;\mbox{and}\label{eq_1321} \\
&\|F_{\nu_i}\|_{\LR} \simeq \|(f_{\nu_i},P_{\nu_i})\| \simeq M_{\nu_i}(f_{\nu_i},P_{\nu_i}). \label{eq_1322}
\end{align}

Suppose that $Q \in \CZ$ satisfies $Q \cap 100 Q^\sharp_\mu = \emptyset$ and $(1.1)Q \cap 9Q^\sharp_\mu \neq \emptyset$. Hence, $\delta_{Q} \geq 100 \delta_{Q^\sharp_\mu}$. However, this contradicts the second bullet point in Lemma \ref{moregg}. Therefore, $Q \notin CZ_\mu^\sharp \implies (1.1)Q \subset \R^n \setminus 9Q_\mu^\sharp$. Thus, by \textbf{(POU2)} we have
\begin{equation}
\theta_\nu \equiv 0 \;\; \mbox{on} \;\; 9 Q^\sharp_\mu \cap Q^\circ \quad \mbox{for} \; \nu \in \{1,\ldots,\nu_{\max}\} \setminus \{\nu_1,\ldots,\nu_K\} .\label{eq_1110}
\end{equation} 
Define $F^\circ \in L^{m,p}(Q^\circ)$ by
$$F^\circ(x) := \sum_{i=1}^K F_{\nu_i}(x) \theta_{\nu_i}(x) + \sum_{\nu \not \in \{\nu_1,\ldots,\nu_K\} } R'(x) \theta_\nu(x) \qquad (x \in Q^\circ).$$
From \textbf{(POU2)}, \eqref{eq_1321} and the definition of $S_{\nu_i}$, we have $F_{\nu_i} = f_{\nu_i}$ on $\supp(\theta_{\nu_i}) \cap E \cap 9Q^\sharp_\mu$. Hence, from \textbf{(POU5)} and \eqref{eq_1110}, we have $F^\circ = f^\sharp_\mu$ on $E \cap 9 Q^\sharp_\mu$ (recall that $E \subset (1/8)Q^\circ$). 

From \textbf{(POU4)} and \eqref{eq_1321}, we find that $J_{x_{\nu_i}} (F^\circ) = P_{\nu_i}$ for each $i=1,\ldots,K$. To summarize, we have
\begin{equation} \label{ext4}
F^\circ = f_\mu^\sharp \; \mbox{on} \; E \cap 9Q^\sharp_\mu \; \mbox{and} \; J_{x_{\nu_i}} (F^\circ) = P_{x_{\nu_i}} \; \mbox{for each} \; i=1,\ldots,K.
\end{equation} 

We prepare to estimate the seminorm $\|F^\circ\|_{L^{m,p}(Q^\circ)}$ using Lemma \ref{whittrick}. 

Take $G_\nu = F_{\nu}$ when $\nu \in \{\nu_1,\ldots,\nu_K\}$, and $G_\nu = R'$ when $\nu \in \{1,\ldots,\nu_{\max}\} \setminus \{\nu_1,\ldots,\nu_K\}$. Lemma \ref{whittrick} provides the estimate:
\begin{equation*}
\|F^\circ\|^p_{L^{m,p}(Q^\circ)} \lesssim \sum_{\nu=1}^{\nu_{\max}} \|G_\nu\|^p_{L^{m,p}(1.1Q_\nu)} + \sum_{\nu \leftrightarrow \nu'} \left|J_{x_{\nu}} (G_\nu) - J_{x_{\nu'}} (G_{\nu'})\right|_\nu^p.
\end{equation*} 
The terms with $\nu \notin \{\nu_1,\ldots, \nu_K\}$ and $\nu, \nu' \notin \{\nu_1,\ldots,\nu_K\}$ in the first and second sum vanish, respectively. In view of $\eqref{eq_1321}$, the remaining terms in the second sum are of the form $\left|P_\nu - P_{\nu'}\right|^p_{\nu}$ and $\left|P_\nu - R'\right|^p_\nu$ ($\nu, \nu' \in \{\nu_1,\ldots,\nu_K\}$). (Here, we also use (\ref{sec_rep}.\ref{eq_1111}).) Thus,
\begin{align*}
\|F^\circ\|^p_{L^{m,p}(Q^\circ)} &\lesssim \sum_{i=1}^{K} \|F_{\nu_i}\|^p_{L^{m,p}(1.1Q_{\nu_i})} + \sum_{i,j=1}^K \left|P_{\nu_i} - P_{\nu_j}\right|_{\nu_i}^p\\
&\simeq \sum_{i=1}^{K} M_{\nu_i}(f_{\nu_i},P_{\nu_i})^p + \sum_{i,j=1}^K \left|P_{\nu_i} - P_{\nu_j}\right|_{\nu_i}^p \quad (\mbox{thanks to} \;\eqref{eq_1322}).
\end{align*}
Lemma \ref{patch2} yields $F \in \LR$ such that $F= F^\circ$ on $0.99 Q^\circ$  and $\|F\|_{\LR}\leq C \|F^\circ\|_{L^{m,p}(Q^\circ)}$. Therefore,
\begin{equation} \label{eq_3336}
\|F\|_{\LR}^p \lesssim \sum_{i=1}^{K} M_{\nu_i}(f_{\nu_i},P_{\nu_i})^p + \sum_{i,j=1}^K \left|P_{\nu_i} - P_{\nu_i}\right|_{\nu_i}^p.\end{equation}
Recall that $E \cup E' \subset 0.99 Q^\circ$ (see (\ref{sec_ind}.\ref{cube}) and Lemma \ref{basegeom}). Thus, \eqref{ext4} yields
\begin{equation}\label{ext5}
F = f^\sharp_\mu \; \mbox{on}\; {E \cap 9Q^\sharp_\mu} \; \mbox{and} \; J_{x_{\nu_i}} (F) = P_{\nu_i} \; \mbox{for each}  \; i=1,\ldots,K.
\end{equation}

Let $H \in \LR$ be an arbitrary function that satisfies
$$\circledast \qquad H = f^\sharp_\mu \; \mbox{on} \; E \cap 9Q^\sharp_\mu \; \mbox{and} \; J_{x_{\nu_i}} (H) = P_{\nu_i}, \; \mbox{for each} \; i=1,\ldots,K.$$ 
From \eqref{eq_1322} and the definition of the trace seminorm, we have
$$M_{\nu_i}(f_{\nu_i},P_{\nu_i})^p \simeq \|(f_{\nu_i},P_{\nu_i})\|^p \leq \|H\|^p_{\LR}, \;\; \mbox{for each} \; i=1,\ldots,K,$$ 
while the Sobolev inequality implies that
$$\left|P_{\nu_i}-P_{\nu_j}\right|^p_{\nu_i} = \left|J_{x_{\nu_i}}(H) - J_{x_{\nu_j}}(H)\right|^p_{\nu_i} \lesssim \|H\|^p_{\LR}, \;\; \mbox{for each} \; i,j=1,\ldots,K.$$ 
(Here, we use \eqref{ls} and $x_{\nu_i} \in Q_{\nu_i}$.) Summing these two estimates over $i,j \in \{1,\ldots,K\}$, and using \eqref{ls}, provides a bound on the right-hand side of \eqref{eq_3336} by $C \|H\|^p_{\LR}$. We now take $E'_\mu := \{x_{\nu_1},\ldots,x_{\nu_K}\}$.  Taking the infimum in this estimate, with respect to $H \in \LR$ that satisfy $\circledast$, we have
$$  \sum_{i=1}^{K} M_{\nu_i}(f_{\nu_i},P_{\nu_i})^p + \sum_{i,j=1}^K \left|P_{\nu_i} - P_{\nu_j}\right|_{\nu_i}^p\lesssim \left\| \left(f^\sharp_\mu, (P_{\nu_i})_{i=1}^K \right) \right\|^p_{L^{m,p}(E \cap 9Q^\sharp_\mu;E'_\mu)}.$$ 
Together with \eqref{eq_3336} and \eqref{ext5}, this tells us that
$$ \sum_{i=1}^{K} M_{\nu_i}(f_{\nu_i},P_{\nu_i})^p + \sum_{i,j=1}^K \left|P_{\nu_i} - P_{\nu_j}\right|_{\nu_i}^p\simeq \left\| \left(f^\sharp_\mu, (P_{\nu_i})_{i=1}^K \right) \right\|^p_{L^{m,p}(E \cap 9Q^\sharp_\mu;E'_\mu)}.$$
For every $R' \in \cP$, the definition of the trace seminorm gives
\begin{equation*}
\left\|\left(f^\sharp_\mu, R'\right)\right\|^p_{L^{m,p}(E \cap 9Q^\sharp_\mu; x^\sharp_\mu)} = \inf \left\{ \left\|\left(f^\sharp_\mu, (P_{\nu_i})_{i=1}^K \right)\right\|^p_{L^{m,p}(E \cap 9Q^\sharp_\mu;E'_\mu)}: (P_{\nu_i})_{i=1}^K \in Wh(E'_\mu),\; P_{\nu_1} = R' \right\}.
\end{equation*}
Thus, by the last two lines and \eqref{eee}, we have
\begin{align*}
\left\|\left(f^\sharp_\mu, R'\right)\right\|^p_{L^{m,p}(E \cap 9Q^\sharp_\mu; x^\sharp_\mu)} \simeq \inf \left\{ \sum_{i=1}^{K} \sum_{j=1}^{L_i} |\lambda_i^j(f_{\nu_i},P_{\nu_i})|^p + \sum_{i,j=1}^K \left|P_{\nu_i} - P_{\nu_j}\right|_{\nu_i}^p : (P_{\nu_i})_{i=1}^K \in Wh(E'_\mu),\; P_{\nu_1} = R' \right\}.
\end{align*}

By an application of Lemma \ref{lpmin}, we can solve for $(P_{\nu_i})_{i=2}^K$ that depends linearly on $(f^\sharp_\mu,R')$ and minimizes the expression inside the infimum to within the multiplicative factor $C(p,k)$; note that we are solving for $k=\dim(\cP) \cdot( K - 1)$ real variables. In view of \eqref{ls}, the constant $C(k,p)$ from the lemma is universal and we have shown that
\begin{equation} \label{eq_1120}
\|(f^\sharp_\mu,R')\|^p_{L^{m,p}(E \cap 9Q^\sharp_\mu;x_{\mu}^\sharp)} \simeq \sum_{i=1}^L |\lambda_i(f^\sharp_\mu,R')|^p,
\end{equation}
for some integer $L$ and some linear functionals $\lambda_1,\ldots,\lambda_L$.

We apply Lemma \ref{lpmin} once more to solve for $R' = R^\sharp_\mu \in \cP$ that depends linearly on $(f^\sharp_\mu,P)$, and minimizes the right-hand side of \eqref{eq_1120} to within a multiplicative factor $C(p,k)$, subject to the condition [$\partial^\alpha R^\sharp_\mu(x^\sharp_\mu) = \partial^\alpha P(x^\sharp_\mu)$ for all $\alpha \in \cA$]. In this case, we solve for the remaining $k=\dim(\cP) - \#(\cA)$ coefficients of the polynomial $R^\sharp_\mu$. Again, the constant $C(k,p)$ is universal. It follows that $R^\sharp_\mu$ satisfies \eqref{jetop1} and \eqref{jetop2}. This concludes the proof of Lemma \ref{jetop}.
\end{proof}

\section{Estimates for Auxiliary Polynomials}\label{sec_aux}
\setcounter{equation}{0}

Recall that for each $x \in Q^\circ$ we defined polynomials $(P_\alpha^x)_{\alpha \in \cA}$ that satisfy (\ref{sec_ind}.\ref{mainprops})-(\ref{sec_ind}.\ref{mainprops4}).
\begin{lem}
\label{lem201}
Let $Q \in \CZ$ be fixed. Then, the following estimate holds:
\begin{equation*}
|\partial^\beta P^y_\alpha(y)| \leq C \delta_Q^{|\alpha| - |\beta|} \qquad (\alpha \in \cA, \; \beta \in \cM,\; y \in Q).
\end{equation*}
\end{lem}
\begin{proof}
Define $E^+ := (E \cap 3Q^+) \cup \{y\}$, and let $x \in E^+$ be arbitrary. From (\ref{sec_ind}.\ref{mainprops1})-(\ref{sec_ind}.\ref{mainprops4}), it follows that the collection $(P^x_\alpha)_{\alpha \in \cA}$ forms an $(\cA,x,C\epsilon,1)$-basis for $\sigma(x,E)$. Set $C' = 30^{m} C$. Then, by the definition of $(\cA,x,C \epsilon,1)$-bases,  $(P^x_\alpha)_{\alpha \in \cA}$ forms an $(\cA,x,C'\epsilon,30)$-basis for $\sigma(x,E)$. Since the dyadic cube $Q^+ \subset Q^\circ$ satisfies $\delta_{Q^+} \leq \delta_{Q^\circ} = 1$, by the remark in section \ref{poly_basis}, $(P^x_\alpha)_{\alpha \in \cA}$ forms an $(\cA,x,C'\epsilon,30\delta_{Q^+})$-basis for $\sigma(x,E)$, hence also for $\sigma(x,E \cap 3Q^+)$.

Define $\epsilon_1 = \min\{c_1,\epsilon_0/C_1\}$, where $c_1$ and $C_1$ are the constants from Lemma \ref{techlem}. For the sake of contradiction, suppose that
$$\max \bigl\{ \lvert \partial^\beta P^{\hy}_\alpha(\hy) \rvert \cdot (30 \delta_{Q^+})^{|\beta|-|\alpha|} : \hy \in E^+, \alpha \in \cA, \beta \in \cM \bigr\} >\epsilon_1^{-D-1}.$$
We assume that $\epsilon \leq \epsilon_1^{2D+2}/C'$. Then the hypotheses of Lemma \ref{techlem} hold with the parameters
$$ \left(\epsilon_1,\epsilon_2,\delta,\cA,E_1,E_2,(\widetilde{P}^x_\alpha)_{\alpha \in \cA,x \in E_2}\right) := \left(\epsilon_1,C' \epsilon,30 \delta_{Q^+},\cA,E \cap 3Q^+,E^+,(P^x_\alpha)_{\alpha \in \cA,x \in E^+}\right) .$$
Thus, there exists $\oA < \cA$, such that $\sigma(x,E \cap 3Q^+)$ contains an $(\oA,x,C_1\epsilon_1,30\delta_{Q^+})$-basis for all $x \in E^+$. Since $C_1 \epsilon_1 \leq \epsilon_0$, we deduce that $Q^+$ is OK. This contradicts $Q \in \CZ$, and completes the proof of Lemma \ref{lem201}.
\end{proof}

\section{Estimates for Local Extensions}\label{sec_ess}
\setcounter{equation}{0}

\begin{lem}
Let $Q \in \CZ \setminus \{Q_1\}$ and let $x_Q$ be the representative point for $Q$. Then for every $P \in \cP$, we have
\begin{equation} \label{ess1}
\|(0|_{E\cap 1.1Q},P)\|_{L^{m,p}(E \cap 1.1Q;x_Q)} \leq \|(0|_E,P)\|_{L^{m,p}(E;x_Q)} \lesssim  \left|P\right|_{x_Q,\delta_Q}.
\end{equation}
Moreover, if $P \in \cP$ satisfies $\partial^\alpha P(x_Q) = 0$ for all $\alpha \in \mathcal{A}$, then
\begin{equation} \label{ess2}
\|(0|_{E \cap 9 Q}, P)\|_{L^{m,p}(E \cap 9 Q; x_Q)} \simeq \left|P\right|_{x_Q,\delta_Q}.
\end{equation}
\label{esslem}
\end{lem}
\begin{proof} Since $Q \neq Q_1$ we have that $x_Q \neq x_1=z$. Thus, by the properties of the representative points listed in Lemma \ref{basegeom}, there exists $\theta \in C^\infty_0(Q)$ with 
\begin{align}
\label{p701} & \theta \equiv 0 \; \mbox{in a neighborhood of} \; E; \\ 
\label{p702} & \theta \equiv 1 \; \mbox{in a neighborhood of} \; x_Q; \; \mbox{and} \\
\label{p703} & |\partial^\alpha \theta| \lesssim \delta_{Q}^{-|\alpha|} \; \mbox{when} \; |\alpha| \leq m.
\end{align} 
Define $H := \theta P$. By a straightforward calculation using $\supp(\theta) \subset Q$ and \eqref{p703}, we have
$$\|H\|^p_{\LR} \lesssim \sum_{|\alpha| \leq m-1} |\partial^\alpha P(x_Q)|^p \delta_Q^{n - (m-|\alpha|)p} = \left[ \left|P\right|_{x_Q,\delta_Q} \right]^p.$$ 
Moreover, \eqref{p701} and \eqref{p702} show that $H$ is a competitor in the infimum that defines
$$\|(0|_E, P)\|_{L^{m,p}(E;x_Q)} := \inf \left\{ \|G\|_{\LR} : G|_{E} = 0, \; J_{x_Q} G = P \right\},$$ 
and therefore
$$\|(0|_E, P)\|_{L^{m,p}(E;x_Q)} \leq \|H\|_{\LR} \lesssim \left|P\right|_{x_Q,\delta_Q}.$$
Finally, by definition of the trace seminorm we have 
\begin{equation}
\label{eq_1410}
\|(0|_{E \cap (1.1)Q}, P)\|_{L^{m,p}(E \cap 1.1Q;x_Q)} \leq \|(0|_{E \cap 9Q}, P)\|_{L^{m,p}(E \cap 9Q;x_Q)} \leq \|(0|_E,P)\|_{L^{m,p}(E;x_Q)}. 
\end{equation}
This completes the proof of \eqref{ess1}.

To prove \eqref{ess2}, we let $P \in \cP$ be given with 
\begin{equation} \label{feffy} \partial^\alpha P(x_Q) = 0 \;\mbox{for all} \;\alpha \in \mathcal{A},
\end{equation} 
and establish the reverse inequality 
\begin{equation} \label{reffy}\|(0|_{E \cap 9Q},P)\|_{L^{m,p}(E\cap 9Q;x_Q)} \geq c \left|P\right|_{x_Q,\delta_Q}.
\end{equation}
From the definition of $\sigma(\cdot,\cdot)$ and the definition of the trace seminorm, we have
\begin{equation} \label{reform}
\bigl\{P \in \cP : \|(0|_{E \cap 9Q},P)\|_{L^{m,p}(E\cap 9Q;x_Q)} \leq 1/2 \bigr\} \subset \sigma(x_Q, E \cap 9 Q).
\end{equation}
Let $\epsilon_1>0$ be some small universal constant, to be determined by the end of the proof. For the sake of contradiction, suppose that \eqref{reffy} fails with $c = \epsilon_1^{D+1}$, for some $P \in \cP$ that satisfies \eqref{feffy}. Upon rescaling $P$ and using \eqref{reform}, we see that
\begin{align} 
\label{eq:700} &P \in C\epsilon_1^{D+1} \cdot \sigma\left(x_Q,E \cap 9Q\right); \; \mbox{and}\\
\label{eq:701} &\max_{\beta \in \cM} \left\{ \left|\partial^\beta P(x_Q)\right| \delta_{Q}^{n/p + |\beta| - m} \right\}= 1.
\end{align}
For each integer $\ell \geq 0$, we define
$$\Delta_\ell = \left\{\alpha \in \cM : |\partial^\alpha P(x_Q)| \delta_Q^{\frac{n}{p} + |\alpha| - m} \in (\epsilon_1^\ell,1]\right\}.$$ Note that $\Delta_{\ell} \subset \Delta_{\ell + 1}$ for each $\ell\geq 0$ and that $\Delta_\ell \neq \emptyset$ for $\ell \geq 1$ by \eqref{eq:701}. Since $\cM$ contains $D$ elements, there exists $0 \leq \ell_* \leq D$ with $\Delta_{\ell_*} = \Delta_{\ell_* + 1} \neq \emptyset$. Let $\oa \in \cM$ be the maximal element of $\Delta_{\ell_*}$. Thus,\begin{equation}\label{eq_23}
|\partial^\oa P(x_Q)| \delta_{Q}^{n/p + |\oa| - m} \geq \epsilon_1^{\ell_*} .
\end{equation}
Moreover, for every $\beta \in \cM$ with $\beta > \oa$, we have $\beta \notin \Delta_{\ell_*} = \Delta_{\ell_*+1}$, and hence
\begin{equation}\label{eq_24}
|\partial^\beta P(x_Q)| \delta_{Q}^{n/p + |\beta| - m} \leq \epsilon_1^{\ell_*+1} \quad (\beta \in \cM, \beta > \oa).
\end{equation}

Thanks to \eqref{eq_23}, we may define $P_{\oa} := \left[\partial^\oa P(x_Q)\right]^{-1} P$. Hence,
\begin{equation}\label{e88}
\partial^\oa P_\oa(x_Q) = 1,
\end{equation} 
while \eqref{eq:700},\eqref{eq_23} and $\ell_* \leq D$ imply that
$$P_{\oa} \in C \epsilon_1 \delta_{Q}^{|\oa| + n/p - m} \cdot \sigma(x_Q,E \cap 9 {Q}).$$
Thus, there exists $\varphi_{\oa} \in \LR$ with 
\begin{align}
\label{p801} &\varphi_\oa = 0 \; \mbox{on} \; E \cap 9{Q} ; \\
\label{p802} &J_{x_Q} (\varphi_\oa) = P_\oa; \; \mbox{and}\\
\label{p803} &\|\varphi_\oa\|_{\LR} \leq C \epsilon_1 \delta_{{Q}}^{|\oa|+n/p-m}.
\end{align}
From \eqref{eq:701},\eqref{eq_23} and \eqref{eq_23},\eqref{eq_24}, we obtain
\begin{align}
&|\partial^\beta P_\oa(x_Q)| \leq \epsilon_1^{-D} \delta_{Q}^{|\oa| - |\beta|} \quad (\beta \in \cM), \;\mbox{and}\label{e90}\\
&|\partial^\beta P_\oa(x_Q)| \leq \epsilon_1 \delta_{Q}^{|\oa| - |\beta|} \quad\;\;\; (\beta \in \cM, \beta > \oa), \label{e91}
\end{align}
while \eqref{feffy} and \eqref{eq_23} yield
\begin{equation}
\oa \notin \cA. \label{nottt}
\end{equation}

Let $x \in E \cap 9{Q}$ be arbitrary. Define $P^x_\oa = J_{x}( \varphi_\oa)$. Thus,
\begin{align}& \Bigl\lvert \partial^\beta P^x_\oa(x) - \sum_{|\gamma| \leq m-1-|\beta|} \frac{1}{\gamma!} \partial^{\beta + \gamma}P_{\oa}(x_Q)\cdot (x-x_Q)^\gamma \Bigr\rvert = \lvert \partial^\beta P^x_\oa(x) - \partial^\beta P_\oa(x)\rvert  \notag{} \\ 
& \;\; = \lvert \partial^\beta \varphi_\oa(x) - \partial^\beta J_{x_Q} (\varphi_\oa)(x) \rvert \leq C' \epsilon_1 \cdot \delta_{Q}^{|\oa|+n/p-m} \cdot \lvert x-x_Q \rvert^{m-n/p-|\beta|} \notag{} \\
& \leq C'' \epsilon_1 \cdot \delta_{Q}^{|\oa| - |\beta|} \qquad\qquad\qquad (\beta \in \cM). \label{mama}
\end{align}
(Here, we have used \eqref{p802}; \eqref{p803} and the standard Sobolev inequality; and $\lvert x-x_Q\rvert \leq C \delta_{Q}$.)
By \eqref{p801},\eqref{p803};  inserting \eqref{e88},\eqref{e91} into \eqref{mama}; and inserting \eqref{e90} into \eqref{mama}, we have
\begin{align}
& \label{e89}   P^x_\oa \in C_3 \epsilon_1 \cdot \delta_{Q}^{|\oa|+n/p-m} \cdot \sigma(x,E \cap 9{Q});& &  \\
&\label{e92}  \lvert \partial^\beta P^x_\oa(x) - \delta_{\beta\oa} \rvert \leq C_3 \epsilon_1 \cdot \delta_{Q}^{|\oa|-|\beta|}%
& &(\beta \in \cM, \beta \geq \oa); \;\mbox{and}\\
&\label{e93}  \lvert \partial^\beta P^x_\oa(x) \rvert \leq C_3 \epsilon_1^{-D} \cdot \delta_{Q}^{|\oa|-|\beta|}%
& &(\beta \in \cM).
\end{align}

We recall the defining properties (\ref{sec_ind}.\ref{mainprops1})-(\ref{sec_ind}.\ref{mainprops4}) of the polynomials $P^{x}_\alpha$:
\begin{subequations}
\begin{align} \label{pip}
& \;P^x_\alpha \in C'\epsilon \cdot \sigma(x,E) &&(\alpha \in \cA);\\
&\; \partial^\beta P^x_\alpha(x) = \delta_{\alpha \beta} && (\alpha,\beta \in \cA);\label{pip2}\\
& \lvert \partial^\beta P^x_\alpha(x)\rvert \leq C'\epsilon &&(\alpha \in \cA, \; \beta \in \cM, \; \beta>\alpha); \;\mbox{and}\label{pip3} \\
& \lvert \partial^\beta P^x_\alpha(x)\rvert \leq C' && (\alpha \in \cA,\; \beta \in \cM). \label{pip4}
\end{align}
\end{subequations}
Let $Q' \in \CZ$ be such that $x \in Q'$. Then $9Q \cap Q' \neq \emptyset$, since $x \in 9Q$. Thus $\delta_{Q'} \leq 50 \delta_Q$, thanks to the second bullet point in Lemma \ref{moregg}. Also, $\lvert \partial^\beta P^x_\alpha(x) \rvert \leq C \delta_{Q'}^{|\alpha|-|\beta|}$ for $\alpha \in \cA$,  $\beta \in \cM$, thanks to Lemma \ref{lem201}. Thus, we have $\lvert \partial^\beta P^x_\alpha(x) \rvert \leq C' \delta_Q^{|\alpha| - |\beta|}$ for every $\alpha \in \cA$ and $\beta \in \cM$ with $|\beta| \leq |\alpha|$. In combination with \eqref{pip4} this gives
\begin{equation}
\label{pip5} \lvert \partial^\beta P^x_\alpha(x) \rvert \leq C' \delta_Q^{|\alpha| - |\beta|} \quad (\alpha \in \cA, \; \beta \in \cM).
\end{equation}
(Here, we use that $\delta_Q \leq 1$ since $Q \subset Q^\circ = (0,1]^n$.)

Define the polynomial
$$ \widetilde{P}^x_\oa = P^x_\oa - \sum_{\alpha \in \cA , \alpha < \oa} \left[\partial^\alpha P^x_\oa(x)\right] \cdot P_\alpha^x.$$

From \eqref{e89},\eqref{e93},\eqref{pip}, we obtain
\begin{align}
\widetilde{P}^x_{\oa} &\in  \Bigl[ C_3 \epsilon_1 \cdot \delta_{Q}^{|\oa|+n/p-m}  + \sum_{\alpha \in \cA , \alpha < \oa} \Bigl[ C_3 \epsilon_1^{-D} \cdot \delta_Q^{|\oa|-|\alpha|} \Bigr] \cdot C' \epsilon \Bigr] \cdot \sigma(x,E \cap 9{Q}) \notag{}\\
& \subset C_4 \left[\epsilon_1 + \epsilon \epsilon_1^{-D}\right] \cdot \delta_{Q}^{|\oa|+n/p - m} \cdot \sigma(x,E \cap 9{Q}). \label{alm1} 
\end{align}
(Here, we have also used $\sigma(x,E) \subset \sigma(x, E \cap 9Q)$ and $|\alpha| \leq m - 1 < m-n/p$.)

For each $\beta \in \cA$, $\beta < \oa$, by \eqref{pip2} we have 
\begin{equation} \label{alm2}
\partial^\beta \widetilde{P}^x_{\oa}(x) = \partial^\beta P_{\oa}^x(x) - \partial^\beta P_{\oa}^x(x) = 0.
\end{equation}

Let $\beta \in \cM$ with $\beta \geq \oa$ be given. Note that if $\alpha \in \cA$ and $\alpha < \oa$, then $\alpha < \beta$. We estimate
\begin{align}
|\partial^\beta \widetilde{P}^x_\oa(x) - \delta_{\beta \oa}| &\leq |\partial^\beta P^x_\oa(x) - \delta_{\beta \oa}| \;  + \sum_{\alpha \in \cA , \alpha < \oa} | \partial^\alpha P^x_{\oa}(x)| \cdot |\partial^{\beta} P_\alpha^x(x)| \notag{} \\%
&  \leq C_3 \epsilon_1 \cdot \delta_{Q}^{|\oa|-|\beta|} \quad\;+ \sum_{\alpha \in \cA, \alpha < \oa} \left[ C_3 \epsilon_1^{-D} \cdot \delta_{Q}^{|\oa|-|\alpha|} \right] \cdot C' \epsilon \;\;\; (\mbox{thanks to} \; \eqref{e92},\eqref{e93},\eqref{pip3}) \notag{} \\%
&\leq C_4 \left[\epsilon_1 + \epsilon_1^{-D} \epsilon \right] \cdot \delta_Q^{|\oa|-|\beta|} \qquad\qquad (\mbox{since} \; \alpha < \beta \implies  |\alpha| \leq |\beta|). \label{alm3}
\end{align}

On the other hand, consider $\beta \in \cM$ arbitrary. We estimate
\begin{align}
|\partial^\beta \widetilde{P}^x_\oa(x)| &\leq |\partial^\beta P^x_\oa(x) | \quad\;\;\; + \sum_{\alpha \in \cA , \alpha < \oa} | \partial^\alpha P^x_{\oa}(x)| \cdot |\partial^{\beta} P_\alpha^x(x)| \notag{}\\%
&  \leq C_3 \epsilon_1^{-D}\delta_{Q}^{|\oa|-|\beta|} \; + \sum_{\alpha \in \cA, \alpha < \oa} \left[C_3 \epsilon_1^{-D} \cdot \delta_{Q}^{|\oa|-|\alpha|} \right] \cdot C' \delta_Q^{|\alpha|-|\beta|} \;\;\;\; (\mbox{thanks to} \; \eqref{e93},\eqref{pip5})\notag{} \\%
&\leq  C_4 \epsilon_1^{-D} \cdot \delta_Q^{|\oa| - |\beta|}. \label{alm4}
\end{align}

Define $\oA = \{\alpha \in \cA : \alpha <\oa\} \cup \{\oa\}$. From \eqref{nottt} and the definition of the order on multi-index sets, we have
\begin{equation} \label{lesss}
\oA < \cA.
\end{equation}

For small enough $\epsilon$ and $\epsilon_1$ such that $\epsilon \leq \epsilon_1^{D+1}$, by \eqref{alm3} we have $|\partial^\oa \widetilde{P}^x_\oa(x)| \geq 1/2$. Thus, the following polynomial is well-defined:
$$\widehat{P}^x_\oa = \left[\partial^\oa \widetilde{P}^x_\oa(x)\right]^{-1} \cdot \widetilde{P}^x_\oa.$$
Thanks to \eqref{alm1}-\eqref{alm4}, these properties are immediate:
\begin{subequations}
\begin{align}
&\hP_{\oa}^x \in 2 C_4 \cdot \left[\epsilon_1 + \epsilon \epsilon_1^{-D}\right] \cdot \delta_{Q}^{|\oa|+n/p - m} \sigma(x,E \cap 9{Q}); && \label{palm1} \\
& \partial^{\beta} \hP^x_{\oa}(x) = \delta_{\beta \oa} && (\beta \in \oA); \label{palm2}\\
& |\partial^\beta \hP^x_\oa(x)| \leq 2C_4 \cdot  \left[\epsilon_1 + \epsilon_1^{-D} \epsilon \right] \cdot \delta_{Q}^{|\oa| - |\beta|} && (\beta \in \cM, \; \beta > \oa); \;\mbox{and}\label{palm3} \\
& |\partial^\beta \hP^x_\oa(x)| \leq 2C_4 \cdot \epsilon_1^{-D} \cdot \delta_{Q}^{|\oa|-|\beta|} && (\beta \in \cM). \label{palm4}
\end{align}
\end{subequations}

For each $\alpha \in \oA \setminus \{\oa\}$, we define the polynomial
$$\widehat{P}^x_\alpha = P^{x}_\alpha - \left[\partial^\oa P^{x}_\alpha(x) \right] \cdot \widehat{P}^x_\oa.$$
Then, we have
\begin{equation}
\partial^\oa \widehat{P}^x_\alpha(x) = \partial^{\oa} P^x_\alpha(x) - \left[ \partial^{\oa} P^x_\alpha(x)  \right] \cdot \partial^{\oa} \hP^x_{\oa}(x) = 0 \;\;\; (\mbox{thanks to} \; \eqref{palm2}). \label{psalm1}
\end{equation}
For $ \beta \in \oA \setminus \{\oa\}$, we have 
\begin{equation}
\partial^\beta \widehat{P}^x_\alpha(x) = \partial^{\beta} P^x_\alpha(x) -  \left[ \partial^{\oa} P^x_\alpha(x)\right] \cdot \partial^{\beta} \hP^x_{\oa}(x) = \delta_{\beta \alpha} \;\;\;(\mbox{thanks to} \; \eqref{pip2}, \eqref{palm2}). \label{psalm2}
\end{equation}
Note that $\alpha < \oa$, because $\alpha \in \oA \setminus \{\oa\}$. Thus, for each $\beta \in \cM$ with $\beta > \alpha$, we obtain
\begin{align}
|\partial^\beta \widehat{P}^x_\alpha(x)| &\leq C' \epsilon + \left[C' \epsilon \right] \cdot  2C_4 \epsilon_1^{-D} \cdot\delta_Q^{|\oa| - |\beta|} \quad (\mbox{thanks to} \; \eqref{pip3}, \eqref{palm4}) \notag{}\\
& \leq C_5 \epsilon \cdot \epsilon_1^{-D} \cdot \delta_Q^{|\alpha|-|\beta|} \;\;\; (\mbox{since} \; |\alpha| \leq |\oa|, \; |\alpha| \leq |\beta|). \label{psalm3}
\end{align}
Finally, from \eqref{pip}, \eqref{pip4}, \eqref{palm1}, we obtain
\begin{equation*}\widehat{P}^x_\alpha \in \left[ C' \epsilon + \left[C'\right] \cdot 2C_4 \cdot (\epsilon_1 + \epsilon \epsilon_1^{-D}) \cdot \delta_{Q}^{|\oa|+n/p - m} \right]\cdot \sigma(x,E \cap 9{Q}). \end{equation*}
We may assume that $\epsilon \leq \epsilon_1$. Note that $|\alpha| \leq |\oa|$, $ |\alpha| + n/p - m \leq 0$ and $\delta_Q \leq 1$. Therefore,
\begin{equation} \widehat{P}^x_\alpha \in C_5 \cdot \left[\epsilon_1 + \epsilon \epsilon_1^{-D}\right] \cdot \delta_{Q}^{|\alpha|+n/p - m}\cdot \sigma(x,E \cap 9{Q}). \label{psalm4} \end{equation}

Recall that $x \in E \cap 9Q$ was arbitrary. Thus, from \eqref{palm1},\eqref{palm2},\eqref{palm3} and \eqref{psalm1}-\eqref{psalm4}, it follows that $(\hP^x_\alpha)_{\alpha \in \oA}$ forms an $(\oA,x,C_6 \cdot[\epsilon_1 + \epsilon_1^{-D} \epsilon],\delta_{{Q}^+})$-basis for $\sigma(x,E\cap 9{Q})$ for each $x \in E \cap 9{Q}$, hence for $\sigma(x,E\cap 3{Q}^+)$ for each  $x \in E \cap 3{Q}^+$. (Here, we use that $E \cap 3{Q}^+ \subset E \cap 9{Q}$.)

Fix $\epsilon_1$ small enough so that the preceding arguments hold, and so that $\epsilon_1 \leq \epsilon_0/2C_6$. We may assume that $\epsilon \leq \epsilon_1^{D+1}$. Therefore, $\sigma(x,E \cap 3 {Q}^+)$ contains an $(\oA,x,\epsilon_0,\delta_{{Q}^+})$-basis for each $x \in E \cap 3 {Q}^+$. Since $\oA < \cA$, by definition the cube ${Q}^+$ is OK, which contradicts the hypothesis that ${Q} \in \CZ$. This completes the proof of \eqref{reffy}, which was the unverified inequality in \eqref{ess2}. The proof of Lemma \ref{esslem} is now complete.
\end{proof}

\section{The Jet of a Near Optimal Extension I}\label{sec_ac}
\setcounter{equation}{0}

\subsection{Another Sobolev Inequality}

Proposition \ref{keygeom} produces for each $Q_\nu$ an associated keystone cube $\key(Q_\nu) = Q^\sharp_{\kappa(\nu)}$ and a finite sequence of CZ cubes $\cS_\nu = (Q_{\nu,1},\ldots,Q_{\nu,L_\nu})$, such that
\begin{equation}
\label{t0}
\begin{aligned}
&Q_\nu = Q_{\nu,1} \leftrightarrow \cdots \leftrightarrow Q_{\nu,L_\nu}= Q^\sharp_{\kappa(\nu)}, \;\mbox{with} \\
&\delta_{Q_{\nu,k}} \leq C (1-c)^{k-j} \delta_{Q_{\nu,j}} \;\; \mbox{for all} \; 1 \leq j \leq k \leq L_\nu,\; \mbox{where} \\
& \quad\quad C>0 \; \mbox{and} \; 0<c<1 \; \mbox{are universal constants}.
\end{aligned}
\end{equation}
\begin{lem}
\label{globsob}
Let $F \in L^{m,p}(Q^\circ)$. Then
$$\sum_{\nu=1}^{\nu_{\max}} \left|J_{x_\nu} (F) - J_{x^\sharp_{\kappa(\nu)}} (F)\right|_\nu^p \lesssim \|F\|^p_{L^{m,p}(Q^\circ)}.$$
\end{lem}

\begin{proof}
By applying (\ref{sec_rep}.\ref{ti}), we have
\begin{align}
\label{t1} X:=\sum_{\nu=1}^{\nu_{\max}} \left|J_{x_\nu}(F) - J_{x^\sharp_{\kappa(\nu)}}(F)\right|_\nu^p &\lesssim  \sum_{\nu=1}^{\nu_{\max}} \sum_{|\alpha| \leq m-1}\left\lvert \partial^\alpha \left[J_{x_\nu}(F) -  J_{x^\sharp_{\kappa(\nu)}}(F) \right](x^\sharp_{\kappa(\nu)}) \right|^p \delta_{Q_\nu}^{n - (m-|\alpha|)p}.
\end{align}

Let $x_{\nu,k}$ denote the representative point for $Q_{\nu,k}$. We fix some universal constant $\epsilon' \in (0, 1-n/p)$. Note that $x_{\nu,L_\nu} = x_{\kappa(\nu)}^\sharp$, since $Q_{\nu,L_{\nu}} = Q^\sharp_{\kappa(\nu)}$. Thus the right-hand side of \eqref{t1} is given by
\begin{align}
&\sum_{\nu=1}^{\nu_{\max}} \sum_{|\alpha| \leq m-1}\delta_{Q_\nu}^{n-(m-|\alpha|)p} \left| \sum_{k=1}^{L_\nu-1} \partial^\alpha \left[ J_{x_{\nu,k}}(F) - J_{x_{\nu,k+1}}(F)\right](x_{\kappa(\nu)}^\sharp) \cdot \delta_{Q_{\nu,k}}^{\frac{n}{p} - (m-|\alpha|) + \epsilon'} \cdot \delta_{Q_{\nu,k}}^{-\frac{n}{p} + (m-|\alpha|) - \epsilon'} \right|^p \notag{}\\
\leq &\sum_{\nu=1}^{\nu_{\max}} \sum_{|\alpha| \leq m-1} \delta_{Q_\nu}^{n-(m-|\alpha|)p} \left[ \sum_{k=1}^{L_\nu-1} \left|\partial^\alpha \left[ J_{x_{\nu,k}}(F) - J_{x_{\nu,k+1}}(F)\right](x_{\kappa(\nu)}^\sharp)\right|^p \cdot\delta_{Q_{\nu,k}}^{n - (m-|\alpha|)p + \epsilon'p}  \right] \notag{} \\
& \qquad\qquad\qquad\qquad\;\;\cdot \left[ \sum_{k=1}^{L_\nu - 1} \delta_{Q_{\nu,k}}^{-\frac{np'}{p} + (m-|\alpha|)p' - \epsilon'p'}\right]^{p/p'}\notag{}\\
&\left(\mbox{by H\"older's inequality}; \;\mbox{here,} \; p' \; \mbox{is the dual exponent to} \; p, \; \mbox{so that} \; \frac{1}{p'} + \frac{1}{p}=1\right).\notag{}
\end{align}
From \eqref{t0}, we have $\delta_{Q_{\nu,k}} \leq C (1-c)^k \delta_{Q_\nu}$. Thus, using $\epsilon' < 1-n/p$, we obtain
\begin{align}
X &\lesssim \sum_{\nu=1}^{\nu_{\max}} \sum_{|\alpha| \leq m-1} \delta_{Q_\nu}^{n - (m-|\alpha|)p} \cdot \delta_{Q_\nu}^{-n + (m-|\alpha|)p - \epsilon' p} \cdot\left[ \sum_{k=1}^{L_\nu-1}  \left|\partial^\alpha \left[J_{x_{\nu,k}}(F) - J_{x_{\nu,k+1}}(F)\right](x_{\kappa(\nu)}^\sharp)\right|^p \cdot \delta_{Q_{\nu,k}}^{n - (m-|\alpha|)p +\epsilon'p} \right] \notag{} \\
& = \sum_{\nu=1}^{\nu_{\max}} \delta_{Q_\nu}^{-\epsilon' p} \sum_{k=1}^{L_\nu-1} \sum_{|\alpha| \leq m-1} \left| \partial^\alpha \left[ J_{x_{\nu,k}}(F) - J_{x_{\nu,k+1}}(F)\right](x^\sharp_{\kappa(\nu)}) \right|^p \cdot \delta_{Q_{\nu,k}}^{n-(m-|\alpha|)p + \epsilon' p} \notag{}\\
& = \sum_{\nu=1}^{\nu_{\max}} \delta_{Q_\nu}^{-\epsilon' p} \sum_{k=1}^{L_\nu-1} \delta_{Q_{\nu,k}}^{\epsilon' p} \left\{\sum_{|\alpha| \leq m-1} \left| \partial^\alpha \left[ J_{x_{\nu,k}}(F) - J_{x_{\nu,k+1}}(F)\right](x^\sharp_{\kappa(\nu)}) \right|^p \cdot \delta_{Q_{\nu,k}}^{n-(m-|\alpha|)p} \right\}. \label{t2}
\end{align}
Thanks to \eqref{t0} it follows that $\lvert x_{\nu,k} - x^\sharp_{\kappa(\nu)} \rvert \leq C \delta_{Q_{\nu,k}}$ (recall that $x_{\nu,k} \in Q_{\nu,k}$ and $x^\sharp_{\kappa(\nu)} \in Q^\sharp_{\kappa(\nu)}$). Therefore, continuing from \eqref{t2} and applying (\ref{sec_not}.\ref{ti0}), we have
\begin{align}
X &\lesssim \sum_{\nu=1}^{\nu_{\max}} \delta_{Q_\nu}^{- \epsilon' p} \sum_{k=1}^{L_\nu - 1} \delta_{Q_{\nu,k}}^{\epsilon' p} \left\{ \sum_{|\alpha| \leq m-1} \left|\partial^\alpha \left[ J_{x_{\nu,k}}(F) - J_{x_{\nu,k+1}}(F) \right](x_{\nu,k}) \right|^p \delta_{Q_{\nu,k}}^{n-(m-|\alpha|)p} \right\}\notag{} \\
	& = \sum_{\onu \leftrightarrow \onu'} \sum_{\nu=1}^{\nu_{\max}} \sum_{k=1}^{L_\nu-1} \delta_{Q_\nu}^{-\epsilon' p} \delta_{Q_{\onu}}^{\epsilon' p} \left| J_{x_{\onu}}(F) - J_{x_{\onu'}}(F)\right|^p_{x_{\onu},\delta_{Q_{\onu}}} \mathbbm{1}_{\tiny Q_{\onu} = Q_{\nu,k} \;\wedge\; Q_{\onu'} = Q_{\nu,k+1}  }.\label{t3}
\end{align}

Thanks to \eqref{t0}, any given $Q_{\onu}$ can arise as $Q_{\nu,k}$ (for fixed $\nu$) for at most $C$ distinct $k$. Thus, from \eqref{t3} we obtain
\begin{equation*}
 X \lesssim %
\sum_{\onu \leftrightarrow \onu'} \sum_{\nu=1}^{\nu_{\max}} \left( \delta_{Q_{\onu}} / \delta_{Q_\nu} \right)^{\epsilon' p} \cdot %
\left| J_{x_{\onu}}(F) - J_{x_{\onu'}}(F) \right|^p_{x_{\onu},\delta_{Q_{\onu} }} \cdot %
\mathbbm{1}_{\tiny \exists k \; \mbox{s.t.} \; Q_{\onu} = Q_{\nu,k} \; \wedge \; Q_{\onu'} = Q_{\nu,k+1} }.
\end{equation*}
If there exists $k$ such that $Q_{\onu} = Q_{\nu,k}$, then $\delta_{Q_\nu} \geq c \delta_{Q_{\onu}}$ and $Q_{\onu} \subset C Q_\nu$, as follows from \eqref{t0}. Therefore, we obtain
\begin{equation*}
X \lesssim %
\sum_{\onu \leftrightarrow \onu'} \left| J_{x_{\onu}}(F) - J_{x_{\onu'}}(F) \right|^p_{x_{\onu},\delta_{Q_{\onu}} } \cdot  \sum \left\{ \left( \frac{\delta_{Q_{\onu}} }{ \delta_{Q} } \right)^{\epsilon' p} : \; \mbox{dyadic} \; Q \; \mbox{s.t.} \; \delta_Q \geq c \delta_{Q_{\onu}} \;\mbox{and} \;Q_{\onu} \subset C Q\right\}.
\end{equation*}
The inner sum is at most $C$, for each fixed $\onu$; hence,
\begin{align*}
X &\lesssim \sum_{\onu \leftrightarrow \onu'} \left| J_{x_{\onu}}(F) - J_{x_{\onu'}}(F) \right|^p_{x_{\onu},\delta_{Q_{\onu}} } \lesssim \sum_{\onu \leftrightarrow \onu'} \|F\|^p_{L^{m,p}([1.1 Q_{\onu} \cup 1.1 Q_{\onu'}] \cap Q^\circ)}\\
& (\mbox{thanks to} \; (\ref{sec_not}.\ref{SET}); \; \mbox{we also use} \; x_{\onu} \in Q_{\onu}, \; x_{\onu'} \in Q_{\onu'} \;\mbox{and the Good Geometry of the CZ cubes})\\
& \lesssim \|F\|^p_{L^{m,p}(Q^\circ)} \\
& (\mbox{thanks to the first bullet point in Lemma \ref{moregg} and Good Geometry of the CZ cubes}).
\end{align*}
This completes the proof of Lemma \ref{globsob}.
\end{proof}

\subsection{New Constraints on Extensions}

Every Whitney field $\vec{P} \in Wh(E')$ has a natural restriction, denoted $\vP^\sharp \in Wh(E^\sharp)$, to the keystone representative points $E^\sharp = \{x^\sharp_{\mu} : \mu=1,\ldots,\mu_{\max}\}$. That is,
\begin{equation} \label{conv1} P^\sharp_\mu = P_\nu,\; \mbox{where} \; \nu \in \{1,\ldots,\nu_{\max}\} \; \mbox{is such that} \; Q_\nu = Q^\sharp_\mu.
\end{equation} 

\begin{defn}
Let $\vP = (P_\nu)_{\nu=1}^{\nu_{\max}} \in \Wh(E')$ and $P \in \cP$. We say that $\vP$ is Coherent with $P$ provided that
$$P_{1} = P,\;\; \mbox{and} \;\;  \partial^\alpha P^\sharp_{\mu}(x^\sharp_\mu) = \partial^\alpha P(x^\sharp_{\mu}) \;\;\; \mbox{for} \; \mu=1,\ldots,\mu_{\max}, \; \mbox{and} \; \alpha \in \cA.$$
We say that $\vP$ is Constant Along Paths provided that
$$P_{\nu} = P^\sharp_{\kappa(\nu)} \;\; \mbox{for} \; \nu=2,\ldots,\nu_{\max}.$$
Similarly, we say that a function $F \in \LR$ is Coherent with $P$ when $J_{E'}(F)$ is Coherent with $P$, while we say that $F$ is Constant Along Paths when $J_{E'}(F)$ is Constant Along Paths.
\end{defn}

The main result of this section states that there always exist near optimal extensions of the data $(f,P)$ that are Coherent with $P$ and Constant Along Paths.
\begin{prop}
\label{special} Given $(f,P) \in L^{m,p}(E;z)$, there exists $\widehat{F} \in \LR$ with the following properties:
\begin{enumerate}[(i)]
\item $\widehat{F}$ extends the data $(f,P)$.
\item $\|\widehat{F}\|_{\LR} \leq C \cdot \|(f,P)\|_{L^{m,p}(E;z)}$.
\item $\widehat{F}$ is Coherent with $P$.
\item $\widehat{F}$ is Constant Along Paths.
\end{enumerate}
\end{prop}
\begin{proof} By subtracting $P|_E$ from $f$ and subtracting $P$ from the $\widehat{F}$ that we seek, we may assume that $P=0$. For the remainder of the proof of Proposition \ref{special}, $\|(\cdot , \cdot)\|$ denotes the trace seminorm on $L^{m,p}(E;z)$.

First we draw some immediate conclusions from subsection \ref{aux}. Denote the $(C,C\epsilon)$ near-triangular matrix $A^\nu = A^{x_\nu}$ from (\ref{sec_ind}.\ref{mainprops}) for each $\nu=1,\ldots,\nu_{\max}$. Then (\ref{sec_ind}.\ref{p101})-(\ref{sec_ind}.\ref{p103}) and (\ref{sec_ind}.\ref{mainprops})-(\ref{sec_ind}.\ref{mainprops4}) imply that the following properties are satisfied by the function
\begin{align}
\label{eq_11} & \varphi_{\nu,\alpha} := \sum_{\gamma \in \cA} A^\nu_{\alpha \gamma} \cdot \varphi_\gamma \;\;\; (\alpha \in \cA) :\\
\label{eq:400a} &\varphi_{\nu,\alpha} = 0 \; \mbox{on} \; E,\\
\label{eq:401a} &J_{x_{\nu}}(\varphi_{\nu,\alpha}) = P^{x_{\nu}}_{\alpha}, \; \mbox{and}\\
\label{eq:402} &\partial^\beta \varphi_{\nu,\alpha}(x_{\nu}) = \delta_{\alpha \beta} \quad\quad (\beta \in \cA).
\end{align}
Note that (\ref{sec_ind}.\ref{p103}) states that
\begin{equation}
\|\varphi_\gamma\|_{\LR} \leq \epsilon \qquad (\gamma \in \cA). \label{eq:401}
\end{equation}
From \eqref{eq:401a} and Lemma \ref{lem201}, we also have
\begin{equation} \label{eq:404a}
|\partial^\beta \varphi_{\nu,\alpha}(x_{\nu})| = |\partial^\beta P^{x_\nu}_\alpha(x_{\nu})| \leq C \delta_{Q_\nu}^{|\alpha|-|\beta|} \quad (\alpha \in \cA, \beta \in \cM, \nu=1,\ldots,\nu_{\max}).
\end{equation}
Since the inverse of a near-triangular matrix is near-triangular with comparable parameters, by \eqref{eq_11} we have
\begin{equation}
\varphi_{\nu',\alpha} = \sum_{\beta \in \cA} \omega^{\nu \nu'}_{\alpha\beta} \varphi_{\nu,\beta}, \quad \mbox{where} \; |\omega^{\nu \nu'}_{\alpha \beta}| \leq C \; \mbox{for} \; \alpha, \beta \in \cA, \; \nu,\nu' \in \{1,\ldots,\nu_{\max}\}. \label{doubtroub}
\end{equation} 

Let $F \in \LR$ satisfy 
\begin{align}
\label{p502} &F= f \; \mbox{on} \; E \; \; \mbox{and} \;\; J_{z} (F) = 0; \; \mbox{and}\\
\label{p503} &\|F\|_{\LR} \leq 2 \|(f,0)\|.
\end{align}
We first modify $F$ to a function $\widetilde{F}$ that is Coherent with $0 \in \cP$. Since $J_{z} (F) = 0$ and $|z-x_\nu| \leq 1$ for $\nu=1,\ldots,\nu_{\max}$, the Sobolev inequality (\ref{sec_not}.\ref{SET}) yields a basic estimate on the size of the derivatives:
\begin{equation}
\label{eq200}
\lvert \partial^\alpha F(x_\nu) \rvert = \lvert \partial^\alpha (F-J_{z} (F)) (x_\nu) \rvert \lesssim \|F\|_{\LR} \quad (\alpha \in \cM, \nu=1,\ldots,\nu_{\max}).
\end{equation}

We define
\begin{equation}
\label{stuffy} F_\nu := F - \sum_{\alpha \in \cA} \partial^\alpha F(x_{\nu})\varphi_{\nu,\alpha} \;\;\; \mbox{for each} \; \nu=1,\ldots,\nu_{\max}.
\end{equation}
Therefore,
\begin{equation}
\label{stff} F_\nu \oeq{\eqref{eq:400a}, \eqref{p502}} f \; \mbox{on} \; E \;\; \mbox{and} \;\; \partial^\beta F_\nu(x_{\nu}) \oeq{\eqref{eq:402}} 0 \;\; \mbox{for all} \; \beta \in \cA.
\end{equation}
Note that $x_1 = z$. Thanks to \eqref{p502} and \eqref{stuffy}, it follows that $F_{1} = F$.

Recall the partition of unity $\{\theta_\nu\}_{\nu=1}^{\nu_{\max}}$ that satisfies \textbf{(POU1-5)}. We define $\widetilde{F} \in L^{m,p}(Q^\circ)$ by
$$\widetilde{F}(x) = \sum_{\nu=1}^{\nu_{\max}} F_\nu(x) \theta_\nu(x) \quad (x \in Q^\circ).$$ 
The following properties hold:
\begin{equation}
\begin{aligned}
&\widetilde{F} = f \; \mbox{on} \; E \qquad ( \mbox{thanks to \textbf{(POU5)} and \eqref{stff}} ). \\
&J_{z} \widetilde{F}  = J_{x_1} \widetilde{F}  = J_{x_1} F_{1} = J_{x_1} F \oeq{} 0 \qquad (\mbox{thanks to \textbf{(POU4)} and \eqref{p502}}). \\
&\partial^\beta \widetilde{F}(x_{\nu}) = \partial^\beta F_\nu(x_{\nu}) = 0 \;\; \mbox{for} \; \nu=1,\ldots,\nu_{\max}, \; \beta \in \cA \quad (\mbox{thanks to \textbf{(POU4)} and \eqref{stff}}).
\end{aligned}
\label{eq206}
\end{equation}
Thus, $\widetilde{F}$ extends the data $(f,0)$ and $\widetilde{F}$ is Coherent with $0 \in \cP$. We now turn to estimating the seminorm $\|\widetilde{F}\|_{L^{m,p}(Q^\circ)}$. A straightforward application of Lemma \ref{whittrick} shows that
\begin{align}
\|\widetilde{F}\|^p_{L^{m,p}(Q^\circ)} &\lesssim \sum_{\nu=1}^{\nu_{\max}}\|F_\nu\|^p_{L^{m,p}(1.1Q_\nu)} + \sum_{ \nu \leftrightarrow \nu'} \sum_{\ob \in \cM} \left\lvert \partial^{\ob} (F_\nu - J_{x_{\nu'}}F_{\nu'})(x_{\nu}) \right\rvert^p \delta_{Q_\nu}^{n+(|\ob|-m)p}. \notag{} \\ 
&\lesssim \sum_{ \nu \leftrightarrow \nu'} \left[ \|F_\nu\|^p_{L^{m,p}(1.1Q_\nu \cup 1.1 Q_{\nu'})} +  \sum_{\ob \in \cM} \left\lvert \partial^{\ob} (F_\nu - F_{\nu'})(x_{\nu}) \right\rvert^p \delta_{Q_\nu}^{n+(|\ob|-m)p} \right], \label{nnn} \end{align}
thanks to the Sobolev inequality (\ref{sec_not}.\ref{SET}). (Here, (\ref{sec_not}.\ref{SET}) applies because $x_\nu \in Q_\nu$, $x_{\nu'} \in Q_{\nu'}$ and $\nu \leftrightarrow \nu'$.)

Let $\nu, \nu' \in \{1,\ldots,\nu_{\max}\}$ satisfy $\nu \leftrightarrow \nu'$. To start, we bound the terms from the innermost sum on the right-hand side of \eqref{nnn}. Using (\ref{doubtroub}) and \eqref{stuffy}, we write
\begin{equation}
F_{\nu'} = F - \sum_{\alpha,\beta \in \cA}  \partial^\alpha F(x_{\nu'}) \omega^{\nu \nu'}_{\alpha \beta} \varphi_{\nu,\beta}.\label{eq202}
\end{equation}
Since $\cA$ is monotonic, if $\gamma \in \cA$ and $|\gamma'| \leq m-1-|\gamma|$, then $\gamma + \gamma' \in \cA$, and hence $\partial^{\gamma + \gamma'} (F_{\nu'})(x_{\nu'})= 0$, thanks to \eqref{stff}. Thus by a Taylor expansion, we have
\begin{equation}
\label{eq203} \partial^\gamma J_{x_{\nu'}} (F_{\nu'})(x_{\nu}) = \sum_{|\gamma'| \leq m-1 - |\gamma|} \mbox{coeff}(\gamma,\gamma') \cdot \partial^{\gamma+\gamma'} F_{\nu'}(x_{\nu'}) \cdot (x_{\nu} - x_{\nu'})^{\gamma'} = 0 \quad (\gamma \in \cA).
\end{equation}
For each $\ob \in \cM$, we may now bound
\begin{align}
\lvert \partial^{\ob} (F_\nu &- F_{\nu'})(x_{\nu}) \rvert = \Bigl\lvert \sum_{\beta \in \cA}  \partial^\beta F(x_{\nu}) \partial^{\ob} \varphi_{\nu,\beta}(x_{\nu}) - \sum_{\alpha, \beta \in \cA}  \partial^\alpha F(x_{\nu'})\omega^{\nu\nu'}_{\alpha \beta}  \partial^{\ob} \varphi_{\nu,\beta}(x_{\nu}) \Bigr\rvert \notag{} \\ 
& (\mbox{from \eqref{stuffy} and \eqref{eq202}}) \notag{}\\
&\qquad\qquad\quad \leq \sum_{\beta \in \cA} \left| \partial^\ob \varphi_{\nu,\beta}(x_{\nu}) \right| \left[ \Bigl\lvert \partial^\beta F(x_{\nu}) - \sum_{\alpha \in \cA} \partial^\alpha F(x_{\nu'})  \omega^{\nu\nu'}_{\alpha \beta}  \Bigr\rvert \right]  \notag{} \\
&\qquad\qquad\quad = \sum_{\beta \in \cA} \left| \partial^{\ob} \varphi_{\nu,\beta}(x_{\nu}) \right| \left[ \Bigl\lvert \partial^\beta F_{\nu'}(x_{\nu}) \Bigr\rvert \right] \notag{}\\
& (\mbox{from \eqref{eq:402} and \eqref{eq202}}) \notag{}\\
&\qquad\qquad\quad = \sum_{\beta \in \cA} \left| \partial^\ob \varphi_{\nu,\beta}(x_{\nu}) \right| \left[ \Bigl\lvert \partial^\beta (F_{\nu'} - J_{x_{\nu'}} (F_{\nu'}) )(x_{\nu}) \Bigr\rvert \right]  \notag{}\\
&(\mbox{from \eqref{eq203}}) \notag{}\\
& \qquad\qquad\quad \lesssim \sum_{\beta \in \cA} \delta_{Q_\nu}^{|\beta| - |\ob|} \delta_{Q_\nu}^{m-n/p-|\beta|} \|F_{\nu'}\|_{L^{m,p}(1.1 Q_\nu \cup 1.1 Q_{\nu'})} \notag{} \\
& (\mbox{from \eqref{eq:404a} and the Sobolev inequality (\ref{sec_not}.\ref{SET})} ) \notag{} \\
& \qquad\qquad\quad \lesssim \delta_{Q_\nu}^{m-n/p-|\ob|} \|F_{\nu'}\|_{L^{m,p}(1.1 Q_\nu \cup 1.1 Q_{\nu'})}. \label{eq204}
\end{align}
Next, we use $|A^\nu_{\alpha\gamma}| \leq C$, \eqref{eq_11}, \eqref{eq200}, \eqref{stuffy}, the first bullet point in Lemma \ref{moregg} and the Good Geometry of the CZ cubes, to obtain
\begin{align}
\label{eq204a}
&\sum_{\nu \leftrightarrow \nu'} \|F_\nu\|^p_{L^{m,p}(1.1Q_\nu \cup 1.1 Q_{\nu'})} \lesssim \sum_{\nu \leftrightarrow \nu'} \left[\|F\|^p_{L^{m,p}(1.1Q_\nu \cup 1.1 Q_{\nu'})} + \sum_{\gamma \in \cA} \|\varphi_{\gamma}\|_{L^{m,p}(1.1Q_\nu \cup 1.1 Q_{\nu'})}^p \|F\|^p_{\LR} \right] \notag{} \\
& \lesssim  \|F\|^p_{\LR} + \|F\|^p_{\LR} \sum_{\gamma \in \cA} \|\varphi_{\gamma}\|^p_{\LR} \olesssim{\eqref{eq:401}} \|F\|^p_{\LR}.
\end{align}
Finally, inserting \eqref{eq204} and \eqref{eq204a} into \eqref{nnn}, we have
\begin{equation}\label{eq205}
\|\widetilde{F}\|_{L^{m,p}(Q^\circ)} \lesssim \|F\|_{\LR}.
\end{equation}

Next we modify $\widetilde{F}$ to be Constant Along Paths, without ruining \eqref{eq206} or the control on the seminorm. For each $\nu=2,\ldots,\nu_{\max}$, there exists $\otheta_\nu \in C^\infty_0(Q_\nu)$ that satisfies
\begin{align}
\label{p601} & \otheta_\nu \equiv 1 \; \mbox{in a neighborhood of}  \; x_{\nu}; \\
\label{p602} & \otheta_\nu \equiv 0 \; \mbox{in a neighborhood of} \; E \cup \{x_{\nu'} : \nu' \in \{1,2,\ldots,\nu_{\max}\}, \; \nu' \neq \nu\}; \;\mbox{and}\\
\label{p603} & \lvert \partial^\alpha \otheta_\nu \rvert \lesssim \delta_{Q_\nu}^{-|\alpha|} \; \mbox{when} \; |\alpha| \leq m.
\end{align} 
(Here, the conditions on $E'$ from Lemma \ref{basegeom} are used.) We define $\overline{F} \in L^{m,p}(Q^\circ)$ by
$$\overline{F} = \widetilde{F} + \sum_{\nu=2}^{\nu_{\max}} \otheta_\nu \cdot \left[J_{x_{\kappa(\nu)}^\sharp}(\widetilde{F}) - J_{x_{\nu}}(\widetilde{F}) \right].$$
It is simple to check that

\begin{itemize}
\item[(a)] $\overline{F} = \widetilde{F} + 0 = f$ on $E$; this follows from \eqref{eq206} and \eqref{p602}.
\item[(b)] $\overline{F}$ is Coherent with $0 \in \cP$; in particular, $J_z \overline{F} = 0$; this follows from \eqref{eq206}, \eqref{p602} and \textbf{(K3)} from Proposition \ref{keygeom}, which implies that $x^\sharp_{\kappa(\nu)} = x_\nu$ when $Q_{\nu}$ is keystone.
\item[(c)] $J_{x_{\nu}} (\overline{F}) = J_{x_{\kappa(\nu)}^\sharp}(\overline{F})$ for $\nu=2,\ldots,\nu_{\max}$; this follows from \eqref{p601} and \eqref{p602} together with \textbf{(K3)} from Proposition \ref{keygeom}.
\end{itemize}

We now estimate the seminorm
\begin{align*}
&\|\overline{F}\|^p_{L^{m,p}(Q^\circ)} \lesssim \|\widetilde{F}\|^p_{L^{m,p}(Q^\circ)} + \sum_{\nu=2}^{\nu_{\max}} \sup_{x \in Q_\nu} \left[ \sum_{|\alpha + \beta| = m} \lvert \partial^\alpha \otheta_\nu(x) \rvert^p \left|\partial^{\beta} \left( J_{x_\nu} (\widetilde{F}) - J_{x^\sharp_{\kappa(\nu)}} (\widetilde{F})\right) (x)\right|^p \right] \delta_{Q_\nu}^n \\
&(\mbox{from} \; \supp(\otheta_\nu) \subset Q_\nu) \\
&\qquad\qquad\;\lesssim \|\widetilde{F}\|^p_{L^{m,p}(Q^\circ)} + \sum_{\nu=2}^{\nu_{\max}} \sup_{x \in Q_\nu} \left[ \sum_{\substack{|\beta| \leq m}} \delta_{Q_\nu}^{n-(m-|\beta|)p} \left|\partial^{\beta} \left(J_{x_\nu} (\widetilde{F}) - J_{x^\sharp_{\kappa(\nu)}}(\widetilde{F}) \right) (x) \right|^p \right]\\
&(\mbox{from} \;\eqref{p603})\\
&\qquad\qquad\;\lesssim \|\widetilde{F}\|^p_{L^{m,p}(Q^\circ)} + \sum_{\nu=2}^{\nu_{\max}}  \left|J_{x_\nu} (\widetilde{F}) - J_{x^\sharp_{\kappa(\nu)}} (\widetilde{F})\right|_\nu^p \\ 
&(\mbox{from} \; (\ref{sec_not}.\ref{ti0});\;\mbox{recall that} \; \left|P\right|_\nu = \left|P\right|_{x_\nu,\delta_{Q_\nu}} \;\mbox{and}  \; |x_\nu - x| \leq C \delta_{Q_\nu}\; \mbox{for any}\; x \in Q_\nu)\\
&\qquad\qquad\;\lesssim \|\widetilde{F}\|^p_{L^{m,p}(Q^\circ)}\\
&(\mbox{from Lemma} \; \ref{globsob}).
\end{align*}

By Lemma \ref{patch2}, there exists a function $\widehat{F} \in \LR$ with $\widehat{F} = \overline{F}$ on $0.99Q^\circ$ and with $\|\widehat{F}\|_{\LR} \lesssim \|\overline{F}\|_{L^{m,p}(Q^\circ)}$. Thus, we have
$$\|\widehat{F}\|_{\LR} \lesssim \|\overline{F}\|_{L^{m,p}(Q^\circ)} \lesssim \|\widetilde{F}\|_{L^{m,p}(Q^\circ)} \olesssim{\eqref{eq205}} \|F\|_{\LR} \oleq{\eqref{p503}} 2 \|(f,0)\|.$$ 
It follows from $E \cup E' \subset 0.99 Q^\circ$ and (a)-(c) that $\widehat{F}$ extends the data $(f,0)$, $\widehat{F}$ is Coherent with $0 \in \cP$ and $\widehat{F}$ is Constant Along Paths. This concludes the proof of Proposition \ref{special}.
\end{proof}

\section{A Constrained Problem} \label{sec_mod}
\setcounter{equation}{0}

For each $\nu = 1,\ldots,\nu_{\max}$, we define $E_\nu = E \cap (1.1)Q_\nu$. For a given function $f: E \rightarrow \R$, we denote $f_\nu = f|_{E_\nu}$ here and throughout.

\begin{prop}
There exist a linear operator $T: L^{m,p}(E;E') \rightarrow \LR$ and a collection of linear functionals $\Omega' \subset [\LE]^*$, such that the following hold.
\begin{itemize}
\item $T$ is an extension operator.
\item $T$ is bounded.
\item $\sum \bigl\{ \sp(\omega') : {\omega' \in \Omega'} \bigr\}  \leq C \cdot \#(E)$.
\item $T$ has $\Omega'$-assisted bounded depth.
\end{itemize}
Moreover,
\begin{equation} \label{modnorm}
\|(f,\vP)\|^p \simeq \sum_{\nu=1}^{\nu_{\max}} \|(f_\nu,P_\nu)\|^p + \sum_{\nu \leftrightarrow \nu'} \left\lvert P_{\nu} - P_{\nu'} \right\rvert_\nu^p.
\end{equation}
Above, $\|(f_\nu,P_\nu)\|$ denotes the trace seminorm $\|(f_\nu,P_\nu)\|_{L^{m,p}(E_\nu;x_\nu)}$ and $\|(f,\vP)\|$ denotes the trace seminorm $\|(f,\vP)\|_{L^{m,p}(E;E')}$.
\label{modprop}
\end{prop}
\begin{proof}
For each $\nu = 1,\ldots,\nu_{\max},$ we apply \textbf{(L1-6)} from section \ref{sec_loc} to the subset $\oE_\nu = E_\nu$. Thus there exist linear functionals $\Omega_\nu \subset [L^{m,p}(E_\nu)]^*$ and a linear map $T_\nu: L^{m,p}(E_\nu; x_\nu) \rightarrow \LR$ (with $\Omega_\nu$-assisted bounded depth), such that
\begin{align}
\label{ep00} &T_\nu(f_\nu,P_\nu) = f_\nu  \; \mbox{on} \; E_\nu \;\; \mbox{and} \;  J_{x_\nu}\left( T_\nu(f_\nu,P_\nu) \right) = P_\nu; \; \mbox{and} \\
 \label{ep01} &\|T_\nu(f_\nu,P_\nu)\|_{L^{m,p}(1.1 Q_\nu)} \leq \|T_\nu(f_\nu,P_\nu)\|_{\LR} \simeq \|(f_\nu,P_\nu)\|.
\end{align}
Using the partition of unity $\{\theta_\nu\}$ that satisfies \textbf{(POU1-5)}, we define
$$F^\circ = T^\circ(f,\vP) = \sum_{\nu=1}^{\nu_{\max}} T_\nu(f_\nu,P_\nu) \theta_\nu.$$
Of course, $T^\circ : L^{m,p}(E;E') \rightarrow L^{m,p}(Q^\circ)$ is a linear map.  From \textbf{(POU2)} and \eqref{ep00}, we have that $T_\nu(f_\nu,P_\nu) = f$ on $\supp(\theta_\nu) \cap E$. Hence, by \textbf{(POU5)}, we have $F^\circ = f$ on $E$. From \textbf{(POU4)} and \eqref{ep00}, we also have $J_{E'} F^\circ =\vP$. To summarize, we have
\begin{equation}
F^\circ = f \; \mbox{on} \; E \; \mbox{and} \; J_{E'} F^\circ = \vP. \label{ep08}
\end{equation}
From \eqref{ep00}, \eqref{ep01} and a straightforward application of Lemma \ref{whittrick}, we have
\begin{equation}
\|F^\circ\|^p_{L^{m,p}(Q^\circ)} \lesssim \sum_{\nu=1}^{\nu_{\max}}\|(f_\nu,P_\nu)\|^p + \sum_{\nu \leftrightarrow \nu'} \left|P_\nu-P_{\nu'}\right|_\nu^p. \label{ep09}
\end{equation}

We identify $[L^{m,p}(E_\nu)]^*$ with a subspace of $[L^{m,p}(E)]^*$ through the natural restriction from $L^{m,p}(E)$ to $L^{m,p}(E_\nu)$, and define
$$\Omega' = \bigcup_{\nu=1}^{\nu_{\max}} \Omega_\nu \subset [L^{m,p}(E)]^*.$$
Since $\Omega_\nu$ satisfies \textbf{(L4)} (for the subset $E_\nu = E \cap (1.1)Q_\nu$), we obtain
\begin{equation}
\sum_{\omega \in \Omega'} \sp(\omega) \leq \sum_{\nu=1}^{\nu_{\max}} \sum_{\omega \in \Omega_\nu} \sp(\omega) \leq \sum_{\nu=1}^{\nu_{\max}} C \cdot \#(E \cap (1.1)Q_\nu),
\end{equation}
which is bounded by $C' \cdot \#(E)$ due to the first bullet point in Lemma \ref{moregg}. This proves the third bullet point of Proposition \ref{modprop}.

Using \textbf{(POU2)}, for each fixed $x \in Q^\circ$ we obtain a list of cubes $Q_{\nu_1},\ldots,Q_{\nu_L} \in \CZ$ and linear maps $\psi_{1,x},\ldots, \psi_{L,x} : \cP \rightarrow \cP$, such that
\begin{equation} \label{jets1}
J_x\bigl[T^\circ(f,\vP)\bigr] = \psi_{1,x}\left(J_x\bigl [T_{\nu_1}(f_{\nu_1},P_{\nu_1}) \bigr] \right) + \cdots + \psi_{L,x}\left( J_x \bigl[T_{\nu_L}(f_{\nu_L},P_{\nu_L}) \bigr] \right).
\end{equation}

The list $Q_{\nu_1},\ldots,Q_{\nu_L}$ is composed of the CZ cubes for which $x \in (1.1)Q_{\nu_i}$ (thus, $\theta_{\nu_1},\ldots,\theta_{\nu_L}$ include those cutoff functions which do not vanish in a neighborhood of $x$), and we define $\psi_{i,x}(P) = J_x(\theta_{\nu_i} P)$ ($i = 1,\ldots,L$). The first bullet point in Lemma \ref{moregg} shows that $L$ is controlled by a universal constant. Thus, by \eqref{jets1} and the fact that $T_{\nu_i}$ has $\Omega_{\nu_i}$-assisted bounded depth, the linear map $J_x T^\circ(\cdot,\cdot) : L^{m,p}(E;E') \rightarrow \cP$ has $\Omega'$-assisted bounded depth. Therefore, $T^\circ$ has $\Omega'$-assisted bounded depth.

Next, we apply Lemma \ref{patch2} to the function $F^\circ = T^\circ(f,\vP) \in L^{m,p}(Q^\circ)$. This gives a linear map $T:L^{m,p}(E;E') \rightarrow \LR$ with $\Omega'$-assisted bounded depth, for which $F = T(f,\vP)$ satisfies $F = F^\circ$ on ${0.99 Q^\circ}$ and
\begin{equation}
\label{ep11}\|F\|^p_{\LR} \lesssim \|F^\circ\|^p_{L^{m,p}(Q^\circ)} \olesssim{\eqref{ep09}} \sum_{\nu=1}^{\nu_{\max}} \|(f_\nu,P_\nu)\|^p + \sum_{\nu \leftrightarrow \nu'} \left\lvert P_\nu-P_{\nu'} \right\rvert_\nu^p.
\end{equation}
From $E \cup E' \subset 0.99 Q^\circ$ and \eqref{ep08}, we have
\begin{equation}
\label{ep10} F = f \; \mbox{on} \; E \; \mbox{and} \; J_{E'} (F) = \vP.
\end{equation}
Thus the first and fourth bullet points in Proposition \ref{modprop} hold.

Finally, we must establish \eqref{modnorm} and prove that $\|F\|_{\LR}$ lies within a universal constant factor of the trace seminorm $\|(f,\vP)\|_{L^{m,p}(E;E')}$. Let $H \in \LR$ be an arbitrary function that satisfies 
$$\circledast \; H = f  \; \mbox{on} \;  E \;\; \mbox{and} \; \;J_{E'} H = \vP.$$ 
From the Sobolev inequality (\ref{sec_not}.\ref{SET})  and $J_{E'} H = \vP$, we have $\left\lvert P_\nu - P_{\nu'}\right\rvert_\nu^p \lesssim \|H\|_{L^{m,p}(1.1Q_\nu \cup 1.1Q_{\nu'})}^p$. Hence, by the Good Geometry of the CZ cubes, we have
\begin{equation}\label{1done}
\sum_{\nu \leftrightarrow \nu'} \left|P_\nu - P_{\nu'}\right|_\nu^p \lesssim \sum_{\nu \leftrightarrow \nu'} \|H\|_{L^{m,p}(1.1Q_\nu \cup 1.1Q_{\nu'})}^p \lesssim \|H\|_{\LR}^p.
\end{equation}
Note that $E_\nu= (1.1)Q_\nu \cap E \subset (0.9)(1.3) Q_\nu$ and $x_\nu \in Q_\nu \subset (0.9)(1.3) Q_\nu$. Applying Lemma \ref{patch2} with $Q=(1.3)Q_\nu$ and $a=0.1$ we deduce that $\|(f_\nu,P_\nu)\| \lesssim \|H\|_{L^{m,p}(1.3Q_\nu)}$. Hence, by Good Geometry of the CZ cubes, we have
\begin{equation} \label{2done}
\sum_{\nu=1}^{\nu_{\max}} \|(f_\nu,P_\nu)\|^p \lesssim \sum_{\nu=1}^{\nu_{\max}} \|H\|_{L^{m,p}(1.3Q_\nu)}^p \lesssim \|H\|_{\LR}^p.
\end{equation}
Adding together \eqref{1done} and \eqref{2done} shows that the right-hand side of \eqref{ep11} is bounded by $C\|H\|_{\LR}^p$. Taking the infimum with respect to those $H$ satisfying $\circledast$, we obtain
$$\sum_{\nu=1}^{\nu_{\max}}\|(f_\nu,P_\nu)\|^p + \sum_{\nu \leftrightarrow \nu'} \left\lvert P_\nu-P_{\nu'}\right\rvert_\nu^p \lesssim \|(f,\vP)\|^p.$$ 
The reverse inequality holds as well, thanks to \eqref{ep11} and \eqref{ep10}. Thus, \eqref{modnorm} holds. Together, \eqref{modnorm} and \eqref{ep11} imply that $\|F\|_{\LR} \leq C' \|(f,\vP)\|_{L^{m,p}(E;E')}$. Thus the second bullet point in Proposition \ref{modprop} holds. This concludes the proof of Proposition \ref{modprop}.
\end{proof}

\section{The Jet of a Near Optimal Extension II} \label{sec_jets}
\setcounter{equation}{0}
Recall our notation and assumptions on indexing in the text surrounding (\ref{sec_rep}.\ref{index_ass}). For each $\mu=\mu_{\min},\ldots, \mu_{\max}$, Lemma \ref{jetop} provides a jet $R^\sharp_\mu$ that depends linearly on $(f|_{E \cap 9Q^\sharp_\mu}, P)$ and satisfies
\begin{equation}\label{jop1} 
\partial^\alpha R_\mu^\sharp(x_{\mu}^\sharp) = \partial^\alpha P(x_{\mu}^\sharp) \; \mbox{for every} \; \alpha \in \cA
\end{equation}
and
\begin{align} \label{jop2}
&\|(f_\mu^\sharp,R^\sharp_\mu)\|_{L^{m,p}(E \cap 9 Q_\mu^\sharp; x_{\mu}^\sharp)} \leq \\
&\qquad C  \inf \left\{ \|(f_\mu^\sharp,R')\|_{L^{m,p}(E \cap 9 Q^\sharp_\mu;x_{\mu}^\sharp)}: R' \in \cP, \; \partial^\alpha R'(x_{\mu}^\sharp) = \partial^\alpha P(x_{\mu}^\sharp) \; \mbox{for every} \; \alpha \in \cA \right\}.\notag{}
\end{align}
In the case that $Q_1$ is keystone, we also define $R_1^\sharp = P$. Finally, we define $R_{1} = P$ and $R_\nu = R_{\kappa(\nu)}^\sharp$ for each $\nu=2,\ldots,\nu_{\max}$. 

It is immediate that $\vR = (R_\nu)_{\nu=1} ^{\nu_{\max}}\in \Wh(E')$ is Constant Along Paths, Coherent with $P$ and depends linearly on $(f,P)$. The main result of this section states that $\vR$ is the jet on $E'$ of a near optimal extension of $(f,P)$.
\begin{lem}
Given $(f,P) \in L^{m,p}(E;z)$, let $\vR \in Wh(E')$ be defined as above. Then,
$$\|(f,\vR)\|_{L^{m,p}(E;E')} \lesssim \|(f,P)\|_{L^{m,p}(E;z)}.$$
\label{optlem}
\end{lem}
\begin{proof}
By Proposition \ref{special}, it suffices to establish the bound 
\begin{equation} \label{mbound} \|(f,\vR)\|_{L^{m,p}(E;E')} \lesssim \|(f,\vP)\|_{L^{m,p}(E;E')}\end{equation}
for each $\vP = (P_\nu)_{\nu=1}^{\nu_{\max}} \in \Wh(E')$ that is
\begin{align*}
&\mbox{\textbf{Constant Along Paths}:} \; P_\nu = P_{\kappa(\nu)}^\sharp \; \mbox{for each} \; \nu=2,\ldots,\nu_{\max}, \; \mbox{and} \\
&\mbox{\textbf{Coherent with}} \; \mathbf{P:}\; P_{1}=P \; \mbox{and} \; \partial^\alpha P_{\mu}^\sharp(x^\sharp_\mu) = \partial^\alpha P(x^\sharp_\mu) \; \mbox{for each} \; \mu=1,\ldots,\mu_{\max} , \; \alpha \in \cA.
\end{align*}
Let $\vP \in \Wh(E')$ be fixed as above. Then, we have
\begin{align}
\|(f,\vR)\|^p &\simeq \sum_{\nu=1}^{\nu_{\max}} \|(f_\nu,R_\nu)\|^p + \sum_{\nu \leftrightarrow \nu'} \left|R_\nu - R_{\nu'}\right|_\nu^p \notag{} \\
& \qquad (\mbox{from (\ref{sec_mod}.\ref{modnorm}) in Proposition \ref{modprop}})\notag{}  \\
& \lesssim \sum_{\nu=1}^{\nu_{\max}} \bigl[ \|(f_\nu,P_\nu)\|^p + \|(0|_{E_\nu},P_\nu-R_\nu)\|^p \bigr]\notag{}  \\
& \quad + \sum_{\nu \leftrightarrow \nu'} \bigl[ \left|R_\nu - P_\nu\right|_\nu^p + \left|P_\nu - P_{{\nu'}}\right|^p_\nu + \left|P_{{\nu'}} - R_{{\nu'}}\right|^p_\nu \bigr] \notag{} \\
& \qquad (\mbox{from subadditivity of the seminorms}) \notag{} \\
& \lesssim \sum_{\nu=1}^{\nu_{\max}} \bigl[\|(f_\nu,P_\nu)\|^p + \|(0|_{E_\nu},P_\nu-R_\nu)\|^p + \left|R_\nu - P_\nu \right|_\nu^p  \bigr] + \sum_{\nu \leftrightarrow {\nu'}} \left|P_\nu - P_{{\nu'}} \right|^p_\nu \notag{}  \\
& \qquad (\mbox{from (\ref{sec_rep}.\ref{eq_1111}) and the Good Geometry of the CZ cubes})\notag{} \\
& \simeq \|(f,\vP)\|^p + \sum_{\nu=1}^{\nu_{\max}} \bigl[ \|(0|_{E_\nu},R_\nu - P_\nu)\|^p + \left|R_\nu - P_\nu\right|_\nu^p \bigr] \notag{} \\
& \qquad (\mbox{from (\ref{sec_mod}.\ref{modnorm})})\notag{} \\
& \lesssim \|(f,\vP)\|^p + \sum_{\nu=2}^{\nu_{\max}} \left|R_\nu - P_\nu \right|_\nu^p \label{eq300}\\
& \qquad (\mbox{from} \; P_1 = P = R_{1} \; \mbox{and Lemma \ref{esslem}}). \notag{}
\end{align}
Define $X$ as the second term on the right-hand side of \eqref{eq300}. From the fact that $\vP$ and $\vR$ are Constant Along Paths and (\ref{sec_rep}.\ref{ti}), we have
\begin{align}
X &= \sum_{\nu=2}^{\nu_{\max}} \left|R_\nu - P_\nu\right|_\nu^p =  \sum_{\nu=2}^{\nu_{\max}} \left|R^\sharp_{\kappa(\nu)} - P^\sharp_{\kappa(\nu)}\right|_\nu^p \notag{} \\
&\lesssim \sum_{\mu=1}^{\mu_{\max}} \sum_{ \kappa(\nu)=\mu} \sum_{|\beta| \leq m-1} \left|\partial^\beta (R_\mu^\sharp - P_{\mu}^{\sharp})(x_{\mu}^{\sharp})\right|^p \delta_{Q_\nu}^{n+(|\beta|-m)p}. \label{xbd}
\end{align}
Next, we estimate the dyadic sum of lengthscales arising above.

If $Q^\sharp_\mu$ is the keystone cube arising from $Q_\nu$, then $\delta_{Q_\nu} \geq c \delta_{Q^\sharp_\mu}$ and $Q_\mu^\sharp \subset C Q_\nu$ for a large enough universal constant $C$ and a small enough universal constant $c$. Hence, for each $\mu$ and $\tau > 0$, we have
\begin{align*}
\sum_{\kappa(\nu) = \mu} \left[ \delta_{Q_\nu}\right]^{-\tau} &\leq C'(\tau,m,n,p) \cdot \sum \left\{ \left[\delta_Q\right]^{-\tau}: Q\subset \R^n \; \mbox{dyadic with} \; \left[\delta_Q \geq c \delta_{Q^\sharp_\mu} \; \mbox{and} \; Q_\mu^\sharp \subset C Q\right] \right\} \\
&\leq C''(\tau,m,n,p) \left[\delta_{Q_\mu^\sharp}\right]^{-\tau}.
\end{align*}

Plugging this inequality into \eqref{xbd}, we have
\begin{align}\label{eq_4419}
X \lesssim \sum_{\mu=1}^{\mu_{\max}}  \sum_{|\beta| \leq m-1} \left|\partial^\beta (R_\mu^\sharp - P_{\mu}^{\sharp})(x_{\mu}^{\sharp}) \right|^p \left[\delta_{Q^\sharp_\mu}\right]^{n+(|\beta|-m)p} &= \sum_{\mu=1}^{\mu_{\max}} \left|R_\mu^\sharp - P_{\mu}^{\sharp} \right|^p_{x^\sharp_\mu,\delta_{Q_\mu^\sharp}} \notag{}\\
& = \sum_{\mu=\mu_{\min}}^{\mu_{\max}} \left|R_\mu^\sharp - P_{\mu}^{\sharp} \right|^p_{x^\sharp_\mu,\delta_{Q_\mu^\sharp}},
\end{align}
since when $Q_1$ is keystone we have $P_1^\sharp = P_1 = P = R_1 = R^\sharp_1$ by Coherence with $P$. (Recall the indexing assumption (\ref{sec_rep}.\ref{index_ass}).) 

Note that $\partial^\alpha (R_\mu^{\sharp})(x^\sharp_\mu) = \partial^\alpha P(x^\sharp_\mu) = \partial^\alpha(P_\mu^{\sharp})(x_\mu^\sharp)$ for each $\mu=\mu_{\min},\ldots,\mu_{\max}$ and $\alpha \in \cA$. (This follows from Coherence with $P$.) Next, from Lemma \ref{esslem} and \eqref{eq_4419}, we obtain 
\begin{align}
X &\lesssim \sum_{\mu=\mu_{\min}}^{\mu_{\max}} \|(0|_{E \cap 9Q_\mu^\sharp},R_\mu^{\sharp} - P_{\mu}^{\sharp})\|^p \lesssim \sum_{\mu=\mu_{\min}}^{\mu_{\max}} \left[ \|(f|_{E \cap 9Q_\mu^\sharp},R_\mu^{\sharp})\|^p + \|(f|_{E \cap 9Q_\mu^\sharp},P_\mu^{\sharp})\|^p \right] \notag{} \\
& \lesssim \sum_{\mu=\mu_{\min}}^{\mu_{\max}} \|(f|_{E \cap 9 Q_\mu^\sharp},P_\mu^{\sharp})\|^p, \;\; \mbox{thanks to \eqref{jop2}}. \label{eq_4420}
\end{align}
(The seminorms above are taken in the space $L^{m,p}(E \cap 9 Q_\mu^\sharp, x^\sharp_\mu)$.)

Let $F \in \LR$ be arbitrary with $F= f$ on $E$ and $J_{E'} (F) = \vP$. Applying Lemma \ref{patch2} with $Q = 10 Q_\mu^\sharp$ and $a = 0.1$, we deduce that $\|(f|_{E \cap 9 Q_\mu^\sharp},P_\mu^{\sharp})\| \lesssim \|F\|_{L^{m,p}(10 Q^\sharp_\mu)}$. Thus, \eqref{eq_4420} yields
$$X \lesssim \sum_{\mu=1}^{\mu_{\max}} \|F\|_{L^{m,p}(10 Q^\sharp_\mu)}^p,$$
which is bounded by $C \cdot \|F\|_{\LR}^p$, thanks to Lemma \ref{moregeom}. Taking the infimum with respect to $F \in \LR$ as above, we obtain $X \lesssim \|(f,\vP)\|^p$. Thus \eqref{mbound} holds (see \eqref{eq300} and the definition of $X$ that follows), which completes the proof of Lemma \ref{optlem}.
\end{proof}

\section{The Extension Operator}
\setcounter{equation}{0}

\label{plstow}

In this section we complete the proof of the Extension Theorem for $(E;z)$. We now fix a sufficiently small universal constant $\epsilon(\cA) = \epsilon$, so that the results from previous sections hold.

Proposition \ref{modprop} provides a bounded linear extension operator $\widehat{T} : L^{m,p}(E;E') \rightarrow \LR$ and assists $\widehat{\Omega} \subset [L^{m,p}(E)]^*$, so that $\widehat{T}$ has $\widehat{\Omega}$-assisted bounded depth, and such that
\begin{equation} \label{sumsp1} \sum_{\omega \in \widehat{\Omega}} \sp(\omega) \leq C \cdot \#(E).
\end{equation}

Let $\vR \in Wh(E')$ be as defined in the previous section. In particular, $\vR$ depends linearly on $(f,P)$, $\vR$ is Constant Along Paths and $\vR$ is Coherent with $P$.

We define
$$T(f,P) = \widehat{T}(f,\vR).$$
Since $\vR$ depends linearly on $(f,P)$ and the map $\widehat{T}$ is linear, the map $T$ is clearly linear. Also, since $\widehat{T}$ is an extension operator and $R_1 = P$ (by Coherence of $\vR$ with $P$), we have 
$$T(f,P) = \widehat{T}(f,\vR) = f  \; \mbox{on} \;E\;\; \mbox{and} \;\;J_z\bigl( T(f,P) \bigr) = J_{x_1} \bigl(\; \widehat{T}(f,\vR) \bigr) = R_1 = P.$$ 
Thus, $T$ is an extension operator and we have established \textbf{(E1)} in the statement of the Extension Theorem for $(E,z)$.

We now prove \textbf{(E2)}. The first inequality in \textbf{(E2)} follows trivially from \textbf{(E1)} and the definition of the trace seminorm. Thanks to Lemma \ref{optlem} and the fact that $\widehat{T} : L^{m,p}(E;E') \rightarrow \LR$ is bounded, we have 
\begin{equation} \label{mmm} \|T(f,P)\|_{\LR} = \|\widehat{T}(f,\vR)\|_{\LR} \lesssim \|(f,\vR)\|_{L^{m,p}(E;E')} \lesssim \|(f,P)\|_{L^{m,p}(E;z)}.
\end{equation}
Since also $\|(f,P)\|_{L^{m,p}(E;z)} \leq \|T(f,P)\|_{\LR}$, we may replace every $\lesssim$ with $\simeq$ in the above. This proves that $T$ is bounded, and finishes the proof of \textbf{(E2)}.

Next we estimate the seminorm $\|T(f,P)\|_{\LR}$ and define $M(f,P)$ in order to establish conclusion \textbf{(E3)}. Let us define
\begin{align*}
&\cI := \bigl\{(\nu,\nu') \in \{2,\ldots,\nu_{\max}\}^2 : \; \nu \leftrightarrow \nu' \; \mbox{and}\; \kappa(\nu) \neq \kappa(\nu') \bigr\}; \; \mbox{and}\\
&\cI^\sharp := \bigl\{ (\kappa(\nu),\kappa(\nu')) : (\nu,\nu') \in \cI \bigr\}.
\end{align*}
From \textbf{(K2)} in Proposition \ref{keygeom}, we learn that
\begin{equation} \#(\cI^\sharp) \leq \#(\cI) \leq C \cdot \#(E). \label{k3} 
\end{equation}
Recall that $E_\nu= E \cap (1.1)Q_\nu$ and $f_\nu = f|_{E_\nu}$. It follows from the comment following \eqref{mmm} and from Proposition \ref{modprop} that
\begin{align}\|T(f,P)\|^p_{\LR} & \simeq\; \|(f,\vR)\|^p \simeq \sum_{\nu=1}^{\nu_{\max}} \|(f_\nu,R_\nu)\|^p + \sum_{\nu \leftrightarrow \nu'} \left\lvert R_\nu - R_{\nu'}\right\rvert_\nu^p \notag{}\\
&= \sum_{\nu=1}^{\nu_{\max}} \|(f_\nu,R_\nu)\|^p + \sum_{(\nu,\nu') \in \cI} \bigl|R^\sharp_{\kappa(\nu)} - R^\sharp_{\kappa(\nu')} \bigr|_\nu^p + \sum_{\substack{2 \leq \nu \leq \nu_{\max} \\ \nu \leftrightarrow 1}} \bigl|R_{\kappa(\nu)}^\sharp - R_1\bigr|_\nu^p\label{almost} \\
& \qquad (\mbox{since} \; \vR \; \mbox{is Constant Along Paths}). \notag{}
\end{align}
(Here, $\|(f_\nu,R_\nu)\|$ denotes the trace seminorm on $L^{m,p}(E_\nu;x_\nu)$.) From (\ref{sec_rep}.\ref{ti}), we deduce that
\begin{align*}
\sum_{(\nu,\nu') \in \cI} \bigl\lvert R_{\kappa(\nu)}^\sharp - R_{\kappa(\nu')}^\sharp \bigr\rvert_\nu^p 
&\simeq \sum_{(\nu,\nu') \in \cI} \sum_{\alpha \in \cM} 
\bigl\lvert \partial^\alpha \bigl( R_{\kappa(\nu)}^\sharp - R_{\kappa(\nu')}^\sharp \bigr)(x_{\kappa(\nu)}^\sharp) \bigr\rvert^p \delta_{Q_\nu}^{n+(|\alpha|-m)p} \\ 
&= \sum_{(\mu, \mu') \in \cI^\sharp} \sum_{\alpha \in \cM} \bigl\lvert \partial^\alpha \bigl(R_\mu^\sharp - R_{\mu'}^\sharp \bigr)(x_\mu^\sharp) \bigr\rvert^p \Delta_\alpha(\mu,\mu')^p,\notag{}
\end{align*}
$$\mbox{where} \;\; \Delta_\alpha(\mu,\mu')^p := \sum \left\{ \delta_{Q_\nu}^{n+(|\alpha|-m)p} : (\nu,\nu') \in \cI,\; \kappa(\nu)= \mu ,\; \kappa(\nu')=\mu' \right\}.$$
Plugging into \eqref{almost} this formula, the definition for the norm $\left|\cdot\right|_\nu$ and the formula $M_\nu(f_\nu,P_\nu) = \left( \sum_{\xi \in \Xi_\nu} \lvert \xi(f_\nu,P_\nu) \rvert^p \right)^{1/p}$ that approximates the trace semi-norm $\|(f_\nu,P_\nu)\|$ (see \textbf{(L1-6)} in section \ref{sec_loc}), we have
\begin{equation}
\label{greatstuff}
\begin{aligned} \|T(f,P)\|_{\LR}^p &\simeq \sum_{\nu=1}^{\nu_{\max}} \sum_{\xi \in \Xi_\nu} \lvert \xi(f_\nu,R_\nu) \rvert^p + \sum_{(\mu,\mu') \in \cI^\sharp} \sum_{\alpha \in \cM} \bigl\lvert \partial^\alpha (R_\mu^{\sharp} - R_{\mu'}^{\sharp})(x_{\mu}^{\sharp}) \bigr\rvert^p \Delta_\alpha(\mu,\mu')^p \\
& + \sum_{\substack{2 \leq \nu \leq \nu_{\max} \\ \nu \leftrightarrow 1}} \sum_{\alpha \in \cM}  \bigl\lvert\partial^\alpha(R_{\kappa(\nu)}^\sharp - R_1)(x_\nu)\bigr\rvert^p \delta_{Q_\nu}^{n - (m-|\alpha|)p}.
\end{aligned}
\end{equation}
We record here the fact (also following from \textbf{(L1-6)}), that there exist assists $\Omega_\nu \subset [L^{m,p}(E_\nu)]^*$, such that
\begin{equation}
\label{nicestuff1} 
\begin{aligned}
&\mbox{each functional in}\; \Xi_\nu \; \mbox{has} \; \Omega_\nu\mbox{-assisted bounded depth}, \\
&\#(\Xi_\nu) \leq C \cdot \#(E_\nu) \;\; \mbox{and} \;\; \sum_{\omega \in \Omega_\nu}  \; \sp(\omega) \leq C \cdot \#(E_\nu).
\end{aligned}
\end{equation}

We now simply take
\begin{equation}M(f,P)^p := \; \mbox{the right-hand side of \eqref{greatstuff}}. \label{defnM}
\end{equation}
Thus, $\|T(f,P)\| \simeq M(f,P)$, which establishes \textbf{(E3)}.

\subsection{Assisted Bounded Depth}
First, we define the collection of linear functionals $\Omega \subset [L^{m,p}(E)]^*$ that satisfies \textbf{(E4)}, \textbf{(E5)} and \textbf{(E6)}. Recall that $R^\sharp_\mu \in \cP$ depends linearly on $(f|_{E \cap 9 Q^\sharp_\mu}, P)$ for each $\mu=1,\ldots,\mu_{\max}$. Since $\vR$ is Coherent with $P$ and $\vR$ is Constant Along Paths, we have
\begin{equation} R_1 = P, \; \mbox{and} \; R_\nu = R^\sharp_{\kappa(\nu)} \; \mbox{for each} \; \nu=2,\ldots,\nu_{\max}. \label{defnR2}
\end{equation}
In this subsection, $R^\sharp_\mu$ and $R_\nu$ should be considered as $\cP$-valued linear maps on $L^{m,p}(E, z)$. 
Consider the decomposition
\begin{equation} \label{defnR3} 
R^\sharp_\mu = \widehat{R}^\sharp_\mu + \widetilde{R}^\sharp_\mu, \; \mbox{where} \; \widehat{R}^\sharp_\mu \in \cP \; \mbox{depends only on} \; f|_{E \cap 9Q^\sharp_\mu} \; \mbox{and} \; \widetilde{R}^\sharp_\mu \in \cP \; \mbox{depends only on} \; P.
\end{equation}
First we define linear functionals $\omega_{\beta, \mu} : L^{m,p}(E) \rightarrow \R$, for $\mu=1,\ldots,\mu_{\max}$ and $\beta \in \cM$, by
$$\omega_{\beta, \mu}(f) = \partial^\beta \left[\widehat{R}^\sharp_\mu\right](x_{\mu}^{\sharp}).$$
Therefore, $\sp(\omega_{\beta,\mu}) \leq \#(E \cap 9Q_\mu^\sharp)$. We define
$$\Omega_{\new} := \bigl\{\omega_{\beta, \mu} : \mu=1,\ldots,\mu_{\max}, \; \beta \in \cM \bigr\} \subset [L^{m,p}(E)]^*.$$

Recall the definition of $\widehat{\Omega}$ and $\Omega_\nu$ in the text surrounding \eqref{sumsp1} and \eqref{nicestuff1}, respectively. Note that $\Omega_\nu \subset [L^{m,p}(E_\nu)]^*$ is included in $[L^{m,p}(E)]^*$ through the natural restriction from $L^{m,p}(E)$ onto $L^{m,p}(E_\nu)$. Let us define
\begin{equation} \label{dds}
\Omega :=  \widehat{\Omega} \cup \left( \bigcup_{\nu=1}^{\nu_{\max}} \Omega_\nu \right) \cup \Omega_{\new}.
\end{equation}
Thus, from \eqref{sumsp1} and \eqref{nicestuff1}, we have
\begin{align*}
\sum_{\omega \in \Omega} \sp(\omega) &\leq \sum_{\nu}  \sum_{\omega \in \Omega_\nu} \sp(\omega) \;\;\; + \sum_{\omega \in \widehat{\Omega}} \sp(\omega) + \sum_{\beta, \; \mu} \sp(\omega_{\beta, \mu})  \\
& \lesssim \sum_{\nu} \#(E \cap 1.1 Q_\nu) + \quad \#(E) \quad\; + \sum_{\mu} \#(E \cap 9 Q_\mu^\sharp),
\end{align*}
which is bounded by $C \cdot \#(E)$ thanks to Lemma \ref{moregg}  and Lemma \ref{moregeom}. This proves \textbf{(E4)}.
 
Since $\widehat{T}$ has $\widehat{\Omega}$-assisted bounded depth, there exists a universal constant $r \in \N$ such that, for each point $x \in \R^n$ there exist assists $\omega_1,\ldots,\omega_r \in \widehat{\Omega}$, polynomials $\hP^1,\ldots,\hP^r \in \cP$ and a linear map $\widetilde{\lambda} : L^{m,p}(E;E') \rightarrow \cP$ with $\sp(\widetilde{\lambda}) \leq r$, such that
\begin{equation} \label{ghh1} J_x \bigl[ T(f,P) \bigr] = J_x\left[\widehat{T}(f,\vR)\right] = \omega_1(f) \hP^1 + \cdots + \omega_r(f) \hP^r + \widetilde{\lambda}(f,\vR).\end{equation}
Thanks to \eqref{defnR2}, \eqref{defnR3}, and a change of basis, we obtain
$$J_x\bigl[T(f,P)\bigr] = \omega_1(f) \hP^1 + \cdots + \omega_r(f) \hP^r + \overline{\lambda} \left(f  ,  P,    
  \left( \widehat{R}_\mu^\sharp  + \widetilde{R}_\mu^\sharp \right)_{\mu=1}^{\mu_{\max}}\right),$$
where $\sp(\overline{\lambda}) \leq C \cdot \sp(\widetilde{\lambda}) \leq C \cdot r$. Thus, by placing aside those terms from $\overline{\lambda}$ that depend on $\widehat{R}^\sharp_\mu$ (which depends solely on $f$), we have
$$J_x\bigl[T(f,P)\bigr] = \omega_1(f) \hP^1 + \cdots + \omega_r(f) \hP^r + \omega_{r+1}(f)\hP^{r+1} + \cdots + \omega_{r+s}(f)\hP^{r+s} + \widehat{\lambda}(f,P),$$
where $\omega_{r+1},\ldots,\omega_{r+s} \in \Omega_{\new}$, $\hP^{r+1},\ldots,\hP^{r+s} \in \cP$ are independent of $(f,P)$, and $s \leq C \cdot r$, $\sp(\widehat{\lambda}) \leq \sp(\overline{\lambda}) \leq C \cdot r$.
Thus, $T$ has $\Omega$-assisted bounded depth, which establishes \textbf{(E5)}.

Finally, we define $\Xi \subset [L^{m,p}(E;z)]^*$ as
\begin{align*}
\Xi = \bigl\{\mbox{linear functionals arising on the right-hand side of} \; \eqref{greatstuff} \mbox{, counted with repetition}\bigr\}. \notag{}
\end{align*}
By ``counted with repetition,'' we mean that if some linear functional $\xi : L^{m,p}(E;z) \rightarrow \R$ appears $t$ times on the right-hand side of \eqref{greatstuff}, then we include $t^{1/p} \xi$ in $\Xi$.

From the definition of $M$ in \eqref{defnM}, we find that
$$M(f,P) = \left( \sum_{\xi \in \Xi} |\xi(f,P)|^p \right)^{1/p}.$$
Thus, \textbf{(E6c)} holds.

From \eqref{k3},\eqref{nicestuff1} and the Good Geometry of the CZ cubes, we have
$$\#(\Xi) \leq \sum_{\nu=1}^{\nu_{\max}} \#(\Xi_\nu) + C \cdot \#(\cI^\sharp) + C \cdot \#\{1 \leq \nu \leq \nu_{\max}: \nu \leftrightarrow 1\} \leq C' \sum_{\nu=1}^{\nu_{\max}} \#(E_\nu) + C' \cdot \#(E),$$
which is controlled by $C'' \cdot \#(E)$, thanks to Lemma \ref{moregg}. Thus, \textbf{(E6b)} holds.

One verifies that each functional $\xi \in \Xi$ (appearing on the right-hand side of \eqref{greatstuff}) has $\Omega$-assisted bounded depth using \eqref{nicestuff1}, \eqref{defnR3} and \eqref{dds}. Thus, \textbf{(E6a)} holds.

We have completed the induction step indicated in section \ref{sec_plan}, thus completing the proof of the Main Lemma for all labels $\cA$. Taking $\cA = \emptyset$, we obtain the Extension Theorem for $(E,z)$ whenever $E \subset \R^n$ is finite.

\section{Proofs of the Main Theorems for Finite $E$}\label{sec_pf}
\setcounter{equation}{0}

Let $E \subset \R^n$ be a finite subset with cardinality $\#(E) = N$. Pick $z \in \R^n \backslash E$ that satisfies $d(z,E) \leq 1$. The Extension Theorem for $(E,z)$ produces $(T,M,\Omega,\Xi)$ that satisfy \textbf{(E1-6)}. Thus, $T$ is a bounded extension operator:
\begin{subequations}
\begin{align}
&T(f,P) = f \; \mbox{on} \; E \; \mbox{and} \; J_z (T(f,P)) = P; \; \mbox{and} \label{f0} \\
&\|T(f,P)\|_{\LR} \simeq \|(f,P)\|_{L^{m,p}(E;z)} \;\; \mbox{for any data} \; (f,P). \label{f1}
\end{align}
\end{subequations}
We also have a formula
\begin{equation} \label{f2}
M(f,P) := \left( \sum_{\xi \in \Xi} |\xi(f,P)|^p \right)^{1/p} \; \mbox{with} \; \#(\Xi) \leq C \cdot \#(E),
\end{equation}
such that
\begin{equation}\|(f,P)\|_{L^{m,p}(E;z)} \simeq M(f,P).\label{f3}\end{equation}
Moreover, $T$ and the functionals in $\Xi$ have $\Omega$-assisted bounded depth, while the assists $\Omega$ satisfy
\begin{equation} \label{f4}
\sum_{\omega \in \Omega} \sp(\omega) \leq C \cdot \#(E).
\end{equation}

\subsection{Theorems 1, 2 and 3 for finite sets}

By definition of the trace seminorm, we have
\begin{equation}\label{ruff}
\|f\|_{\LE} = \inf \bigl\{\|(f,P)\|_{L^{m,p}(E;z)} : P \in \cP \bigr\}.
\end{equation}
We use Lemma \ref{lpmin} to choose $R \in \cP$ depending linearly on $f$, with
$$M(f,R) \leq C \cdot \inf \bigl\{M(f,P) : P \in \cP \bigr\}.$$
Define $\widehat{M}(f) := M(f,R)$ and $\widehat{T}(f) := T(f,R)$. Since $T$ is an extension operator we have $\widehat{T}(f) = T(f,R) = f$ on $E$. Therefore, $\widehat{T}$ is an extension operator. From \eqref{f3} and \eqref{ruff}, we have
\begin{equation} \label{3done} \widehat{M}(f) \simeq \|f\|_{\LE}.\end{equation}
Moreover,
$$\|\widehat{T}(f)\|_{\LR} = \|T(f,R)\|_{\LR} \osimeq{\eqref{f1},\eqref{f3}} M(f,R) = \widehat{M}(f) \simeq \|f\|_{\LE},$$
which completes the proof of Theorem \ref{thm1}.

We now define the collections of linear functionals
\begin{align*}
\widehat{\Omega} &:= \Omega \cup \bigl\{f \mapsto \partial^\alpha [R(f)](z) : \alpha \in \cM \bigr\}; \; \mbox{and} \\
\widehat{\Xi} &:= \bigl\{\mbox{linear functionals} \; f \mapsto \xi(f,R(f)) \; \mbox{with} \; \xi \in \Xi, \; \mbox{counted with repetition} \bigr\}.
\end{align*}
(For the definition of ``counted with repetition,'' see the proof of \textbf{(E6c)} in the preceding section.)

From \eqref{f2}, we have $\# (\widehat{\Xi}) \leq \# (\Xi) \leq C N$, and also
$$\widehat{M}(f) = M(f,R) = \left(\sum_{\xi \in \widehat{\Xi}} |\xi(f)|^p\right)^{1/p}.$$
Together with \eqref{3done}, this yields the conclusion of Theorem \ref{thm2}.

By standard arguments (e.g., see the previous section), the extension operator $\widehat{T}$ and the linear functionals belonging to $\widehat{\Xi}$ have $\widehat{\Omega}$-assisted bounded depth. Since $\widehat{\Omega}$ consists of the assists $\Omega$ along with $\#(\cM) = D$ (possibly) new linear functionals, it follows that 
$$\sum_{\omega \in \widehat{\Omega}} \sp(\omega) \leq \sum_{\omega \in \Omega} \sp(\omega) + D \cdot \#(E) \leq C \cdot \#(E),$$
which proves Theorem \ref{thm3}.\hfill \qed

\subsection{Inhomogeneous Sobolev Spaces}

Denote $W^{m,p}(E)$ for the space of real-valued functions on some finite subset $E \subset \R^n$, equipped with the norm
$$\|f\|_{W^{m,p}(E)} := \inf \{ \|F\|_{\WR} : F = f \; \mbox{on} \; E\}.$$ 

Here, we consider a variant of the problem solved in the previous section.

\noindent \underline{$\WR$ Extension Problem:} Let $f : E \rightarrow \R$ be defined on a finite subset $E \subset \R^n$. Find $G \in \WR$ that depends linearly on $f$, with $G|_E = f$ and $\|G\|_{\WR} \leq C \|f\|_{\WE}$.

We solve this problem, thereby proving the analogue of Theorem 1 for $W^{m,p}$ and finite $E$. The obvious analogues of Theorems 2,3 for $W^{m,p}$ also hold; we leave their consideration to the reader.

Let $Q^1,Q^2,\ldots$ be a tiling of $\R^n$ by unit cubes. If $F^i \in W^{m,p}(1.3Q^i)$ is a near optimal extension of $f|_{E \cap 1.1Q^i}$ for each $i=1,2,\ldots,$ it is easy to verify  that
\begin{align*}
&G=\sum_{i=1}^\infty \theta^i F^i \in \WR \; \mbox{is a near optimal extension of} \; f, \; \mbox{where}\\
&\theta^1,\theta^2,\ldots \; \mbox{form a smooth partition of unity on} \; \R^n,\;\mbox{with} \; \theta^i \; \mbox{supported on} \; 1.1Q^i.
\end{align*}
Thus, we may assume that $E \subset (0.9)Q$ for some cube with sidelength $\delta_Q \simeq 1$, and solve for an $F$ that depends linearly on $f$, such that $F = f$ on $E$ and $\|F\|_{W^{m,p}(Q)}$ is nearly minimal. If we could do this, we could take the $F^i$ to depend linearly on $f$, and ultimately we could take $G$ to depend linearly on $f$.

Fix some $z \in 0.9Q$ with $0 <d(z,E) \leq 1$. We use the linear operator $T : L^{m,p}(E;z) \rightarrow \LR$ that satisfies \eqref{f0} and \eqref{f1}. Let $P \in \cP$ be arbitrary. By the Sobolev inequality and \eqref{f0}, we have 
\begin{align}
\|T(f,P)\|_{W^{m,p}(Q)}^p &\lesssim \|T(f,P)\|_{L^{m,p}(\R^n)}^p + \sum_{|\alpha| \leq m-1} |\partial^\alpha P(z)|^p \notag{}\\
&\simeq \|(f,P)\|_{L^{m,p}(E;z)}^p + \sum_{|\alpha| \leq m-1} |\partial^\alpha P(z)|^p \quad (\mbox{thanks to} \; \eqref{f1}). \label{drum}
\end{align}
On the other hand, let $H \in W^{m,p}(Q)$ be arbitrary with $H = f$ on $E$ and $J_z H = P$. Choose $\theta \in C^\infty_0(Q)$ with $\theta \equiv 1$ on a neighborhood of $E \cup \{z\}$, and $|\partial^\alpha \theta| \lesssim 1$ when $|\alpha| \leq m$. Therefore, $\theta H \in \WR$ satisfies 
$$\theta H = f \; \mbox{on}  \; E, \;\; J_z(\theta H) = P, \;\; \mbox{and} \; \|\theta H\|_{\WR} \lesssim \|H\|_{W^{m,p}(Q)}.$$ Thus, by definition of the trace seminorm and the Sobolev inequality, we have
\begin{equation} \|(f,P)\|_{L^{m,p}(E;z)}^p + \sum_{|\alpha| \leq m-1} |\partial^\alpha P(z)|^p \lesssim \|\theta H\|_{\WR}^p \lesssim \|H\|_{W^{m,p}(Q)}^p.\label{drum1}\end{equation}
Thus, $T(f,P) \in W^{m,p}(Q)$ is a near optimal extension of $(f,P)$. Moreover, from \eqref{drum},\eqref{drum1} and \eqref{f2},\eqref{f3}, we have
$$\|T(f,P)\|_{W^{m,p}(Q)}^p \simeq  \|(f,P)\|_{L^{m,p}(E;z)}^p + \sum_{|\alpha| \leq m-1} |\partial^\alpha P(z)|^p \simeq \sum_{\xi \in \Xi} |\xi(f,P)|^p + \sum_{|\alpha| \leq m-1} |\partial^\alpha P(z)|^p.$$
Choose a polynomial $R \in \cP$ that depends linearly on $f$, for which $P=R$ minimizes the right-hand side above to within a universal constant factor. (This is possible thanks to Lemma \ref{lpmin}.) Thus, $T(f,R) \in W^{m,p}(Q)$ is a near optimal extension of $f$. This also yields a solution to the $W^{m,p}(\R^n)$ extension problem, as mentioned before.

\section{Passage to Infinite $E$}\label{sec_inf}
\setcounter{equation}{0}

In this section, we deduce Theorem 1 from the known case of finite $E$. Our plan is as follows.

Let $E \subset \R^n$ be infinite. Pick a countable subset

\begin{equation} 
\label{b1} E^\circ = \{x_0,x_1,x_2,\ldots\} \subset E, \; \mbox{whose closure contains} \; E.
\end{equation}
For each $N \geq 0$, define

\begin{equation}
\label{b2} E_N = \{x_0,x_1,\ldots,x_N\}.
\end{equation}

The known case of Theorem 1 produces a linear map $T_N : L^{m,p}(E_N) \rightarrow \LR$, such that
\begin{equation}
\label{b3} T_N f = f \; \mbox{on} \; E_N
\end{equation}
and
\begin{equation}
\label{b4}   \|T_N f\|_{\LR} \leq C \|f\|_{L^{m,p}(E_N)}
\end{equation}
for all $f \in L^{m,p}(E_N)$.

We hope to pass from $T_N$ to $T$ by taking a Banach limit as $N \rightarrow \infty$. Recall that a Banach limit is a linear map that carries an arbitrary bounded sequence $(t_N)_{N \geq 0}$ of real numbers to a real number denoted $\displaystyle \Blim_{N \rightarrow \infty} t_N$; the defining properties of a Banach limit are
\begin{align}
\label{b5} & \Blim_{N \rightarrow \infty} t_N = \lim_{N \rightarrow \infty} t_N \quad \mbox{whenever} \; \lim_{N \rightarrow \infty} t_N \; \mbox{exists}, \; \mbox{and}  \\
\label{b6} & |\Blim_{N \rightarrow \infty} t_N| \leq \limsup_{N \rightarrow \infty} |t_N|     .
\end{align}
The existence of Banach limits is immediate from the Hahn-Banach theorem. See \cite{DS}.

Thus, for $f \in \LE$, we hope to define
\begin{equation}
\label{b7} Tf(x) := \Blim_{N \rightarrow \infty} \left[ T_N(f|_{E_N})(x)\right] \quad \mbox{for} \; x \in \R^n
\end{equation}
and then prove that $T$ is an extension operator as in Theorem 1.

Unfortunately, without further ideas, the above plan is doomed. For instance, suppose $E \subset \{P_0 = 0\}$ for some polynomial $P_0 \in \cP$. If $T_N$ is an extension operator for $L^{m,p}(E_N)$, then so is
$$T_N^\sharp f(x) := T_Nf(x) + \mu_N f(x_0)P_0(x) \quad (x \in \R^n),$$
where $x_0$ is as in \eqref{b1}, and $\mu_N$ is any real number. In fact, \eqref{b3} and \eqref{b4} hold for the $T_N^\sharp$, with the same constant as for $T_N$. 

Since the sequence $(\mu_N)_{N \geq 0}$ is arbitrary, there is no way to guarantee that the sequence $(T_N(f|_{E_N})(x))_{N \geq 0}$ will be bounded for fixed $x$. Consequently, we cannot guarantee that the Banach limit \eqref{b7} exists. This problem arises because $\LR$ carries a seminorm, not a norm.

To overcome this difficulty, we normalize the extension operators $T_N$ as follows. Among all finite subsets $S \subset E$, we pick $S_0$ to minimize the dimension of the vector space $\cP(S) = \{ P \in \cP : P = 0  \; \mbox{on}  \;S\}$. Such an $S_0$ exists, since any non-empty set of non-negative integers has a minimum.

For any $y \in E$, the subspace $\cP(S_0 \cup \{y\}) \subset \cP(S_0)$ has dimension no less than that of $\cP(S_0)$. Therefore, $\cP(S_0 \cup \{y\}) = \cP(S_0)$. That is, any $P \in \cP$ that vanishes on $S_0$ must also vanish at $y$. Thus,
\begin{equation}
\label{b8} S_0 \subset E \; \mbox{is finite}
\end{equation}
and
\begin{equation}
\label{b9} \mbox{Any polynomial} \; P \in \cP \; \mbox{that vanishes on} \; S_0 \; \mbox{must also vanish on} \; E.
\end{equation}

Let $S_0 = \{y_0,y_1,\ldots,y_L\}$. Without loss of generality, we may pick our $x_0,x_1,\ldots$ in \eqref{b1} so that $x_i = y_i$ for $i=0,\ldots,L$. Therefore,
\begin{equation}
\label{b10} S_0 \subset E_N \quad \mbox{for all} \; N \geq L.
\end{equation}

We pick any projection $\pi_0 : \cP \rightarrow \cP(S_0)$. We will establish the following results.

\begin{lem}
\label{l1} For $N \geq L$, the extension operators $T_N$ in \eqref{b3},\eqref{b4} can be picked to satisfy the additional condition
\begin{equation}
\label{b11} \pi_0 J_{x_0}[T_N f] = 0 \quad \mbox{for all} \; f \in L^{m,p}(E_N), \; \mbox{with} \; x_0 \; \mbox{as in} \; \eqref{b1}. 
\end{equation}
\end{lem}

\begin{lem}
\label{l2} Suppose that $T_N$ satisfy \eqref{b11} for $N \geq L$. Then, for any $f \in \LE$ and for every cube $Q$, the functions $T_N(f|_{E_N})$ are bounded in $W^{m,p}(Q)$.
\end{lem}

If the $T_N$ satisfy \eqref{b11} then $\left[ T_N(f|_{E_N})(x) \right]_{N \geq 0}$ is thus a bounded sequence for each fixed $x \in \R^n$ and $f \in \LE$. Therefore, the Banach limit in \eqref{b7} is well-defined. Clearly, $T$ is a linear map, taking functions $f \in \WE$ to functions $Tf$ defined on $\R^n$. Moreover, for each $k \geq 0$, we have
$$Tf(x_k) = \Blim_{N \rightarrow \infty} T_N(f|_{E_N})(x_k) = f(x_k) \quad \mbox{with} \; x_k \; \mbox{as in} \; \eqref{b1},$$
simply because $x_k \in E_N$ for $N \geq k$, hence $T_N(f|_{E_N})(x_k) = f(x_k)$ for $N \geq k$; see \eqref{b3} and \eqref{b5}. Thus,
\begin{equation}
\label{b12} Tf=f \; \mbox{on} \; E^\circ \;\; \mbox{for any} \; \; f \in \LR, \; \; \mbox{with} \; E^\circ \subset E \; \mbox{as in} \; \eqref{b1}.
\end{equation}
We do not yet know that $Tf \in \LR$. Therefore, we prove the following result.
\begin{lem}\label{l3}
Let $Q \subset \R^n$ be a cube, and let $F_0,F_1,\ldots$ be a bounded sequence in $W^{m,p}(Q)$. Then the function $F$, defined by
\begin{equation}
\label{b13} F(x) = \Blim_{N \rightarrow \infty} F_N(x) \quad(x \in Q)
\end{equation}
belongs to $W^{m,p}(Q)$, and we have
$$\|F\|_{L^{m,p}(Q)} \leq \limsup_{N \rightarrow \infty} \|F_N\|_{L^{m,p}(Q)}.$$
\end{lem}
We apply the above lemma to $F_N = T_N(f|_{E_N})$ for $f \in \LE$. Then $F = Tf$ is given by \eqref{b13}. Since $\|(f|_{E_N})\|_{L^{m,p}(E_N)} \leq \|f\|_{L^{m,p}(E)}$, estimate \eqref{b4} and Lemmas \ref{l2}, \ref{l3} together imply that $Tf$ belongs to $L^{m,p}(Q)$ for any cube $Q$, and moreover
$$\|Tf\|_{L^{m,p}(Q)} \leq C \|f\|_{L^{m,p}(E)}.$$
Here we assume the $T_N$ satisfy \eqref{b11}. Since $Q \subset \R^n$ is an arbitrary cube, it follows that $Tf \in L^{m,p}(\R^n)$, and that
\begin{equation}
\label{b14} \|Tf\|_{\LR} \leq C \|f\|_{\LE} \;\;\; \mbox{for all} \; f \in \LE.
\end{equation}

Also, since the subset $E^\circ$ in \eqref{b1} is dense in $E$, we conclude from \eqref{b12} that
\begin{equation}
\label{b15} Tf = f \; \mbox{on} \; E, \; \mbox{for all} \; f \in \LE.
\end{equation}
Thus, our extension operator $T$ maps $\LE$ to $\LR$ and satisfies \eqref{b14}, \eqref{b15}. We therefore obtain Theorem \ref{thm1}, once we have Lemmas \ref{l1}, \ref{l2}, \ref{l3}.

\subsection{Proof of Lemma \ref{l1}}

To prove Lemma \ref{l1}, let $T_N : L^{m,p}(E_N) \rightarrow \LR$ satisfy \eqref{b3} and \eqref{b4}. For $f \in L^{m,p}(E_N)$, define
$$\widetilde{T}_N f = T_N f - \pi_0 J_{x_0}(T_N f) \in \LR \;\; \mbox{since} \; \pi_0 J_{x_0}(T_N f) \in \cP.$$
Note that $\pi_0 J_{x_0}(\widetilde{T}_N f) = \pi_0 \left[ J_{x_0}(T_N f) - \pi_0 J_{x_0}(T_N f) \right] = 0$, where we have used the fact that $J_{x_0} (T_N f) \in \cP$ (and the fact that $\pi_0^2 = \pi_0$). Also, since $\pi_0 J_{x_0}(T_N f) \in \cP$, we have $\|\pi_0 J_{x_0}( T_N f)\|_{\LR} = 0$, hence
$$\|\widetilde{T}_N f\|_{\LR} = \|T_N f\|_{\LR} \leq C \|f\|_{L^{m,p}(E_N)} \;\; \mbox{by} \; \eqref{b4}.$$

Finally, since $\pi_0 J_{x_0}(T_N f) \in \cP(S_0)$, we have $\pi_0 J_{x_0} (T_N f) = 0$ on $S_0$, hence also on $E$, by \eqref{b9}. Therefore,
$$\widetilde{T}_N f = f \; \mbox{on} \; E, \; \mbox{by} \; \eqref{b3}.$$
Thus, $\widetilde{T}_N : L^{m,p}(E_N) \rightarrow \LR$ is a linear map that satisfies \eqref{b3}, \eqref{b4} and \eqref{b11}. The proof of Lemma \ref{l1} is complete.

\subsection{Proof of Lemma \ref{l2}}

To establish Lemma \ref{l2}, we use the following result.

\begin{lem} \label{l4}
For any cube $Q$ containing $x_0$ there exist constants $A_1(Q), A_2(Q)>0$ such that for all $F \in W^{m,p}(Q)$ with $\pi_0 J_{x_0}(F) = 0$, we have the estimate
\begin{equation}
\label{b16} \sum_{|\alpha| \leq m-1} |\partial^\alpha F(x_0)| \leq A_1(Q) \cdot \left[ \|F\|_{L^{m,p}(Q)} + \max_{y \in S_0} |F(y)| \right],
\end{equation}
and therefore
\begin{equation}
\label{b17} \|F\|_{W^{m,p}(Q)} \leq A_2(Q) \cdot \left[ \|F\|_{L^{m,p}(Q)} + \max_{y \in S_0} | F(y)| \right].
\end{equation}
\end{lem}
\begin{proof}
Fix a cube $Q$ containing $x_0$, and suppose \eqref{b16} fails. Then there exists a sequence of functions $F_N \in W^{m,p}(Q)$ $(N \geq 0)$ such that
\begin{equation}
\label{b18} \sum_{|\alpha| \leq m-1} |\partial^\alpha F_N(x_0)|  = 1
\end{equation}
and
\begin{equation}
\label{b19} \pi_0 J_{x_0}(F_N) = 0
\end{equation}
but
\begin{equation}
\label{b20} \|F_N\|_{L^{m,p}(Q)} \rightarrow 0 \; \mbox{as} \; N \rightarrow \infty \; \mbox{and} 
\end{equation}
\begin{equation}
\label{b21} \max_{y \in S_0} |F_N(y)| \rightarrow 0 \; \mbox{as} \; N \rightarrow \infty.
\end{equation}
By \eqref{b18}, \eqref{b20}, the functions $F_N$ are bounded in $W^{m,p}(Q)$, hence in $C^{m-1,s}(Q)$ with $s = 1-n/p \in (0,1)$. By Ascoli's theorem, a subsequence of the $(F_N)$ converges in $C^{m-1}(Q)$ to a function $F \in C^{m-1,s}(Q)$. From \eqref{b18}-\eqref{b21}, we obtain the following properties of $F$:
\begin{align}
\label{b22} &\qquad \sum_{|\alpha| \leq m-1} |\partial^\alpha F(x_0)|  = 1;  \\
\label{b23} &\;\;\qquad \pi_0 J_{x_0}(F) = 0; \\
\label{b24} &\;\;\qquad \|F\|_{\dot{C}^{m-1,s}(Q)} \leq \limsup_{N \rightarrow \infty} \|F_N\|_{\dot{C}^{m-1,s}(Q)} \leq C \limsup_{N \rightarrow \infty} \|F_N\|_{L^{m,p}(Q)} = 0; \; \mbox{and} \\
\label{b25} &\;\;\qquad F = 0 \; \mbox{on} \; S_0.
\end{align}

From \eqref{b24} we see that $F$ is a polynomial, $F \in \cP$. By \eqref{b25}, we have $F \in \cP(S_0)$, hence $\pi_0 J_{x_0}(F) = \pi_0 F = F$. Therefore, $F=0$ by \eqref{b23}. This contradicts \eqref{b22}.

The above contradiction shows that \eqref{b16} cannot fail. Conclusion \eqref{b17} follows at once from \eqref{b16} and the Sobolev inequality. The proof of Lemma \ref{l4} is complete.
\end{proof}

To prove Lemma \ref{l2}, we fix a cube $Q \subset \R^n$. Without loss of generality, we may suppose that $Q$ contains $x_0$.

For fixed $f \in \LE$, and for any $N \geq L$, we apply Lemma \ref{l4} to $F_N = T_N(f|_{E_N})$. Note that $F_N \in W^{m,p}(Q)$, and $\pi_0 J_{x_0}(F_N) = 0$ by \eqref{b11}. Hence, Lemma \ref{l4} applies, and we learn that
\begin{equation}
\label{b26} \|F_N\|_{W^{m,p}(Q)} \leq A_2(Q) \cdot \left[ \|F_N\|_{L^{m,p}(Q)} + \max_{y \in S_0} |F_N(y)| \right].
\end{equation}
On the other hand \eqref{b4} yields the estimates
\begin{equation}
\label{b27} \|F_N\|_{L^{m,p}(Q)} \leq \|F_N\|_{\LR} \leq C \|(f|_{E_N})\|_{L^{m,p}(E_N)} \leq C \|f\|_{L^{m,p}(E)}.
\end{equation}
Also, since $S_0 \subset E_N$, \eqref{b3} yields $F_N = f|_{E_N}$ on $S_0$, i.e.,
\begin{equation}
\label{b28} F_N = f \; \mbox{on} \; S_0.
\end{equation}

Putting \eqref{b27} and \eqref{b28} into \eqref{b26}, we find that
$$\|F_N\|_{W^{m,p}(Q)} \leq C A_2(Q) \left[ \|f\|_{\LE} + \max_{y \in S_0} |f(y)|\right] \;\; \mbox{for} \; N \geq L.$$
Therefore, the functions $F_0,F_1,F_2,\ldots$ form a bounded subset of $W^{m,p}(Q)$, completing the proof of Lemma \ref{l2}.

\subsection{Proof of Lemma \ref{l3}}

The proof of Lemma \ref{l3} uses the following simple observation. As before, we take $s = 1-n/p \in (0,1)$.

\begin{lem} \label{l5}
Let $A >0$ be a constant, and let $Q \subset \R^n$ be a cube.

For each multi-index $\alpha$ of order $|\alpha| \leq m-1$, let $f^{(\alpha)}$ be a function on $Q$. Assume the following estimates.
\begin{align}
\label{b29} &|f^{(\alpha)}(x+h) - f^{(\alpha)}(x) - \sum_{j=1}^n f^{(\alpha + \mathbbm{1}_j)}(x) h_j | \leq A |h|^{1+s} \\
&\qquad\qquad\qquad \mbox{for} \; |\alpha| \leq m-2, \; x \in Q, \; h = (h_1,\ldots,h_n), \; x+h \in Q. \notag{} \\
& |f^{(\alpha)}(x) - f^{(\alpha)}(y)| \leq A |x-y|^s \;\; \mbox{for} \; |\alpha| \leq m-1, \; x,y, \in Q. \label{b30}
\end{align}
(Here, $\mathbbm{1}_j$ denotes the $j^\th$ unit multi-index, so that $\partial^{\mathbbm{1}_j} = \frac{\partial}{\partial x_j}$.)

Then $f^{(0)} \in C^{m-1,s}(Q)$, and $f^{(\alpha)} = \partial^\alpha f^{(0)}$ on $Q \;$ for each $\alpha$ of order $|\alpha| \leq m-1$.
\end{lem}
\begin{proof}
Hypotheses \eqref{b29},\eqref{b30} show that, for $|\alpha| \leq m-2$, the function $f^{(\alpha)}$ belongs to $C^1(Q)$, and $\partial_{x_j} f^{(\alpha)} = f^{(\alpha + \mathbbm{1}_j)}$. Thus, $f^{(0)} \in C^{m-1}(Q)$ and $\partial^\alpha f^{(0)} = f^{(\alpha)}$ for $|\alpha| \leq m-1$. Hypothesis \eqref{b30} now shows that the derivatives of $f^{(0)}$ up to order $m-1$ are Lipschitz-$s$. Hence, $f^{(0)} \in C^{m-1,s}(Q)$, completing the proof of Lemma \ref{l5}.
\end{proof}

\textit{Proof of Lemma \ref{l3}:}
Let $Q, F_0, F_1, \ldots, F$ be as in the hypotheses of Lemma \ref{l3}. For $|\alpha| \leq m-1$, $N \geq 0$, let $$F_N^{(\alpha)} = \partial^\alpha F_N \in C^{m-1-|\alpha|,s}(Q).$$

Since the $F_N$ are bounded in $W^{m,p}(Q)$, the following estimates hold for some constant $A$.
\begin{align}
\label{b31}& |F_N^{(\alpha)}(x)| \leq  A \quad \mbox{for} \; x \in Q.\\
\label{b32}& |F_N^{(\alpha)}(x+h) - F_N^{(\alpha)}(x) - \sum_{j=1}^n F_N^{(\alpha + \mathbbm{1}_j)}(x) h_j| \leq A |h|^{1 + s }\\
& \notag{} \qquad\qquad\qquad\mbox{for} \;\; |\alpha| \leq m-2, \; x \in Q, \; h = (h_1,\ldots,h_n), \; x + h \in Q.\\
\label{b33}& |F_N^{(\alpha)}(x) - F_N^{(\alpha)}(y)| \leq A |x-y|^s \;\;\; \mbox{for} \; |\alpha| \leq m-1, \; x,y \in Q.
\end{align}

Thanks to \eqref{b31}, the Banach limit
\begin{equation}
\label{b34} F^{(\alpha)}(x) = \Blim_{N \rightarrow \infty} F_N^{(\alpha)}(x) \quad (x \in Q)
\end{equation}
is well-defined. Note that
\begin{equation}
\label{b35}F^{(0)}(x) = \Blim_{N \rightarrow \infty} F_N(x) = F(x), \;\; \mbox{with} \; F \; \mbox{as in the statement of Lemma \ref{l3}}.
\end{equation}
Applying \eqref{b6}, we may pass to the Banach limit and deduce from
\eqref{b31},\eqref{b32},\eqref{b33} for $F_N^{(\alpha)}$ the corresponding estimates for the $F^{(\alpha)}$.
Lemma \ref{l5} now shows that $F = F^{(0)} \in C^{m-1,s}(Q)$, and that $\partial^\alpha F = F^{(\alpha)}$ for each $|\alpha| \leq m-1$.

Hence, to prove that $F \in W^{m,p}(Q)$, and that $\|F\|_{L^{m,p}(Q)} \leq \limsup_{N \rightarrow \infty} \|F_N\|_{L^{m,p}(Q)}$, it is enough to prove that
\begin{equation}
\label{b36} \left| \int_Q \partial_{x_j} \varphi \cdot F^{(\alpha)} dx \right| \leq \limsup_{N \rightarrow \infty} \|F_N\|_{L^{m,p}(Q)} \cdot \|\varphi\|_{L^{p'}(Q)}
\end{equation}
for any test function $\varphi \in C^\infty_0(Q)$ and any $|\alpha| = m -1$, $j \in \{1,\ldots,n\}$. Here, $p'$ is the dual exponent to $p$, and we use the definition of the $L^{m,p}$-seminorm :
$$\|F\|_{L^{m,p}(Q)}^p = \max_{|\alpha| = m} \int_Q |\partial^\alpha F(x)|^p dx. $$

Thus, Lemma \ref{l3} reduces to the task of proving \eqref{b36} for any given $\varphi \in C_0^\infty(Q)$, $|\alpha| = m - 1$, $j \in \{1,\ldots,n\}$. Fix such $\varphi, \alpha, j,$ and let $M$ be any real number greater than $\limsup_{N \rightarrow \infty} \|F_N\|_{L^{m,p}(Q)}$.

For $N$ large enough we have $\|F_N\|_{L^{m,p}(Q)} \leq M$, hence
\begin{equation}
\label{b37} \left| \int_Q \partial_{x_j} \varphi \cdot F_N^{(\alpha)} dx \right| \leq M \|\varphi\|_{L^{p'}(Q)}
\end{equation}
since $F_N^{(\alpha)} = \partial^\alpha F_N$.

We will derive \eqref{b36} by passing to the Banach limit in \eqref{b37}. To do so, we simply approximate the integrals in \eqref{b36},\eqref{b37} by Riemann sums.

We know that
\begin{equation}
\label{b38} |F_N^{(\alpha)}(x) - F_N^{(\alpha)}(y)| \leq A |x-y|^s \;\; \mbox{and} \;\; |F^{(\alpha)}(x) - F^{(\alpha)}(y)| \leq A |x-y|^s
\end{equation}
for $x,y \in Q$, with $A$ independent of $N$.

Let $\delta>0$ be a small number (later, we will take $\delta \rightarrow 0^+$); let $\{Q_\nu\}$ be a partition of $Q$ into subcubes with $\mbox{center}(Q_\nu) = z_\nu$, and with sidelength $\delta_{Q_\nu} < \delta$. Then \eqref{b31} and \eqref{b38} together imply the estimates
\begin{equation}
\label{b39} \left|\int_Q \partial_{x_j} \varphi \cdot F^{(\alpha)}_N dx - \sum_{\nu=1}^{\nu_{\max}} \partial_{x_j} \varphi(z_\nu) \cdot F^{(\alpha)}_N(z_\nu) \cdot \delta_{Q_\nu}^n\right| \leq CA \delta^s \cdot \delta_{Q}^n
\end{equation}
and
\begin{equation}
\label{b40}\left|\int_Q \partial_{x_j} \varphi \cdot F^{(\alpha)} dx - \sum_{\nu=1}^{\nu_{\max}} \partial_{x_j} \varphi(z_\nu) \cdot F^{(\alpha)}(z_\nu) \cdot \delta_{Q_\nu}^n\right| \leq CA \delta^s \cdot \delta_{Q}^n,
\end{equation}
with $C$ independent of $N$.

In particular, \eqref{b39} shows that the sequence $\Bigl( \int_Q \partial_{x_j} \varphi \cdot F_N^{(\alpha)} dx \Bigr)_{N \geq 0}$ is bounded.

Since $F^{(\alpha)}(z_\nu) = \Blim_{N \rightarrow \infty} F^{(\alpha)}_N(z_\nu)$ for each $\nu$, property \eqref{b6} of the Banach limit, together with \eqref{b39}, implies that
\begin{equation}
\label{b41} \left| \left[ \Blim_{N \rightarrow \infty} \int_Q \partial_{x_j} \varphi \cdot F^{(\alpha)}_N dx \right] - \sum_{\nu=1}^{\nu_{\max}} \partial_{x_j} \varphi(z_\nu) \cdot F^{(\alpha)}(z_\nu) \cdot \delta_{Q_\nu}^n \right| \leq CA \delta^s \cdot \delta_Q^n.
\end{equation}
(In \eqref{b41}, the Banach limit in square brackets in well-defined, thanks to \eqref{b37}.)

Comparing \eqref{b40} and \eqref{b41}, we see that
$$\left| \int_Q \partial_{x_j} \varphi \cdot F^{(\alpha)} dx - \left[ \Blim_{N \rightarrow \infty} \int_Q \partial_{x_j} \varphi \cdot F^{(\alpha)}_N dx \right] \right| \leq 2CA \delta^s \cdot \delta_{Q}^n.$$
Since $\delta>0$ is arbitrarily small, we conclude that
\begin{equation}
\label{b42}\int_Q \partial_{x_j} \varphi \cdot F^{(\alpha)} dx = \Blim_{N \rightarrow \infty} \int_Q \partial_{x_j} \varphi \cdot F^{(\alpha)}_N dx.
\end{equation}

From \eqref{b37}, \eqref{b42} and property \eqref{b6} of the Banach limit, we obtain the estimate

\begin{equation}
\label{b43} \left| \int_Q \partial_{x_j} \varphi \cdot F^{(\alpha)} dx \right| \leq M \|\varphi\|_{L^{p'}(Q)}.
\end{equation}
This implies the desired estimate \eqref{b36}, since $M$ in \eqref{b43} is an arbitrary real number greater than $\limsup_{N \rightarrow \infty} \|F_N\|_{L^{m,p}(Q)}$.

We reduced the proof of Lemma \ref{l3} to \eqref{b36}, and we have now proven \eqref{b36}. This completes the proof of Lemma \ref{l3}, and with it, the proof of Theorem \ref{thm1}.

\subsection{Epilogue}

Given $E \subset \R^n$ (possibly infinite), we can also ``construct'' a linear extension operator $T : W^{m,p}(E) \rightarrow W^{m,p}(\R^n)$ by applying the case of finite $E$ and passing to a Banach limit. This exercise is a much easier version of the argument we just explained for $L^{m,p}$. Since $\|\cdot\|_{W^{m,p}(\R^n)}$ is a norm (rather than a seminorm), the pitfalls that we worked to avoid will no longer arise. We leave the details to the reader.

By using a Banach limit, we have sacrificed all knowledge of the structure of our linear extension operator $T$. It would be very interesting to gain some understanding of that structure. Such an understanding is achieved for the $C^m$ case in \cite{F4}, and for the deeper case of $C^{m,s}$ and related spaces in \cite{L1}.

\bibliographystyle{plain}
\bibliography{SobolevExtension}

\begin{thebibliography}{1}


\bibitem{BSS} J.D. Batson, D.A. Spielman and N. Srivastava, 
\emph{Twice-Ramanujan sparsifiers}, In Proceedings of the 41st Annual ACM Symposium on Theory of Computing (2009), 255--262.

\bibitem{BMP} E. Bierstone, P. Milman, W. Paw\l ucki, \emph{Differentiable functions defined on closed sets. A problem of Whitney}, Inventiones Math. \textbf{151, No. 2} (2003), 329--352.

\bibitem{BS1} Y. Brudnyi and P. Shvartsman, \emph{Generalizations of Whitney's extension theorem}, Int. Math. Research Notices \textbf{3} (1994), 129--139.
\bibitem{BS2} Y. Brudnyi and P. Shvartsman, \emph{The Whitney problem of existence of a linear extension operator}, J. Geometric Analysis \textbf{7, No. 4} (1997) , 515--574.


\bibitem{CC}
P.B. Callahan and S.R. Kosaraju, \emph{A decomposition of multi-dimensional point sets with applications to k-nearest neighbors and n-body potential fields}, Journal of the Association for Computing Machinery \textbf{42, No. 1} (1995), 67--90.

\bibitem{DS}
N. Dunford and  J.T. Schwartz, Linear Operators, Part  I: General Theory, Interscience, New York, 1958.

\bibitem{Evans}
L.C. Evans, Partial Differential Equations, American Mathematical Society, Providence, 1998.

\bibitem{F1}
C. Fefferman, \emph{A sharp form of Whitney's extension theorem}, Annals of Math. \textbf{161, No. 1} (2005), 509--577.

\bibitem{F2}
C. Fefferman, \emph{Whitney's extension problem for $C^m$}, Annals of Math. \textbf{164, No. 1} (2006), 313--359.

\bibitem{F3}
C. Fefferman, \emph{{$C^m$} extension by linear operators}, Annals of Math. \textbf{166, No. 3} (2007), 779--835.

\bibitem{F4}
C. Fefferman, \emph{The structure of linear extension operators for {$C^m$}}, Rev. Mat. Iberoam. \textbf{23 No. 1} (2007), 269--280.

\bibitem{F5}
C. Fefferman, \emph{Fitting a $C^m$-smooth function to data III,} Annals of Math. \textbf{170, No.1} (2009), 427--441.

\bibitem{FIL}
C. Fefferman, A. Israel, G.K. Luli,
\emph{The structure of Sobolev extension operators} (preprint) (2012).

\bibitem{FK1}
C. Fefferman and B. Klartag, \emph{Fitting a $C^m$-smooth function to data I,} Annals of Math. \textbf{169, No. 1} (2009), 315--346.

\bibitem{FK2}
C. Fefferman and B. Klartag, \emph{Fitting a $C^m$-smooth function to data II}, Revista Mat. Iberoamericana \textbf{25, No. 1} (2009), 49--273.


\bibitem{G} G. Glaeser, 
\emph{Etudes de quelques algebres tayloriennes}, J. d'Analyse \textbf{6} (1958), 1--124.

\bibitem{I} 
A. Israel,
\newblock \emph{A bounded linear extension operator for ${L}^{2,p}(\mathbb{R}^2)$} (preprint) (2010).


\bibitem{L1}
G.K. Luli,
\newblock \emph{$C^{m,\omega}$ extension by bounded-depth linear operators,} Advances in Math. \textbf{224, No. 5} (2010), 1927--2021.

\bibitem{L2}
G.K. Luli,
\emph{Sobolev extension in one-dimension} (2008), notes available at 

\newblock http://www.math.princeton.edu/\textasciitilde gluli/TH/notes.pdf .

\bibitem{S1}
P. Shvartsman, \emph{Sobolev $W^1_p$-spaces on closed subsets of $R^n$}, Advances in Math. \textbf{220, No. 6} (2009), 1842--1922.

\bibitem{S2}
P. Shvartsman, \emph{Lipschitz spaces generated by the Sobolev-Poincar\'e inequality and extensions of Sobolev functions} (preprint) (2011).

\bibitem{S3}
P. Shvartsman,
\emph{On the sum of a Sobolev space and a weighted $L_p$-space} (preprint) (2011).

\bibitem{SS}
D.A. Spielman and N. Srivastava, \emph{Graph sparsification by effective resistances,} In Proceedings of the 40th Annual ACM Symposium on Theory of Computing (2008), 563--568.

\bibitem{ST} 
D.A. Spielman and S-H. Teng, \emph{Nearly-linear time algorithms for graph partitioning, graph sparsification, and solving linear systems}, In Proceedings of the 36th Annual ACM Symposium on Theory of Computing (2004), 81--90.


\bibitem{W1}
H. Whitney, \emph{Analytic extensions of differentiable functions defined in closed sets}, Transactions A.M.S. \textbf{36} (1934), 63--89.

\bibitem{W2}
H. Whitney, \emph{Differentiable functions defined in closed sets I}, Transactions A.M.S. \textbf{ 36} (1934), 369--387.

\bibitem{W3}
H. Whitney, \emph{Functions differentiable on the boundaries of regions}, Annals of Math \textbf{35} (1934), 482--485.

\end{thebibliography}

\end{document}